\documentclass[twoside]{article}
\usepackage{amsthm,amsmath,amssymb}
{\newtheorem{theorem}{Theorem}[section]
   \newtheorem{proposition}[theorem]{Proposition}
   \newtheorem{lemma}[theorem]{Lemma}
   \newtheorem{corollary}{Corollary}}
{\theoremstyle{definition}
   
   \newtheorem{example}{Example}}
{\theoremstyle{remark}
   
   \newtheorem*{remark}{Remark}}
 
\newcommand{\bfind}[1]{\index{#1}{\bf #1}}
\newcommand{\n}{\par\noindent}

\newcommand{\sn}{\par\smallskip\noindent}
\newcommand{\mn}{\par\medskip\noindent}
\newcommand{\bn}{\par\bigskip\noindent}
\newcommand{\pars}{\par\smallskip}
\newcommand{\parm}{\par\medskip}
\newcommand{\parb}{\par\bigskip}
\newcommand{\isom}{\simeq}

\newcommand{\ovl}[1]{\overline{#1}}
\newcommand{\ec}{\prec_{\exists}}
\newcommand{\sep}{^{\rm sep}}
\newcommand{\dec}{^d}
\newcommand{\Aut}{\mbox{\rm Aut}\,}
\newcommand{\bbox}[1]{\makebox(0,0){\rule[-2ex]{0ex}{5.4ex}#1}}

\newcommand{\rr}{\mbox{\rm rr}\,}
\newcommand{\chara}{\mbox{\rm char}\,}
\newcommand{\trdeg}{\mbox{\rm trdeg}\,}
\newcommand{\subsetuneq}{\mathrel{\raisebox{.8ex}{\footnotesize%
$\displaystyle\mathop{\subset}_{\not=}$}}}
\newcommand{\adresse}{\par\bigskip \small\rm
 Department of Mathematics and Statistics, 
 University of Saskatchewan, \par
 106 Wiggins Road, 
 Saskatoon, Saskatchewan, Canada S7N 5E6 \par
 email: fvk@math.usask.ca \ \ --- \ \ home page:
 http://math.usask.ca/$\,\tilde{ }\,$fvk/index.html}
\font\teneu=eufm10  
\font\seveneu=eufm7 
\font\fiveeu=eufm5  
\def\eu #1{{\mathchoice{{\hbox{\teneu #1}}}{{\hbox{\teneu #1}}}
{{\hbox{\seveneu #1}}}{{\hbox{\fiveeu #1}}}}}
\newcommand{\lv}{\mathbb}
\newcommand{\A}{\lv A}
\newcommand{\N}{\lv N}
\newcommand{\Q}{\lv Q}
\newcommand{\R}{\lv R}
\newcommand{\Z}{\lv Z}
\newcommand{\F}{\lv F}
\newcommand{\Fp}{\F_p}
\newcommand{\Qp}{\Q_p}
\begin{document}
\title{Valuation theoretic and model theoretic
aspects of local uniformization} 
\author{Franz-Viktor Kuhlmann}
\date{Fall 1999}
\maketitle
{\footnotesize\rm
\setlength{\parskip}{-3pt}
\tableofcontents}
%
%
%
%
\section{Introduction}
In this paper, I will take you on an excursion from Algebraic Geometry
through Valuation Theory to Model Theoretic Algebra, and back. If our
destination sounds too exotic for you, you may jump off at the Old World
(Valuation Theory) and divert yourself with problems and examples until
you catch our plane back home.

As a preparation for the foreign countries of the Old World, I recommend
to read the first sections of the paper [V]. You may also look at the
basic facts about valuations in the books [ZS], [R1], [EN], [WA2], [JA].
All other equipment will be distributed on the excursion. I do not
recommend that you try to read about the exotic countries of model
theory. An introduction will be given on the excursion. If you want to
read more, you may ask me for a copy of the chapter ``Introduction to
model theoretic algebra'' of the forthcoming book [K2], and later you
may look at books like [CK], [CH1], [HO], [SA].

Since I want our excursion (and myself) to be as relaxed as possible,
all varieties we meet will be assumed to be affine and irreducible
(wherever it matters). But certainly, we do not assume them to be
non-singular, nor do we assume that the base fields be algebraically
closed or the characteristic be 0.

%
%
\section{What does local uniformization mean?}
Here is the bitter truth of mankind: In most cases we average human
beings are too stupid to solve our problems globally. So we try to
solve them locally. And if we are clever enough (and truly interested),
we then may think of patching the local solutions together to obtain a
global solution.

What is the problem we are considering here? It is the fact that an
algebraic variety has singularities, and we want to get rid of them.
That is, we are looking for a second variety having the same function
field, and having no singularities. This would be the global solution of
our problem. As we are too stupid for it, we are first looking for a
local solution. Naively speaking, ``local'' means something like ``at a
point of the variety''. So local solution would mean that we get rid of
one singular point. We are looking for a new variety where our point
becomes non-singular. But wait, this was nonsense. Because what
is our old, singular point on the new variety? We cannot talk of the
same points of two different varieties, unless we deal with
subvarieties. But passing from varieties to subvarieties or vice versa
will in general not provide the solution we are looking for. So do we
have to forget about local solutions of our problem?

The answer is: no. Let us have a closer look at our notion of ``point''.
Assume our variety $V$ is given by polynomials $f_1,\ldots,f_n\in K[X_1,
\ldots,X_\ell]$. Naively, by a point of $V$ we then mean an $\ell$-tupel
$(a_1,\ldots,a_\ell)$ of elements in an arbitrary extension field $L$ of
$K$ such that $f_i(a_1,\ldots,a_\ell)=0$ for $1\leq i\leq n$. This means
that the kernel of the ``evaluation homomorphism'' $K[X_1,\ldots,X_\ell]
\rightarrow L$ defined by $X_i\mapsto a_i$ contains the ideal $(f_1,
\ldots,f_n)$. So it induces a homomorphism $\eta$ from the coordinate
ring $K[V]=K[X_1,\ldots,X_\ell]/(f_1,\ldots,f_n)$ into $L$ over $K$.
(The latter means that it leaves the elements of $K$ fixed.) However, if
$a'_1,\ldots,a'_\ell\in L'$ are such that $a_i\mapsto a'_i$ induces an
isomorphism from $K(a_1,\ldots,a_\ell)$ onto $K(a'_1,\ldots,a'_\ell)$,
then we would like to consider $(a_1,\ldots,a_\ell)$ and
$(a'_1,\ldots,a'_\ell)$ as the same point of $V$. That is, we are only
interested in $\eta$ up to composition $\sigma\circ\eta$ with
isomorphisms $\sigma$. This we can get by considering the kernel of
$\eta$ instead of $\eta$. This leads us to the modern approach:
to view a point as a prime ideal of the coordinate ring.

But I wouldn't have told you all this if I intended to follow this
modern approach. Instead, I want to build on the picture of
homomorphisms. So I ask you to accept temporarily the convention that a
point of $V$ is a homomorphism of $K[V]$ over $K$ (i.e., leaving $K$
elementwise invariant), modulo composition with isomorphisms. Recall
that $K[V]=K[x_1,\ldots,x_\ell]$, where $x_i$ is the image of $X_i$
under the canonical epimorphism $K[X_1,\ldots, X_\ell]\rightarrow
K[X_1,\ldots,X_\ell]/(f_1,\ldots,f_n)=K[V]$. The function field $K(V)$
of $V$ is the quotient field $K(x_1,\ldots, x_\ell)$ of $K[V]$. It is
generated by $x_1,\ldots,x_\ell$ over $K$, hence it is finitely
generated. Every finite extension of a field $K$ of transcendence degree
at least 1 is called an \bfind{algebraic function field} (over $K$), and
it is in fact the function field of a suitable variety defined over $K$.
When we talk of function fields in this paper, we will always mean
algebraic function fields.

Now recall what it means to look for another variety $V'$ having the
same function field $F:=K(V)$ as $V$ (i.e., being birationally
equivalent to $V$). It just means to look for another set of generators
$y_1,\ldots, y_k$ of $F$ over $K$. Now the points of $V'$ are the
homomorphisms of $K[y_1,\ldots,y_k]$ over $K$, modulo composition with
isomorphisms. But in general, $y_1,\ldots,y_k$ will not lie in
$K[x_1,\ldots,x_\ell]$, hence we do not see how a given homomorphism of
$K[x_1,\ldots,x_\ell]$ could determine a homomorphism of
$K[y_1,\ldots,y_k]$. But if we could extend the homomorphism of
$K[x_1,\ldots,x_\ell]$ to all of $K(x_1,\ldots,x_\ell)$, then this
extension would assign values to every element of $K[y_1,\ldots,y_k]$.
Let us give a very simple example.

\begin{example}                             \label{examp1}
Consider the coordinate ring $K[x]$ of $V=\A_K^1$. That is, $x$ is
transcendental over $K$, and the function field $K(V)$ is just the
rational function field $K(x)$ over $K$. A homomorphism of the
polynomial ring
$K[V]=K[x]$ is just given by ``evaluating'' every polynomial $g(x)$ at
$x=a$. I have seen many people who suffered in school from the fact that
one can also try to evaluate rational functions $g(x)/h(x)$.
The obstruction is that $a$ could be a zero of $h$, and what do we get
then by evaluating $1/h(x)$ at $a$? (In fact, if our homomorphism is not
an embedding, i.e., if $a$ is not transcendental over $K$, then there
will always be a polynomial $h$ over $K$ having $a$ as a root.) So we
have to accept that the evaluation will not only render elements in
$K(a)$, but also the element $\infty$, in which case we say that the
evaluated rational function has a pole at $a$. So we can extend our
homomorphism to a map $P$ on all of $K(x)$, taking into the
bargain that it may not always render finite values. But on the subring
${\cal O}_P=\{g(x)/h(x) \mid h(a)\ne 0\}$ of $K(x)$ on which $P$ is
finite, it is still a homomorphism.
\end{example}

What we have in front of our eyes in this example is one of the two
basic classical examples for the concept of a \bfind{place}. (The other
one, the $p$-adic place, comes from number theory.) Traditionally, the
application of a place $P$ is written in the form $g\mapsto gP$, where
instead of $gP$ also $g(P)$ was used in the beginning, reminding of the
fact that $P$ originated from an evaluation homomorphism. If you
translate the German ``$g$ an der Stelle $a$ auswerten'' literally, you
get ``evaluate $g$ at the place $a$'', which explains the origin of the
word ``place''.

Associated to a place $P$ is its \bfind{valuation ring} ${\cal O}_P\,$,
the maximal subring on which $P$ is finite, and a valuation $v_P^{ }$. In
our case, the value $v_P^{ }(g/h)$ is determined by computing the zero or
pole order of $g/h$ (pole orders taken to be negative integers). In
this way, we obtain values in $\Z$, which is the value group of $v_P^{ }$.
In general, given a field $L$ with place $P$ and associated valuation
$v_P^{ }$, the valuation ring ${\cal O}_P=\{b\in L\mid bP\ne\infty\}=
\{b\in L\mid v_P^{ }b\geq 0\}$ has a unique maximal ideal ${\cal M}_P=
\{b\in L\mid bP=0\}=\{b\in L\mid v_P^{ }b>0\}$. The \bfind{residue field}
is $LP:={\cal O}_P/{\cal M}_P\,$ so that $P$ restricted to ${\cal O}_P$
is just the canonical epimorphism ${\cal O}_P\rightarrow LP$. The
characteristic of $LP$ is called the \bfind{residue characteristic}
of $(L,P)$. If $P$ is the identity on $K\subseteq L$, then $K\subseteq
LP$ canonically. The valuation $v_P^{ }$ can be defined to be the
homomorphism $L^\times \rightarrow L^\times/{\cal O}_P^\times$. The
latter is an ordered abelian group, the \bfind{value group} of $(L,
v_P^{ })$. We denote it by $v_P^{ }L$ and write it additively. Note that
$bP\ne \infty\Leftrightarrow b\in {\cal O}_P \Leftrightarrow v_P^{ }b
\geq 0$, and $bP=0\Leftrightarrow b\in {\cal M}_P \Leftrightarrow
v_P^{ }b>0$.

Instead of $(L,P)$, we will often write $(L,v)$ if we talk of valued
fields in general. Then we will write $av$ and $Lv$ instead of $aP$
and $LP$. If we talk of an \bfind{extension of valued fields} and write
$(L|K,v)$ then we mean that $v$ is a valuation on $L$ and $K$ is endowed
with its restriction. If we only have to consider a single extension of
$v$ from $K$ to $L$, then we will use the symbol $v$ for both the
valuation on $K$ and that on $L$. Similarly, we use ``$(L|K,P)$''.

\pars
Observe that in Example~\ref{examp1}, $P$ is uniquely determined by the
homomorphism on $K[x]$. Indeed, we can always write $g/h$ in a form such
that $a$ is not a zero of both $g$ and $h$. If then $a$ is not a
zero of $h$, we have that $(g/h)P=g(a)/h(a)\in K(a)$. If $a$ is a zero
of $h$, we have that $(g/h)P=\infty$. Thus, the residue field of $P$ is
$K(a)$, and the value group is $\Z$. On the other hand, we have the same
non-uniqueness for places as we had for homomorphisms: also places can
be composed with isomorphisms. If $P,Q$ are places of an arbitrary field
$L$ and there is an isomorphism $\sigma: LP\rightarrow LQ$ such that
$\sigma (bP)=bQ$ for all $b\in {\cal O}_P\,$, then we call $P$ and $Q$
\bfind{equivalent places}. In fact, $P$ and $Q$ are equivalent if and
only if ${\cal O}_P={\cal O}_Q\,$. Nevertheless, it is often more
convenient to work with places than with valuation rings, and we will
just identify equivalent places wherever this causes no problems.

Two valuations $v$ and $w$ are called \bfind{equivalent valuations} if
they only differ by an isomorphism of the value groups; this holds if
and only if                                
$v$ and $w$ have the same valuation ring. As for
places, we will identify equivalent valuations wherever this causes no
problems, and we will also identify the isomorphic value groups.

\pars
At this point, we shall introduce a useful notion. Given a function
field $F|K$, we will call $P$ a \bfind{place of $F|K$} if it is a place
of $F$ whose restriction to $K$ is the identity. We say that $P$ is
\bfind{trivial} on $K$ if it induces an isomorphism on $K$. But then,
composing $P$ with the inverse of this isomorphism, we find that $P$ is
equivalent to a place of $F$ whose restriction to $K$ is the identity.
Note that a place $P$ of $F$ is trivial on $K$ if and only if $v_P^{ }$
is \bfind{trivial} on $K$, i.e., $v_P^{ }K=\{0\}$. This is also
equivalent to $K\subset {\cal O}_P\,$. A place $P$ of $F|K$ is said to
be a \bfind{rational place} if $FP=K$. The \bfind{dimension}
of $P$, denoted by $\dim P$, is the transcendence degree of $FP|K$.
Hence, $P$ is \bfind{zero-dimensional} if and only if $FP|K$ is
algebraic.

\parm
Let's get back to our problem. The first thing we learn from our
example is the following. Clearly, we would like to extend our
homomorphism of $K[V]$ to a place of $K(V)$ because then, it will induce
a map on $K[V']$. Then we have the chance to say that the point we have
to look at on the new variety (e.g., in order to see whether this one is
simple) is the point given by this map on $K[V']$. But this only makes
sense if this map is a homomorphism of $K[V']$. So we have to require:
\[y_1,\ldots,y_k\>\in\>{\cal O}_P\]
(since then, $K[y_1,\ldots,y_k]\subseteq {\cal O}_P\,$, which implies
that $P$ is a homomorphism on $K[y_1,\ldots,y_k]$).

\pars
This being granted, the next question coming to our mind is whether
to every point there corresponds exactly one place (up to equivalence),
as it was the case in Example~\ref{examp1}. To destroy this hope, I
give again a very simple example. It will also serve to introduce
several types of places and their invariants.

\begin{example}                             \label{examp2}
Consider the coordinate ring $K[x_1,x_2]$ of $V=\A_K^2$. That is, $x_1$
and $x_2$ are algebraically independent over $K$, and the function
field $K(V)=K(x_1,x_2)$ is just the rational function field in two
variables over $K$. A homomorphism of the polynomial ring
$K[V]=K[x_1,x_2]$ is given by ``evaluating'' every polynomial
$g(x_1,x_2)$ at $x_1=a_1\,$, $x_2=a_2\,$. For example, let us take
$a_1=a_2=0$ and try to extend the corresponding homomorphism of
$K[x_1,x_2]$ to $K(x_1,x_2)$. It is clear that $1/x_1$ and
$1/x_2$ go to $\infty$. But what about $x_1/x_2$ or even $x_1^m/x_2^n$?
Do they go to $0$, $\infty$ or some non-zero element in $K$? The answer
is: all that is possible, and there are infinitely many ways to extend
our homomorphism to a place of $K(x_1,x_2)$.

There is one way, however, which seems to be the most well-behaved.
It is to construct a \bfind{place of maximal rank}; we will explain
this notion later in full generality. The idea is to learn from
Example~\ref{examp1} where we replace $K$ by $K(x_2)$ and $x$ by
$x_1$, and extend the homomorphism defined on $K(x_2)[x_1]$ by
$x_1\mapsto 0$ to a unique place $Q$ of
$K(x_1,x_2)$. Its residue field is $K(x_2)$ since $x_1Q=0\in K(x_2)$,
and its value group is $\Z$. Now we do the same for $K(x_2)$, extending
the homomorphism given on $K[x_2]$ by $x_2\mapsto 0$ to a unique place
$\ovl{Q}$ of $K(x_2)$ with residue field $K$ and value group $\Z$. We
compose the two places, in the following way. Take $b\in K(x_1,x_2)$. If
$bQ=\infty$, then we set $bQ\ovl{Q}=\infty$. If $bQ\ne\infty$, then
$bQ\in K(x_2)$, and we know what $bQ\ovl{Q}=(bQ)\ovl{Q}$ is. In this
way, we obtain a place $P=Q\ovl{Q}$ on $K(x_1,x_2)$ with residue field
$K$. We observe that for every $g\in K[x_1,x_2]$, we have that
$g(x_1,x_2) Q\ovl{Q}=g(0,x_2)\ovl{Q}=g(0,0)$, so our place $P$ indeed
extends the given homomorphism of $K[x_1,x_2]$. Now what happens to our
critical fractions? Clearly, $(1/x_1)P=(1/x_1)Q\ovl{Q}=(\infty)\ovl{Q}
=\infty$, and $(1/x_2)P=(1/x_2)Q\ovl{Q}=(1/x_2)\ovl{Q}=\infty$. But what
interests us most is that for all $m>0$ and $n\geq 0$, $(x_1^m/x_2^n)P=
(x_1^m/x_2^n)Q\ovl{Q}=0\ovl{Q}=0$. We see that ``$x_1$ goes more
strongly to $0$ than every $x_2^n\,$''. We have achieved this by sending
first $x_1$ to $0$, and only afterwards $x_2$ to $0$. We have arranged
our action ``lexicographically''.

What is the associated value group? General valuation theory (cf.\
[V], \S3 and \S4, or [ZS]) tells us that for every composition
$P=Q\ovl{Q}$, the value group
$v_{\ovl{Q}} (FQ)$ of the place $\ovl{Q}$ on $FQ$ is a convex subgroup
of the value group $v_P^{ }F$, and that the value group $v_Q^{ }F$ of
$P$ is isomorphic to $v_P^{ }F/v_{\ovl{Q}}(FQ)$. If the subgroup
$v_{\ovl{Q}} (FQ)$ is a direct summand of $v_P^{ }F$ (as it is the case
in our example), then $v_P^{ }F$ is the lexicographically ordered direct
product $v_Q^{ }F\times v_{\ovl{Q}}(FQ)$. Hence in our case, $v_P^{ }
K(x_1,x_2)=\Z\times\Z$, ordered lexicographically. The \bfind{rank of
an abelian ordered group $G$} is the number of proper convex
subgroups of $G$ (or rather the order type of the chain of convex
subgroups, ordered by inclusion, if this is not finite). The \bfind{rank
of $(F,P)$} is defined to be the rank of $v_P^{ }F$. See under the name
``hauteur'' in [V]. In our case, the rank is 2. We will see in
Section~\ref{sectig} that if $P$ is a place of $F|K$, then the rank
cannot exceed the transcendence degree of $F|K$. So our place
$P=Q\ovl{Q}$ has maximal possible rank.

\pars
There are other places of maximal rank which extend our given
homomorphism, but there is also an abundance of places of smaller rank.
In our case, ``smaller rank'' can only mean rank 1, i.e., there is only
one proper convex subgroup of the value group, namely $\{0\}$. For
an ordered abelian group $G$, having rank 1 is equivalent to being
archimedean ordered and to being embeddable in the ordered additive
group of $\R$. Which subgroups of $\R$ can we get as value groups?
To determine them, we look at the \bfind{rational rank} of an ordered
abelian group $G$. It is $\rr G:= \dim_\Q \Q\otimes_{\Z} G\>$ (note that
$\Q\otimes_{\Z} G$ is the \bfind{divisible hull} of $G$). This is the
maximal number of rationally independent elements in $G$. We will see in
Section~\ref{sectig} that for every place $P$ of $F|K$ we have that
\begin{equation}                            \label{rrtr}
\rr v_P^{ } F\>\leq\> \trdeg F|K\;.
\end{equation}
Hence in our case, also the rational rank of $P$ can be at most 2.
The subgroups of $\R$ of rank 2 are well known: they are the groups
of the form $r\Z+s\Z$ where $r>0$ and $s>0$ are rationally independent
real numbers. Moreover, through multiplication by $1/r$, the group
is order isomorphic to $\Z+\frac{s}{r}\Z$. As we identify equivalent
valuations, we can assume all rational rank 2 value groups (of a rank
1 place) to be of the form $\Z+r\Z$ with $0<r\in\R\setminus \Q$. To
construct a place $P$ with this value group on $K(x_1,x_2)$, we proceed
as follows. We want that $v_P^{ }x_1=1$ and $v_P^{ }x_2=r$; then it will
follow that $v_P^{ }K(x_1,x_2)=\Z+r\Z$ (cf.\ Theorem~\ref{prelBour}
below). We observe that for such $P$, $v_P^{ }(x_1^m/x_2^n)=m-nr$, which
is $>0$ if $m/n>r$, and $<0$ if $m/n<r$. Hence, $(x_1^m/x_2^n)P=0$ if
$m/n>r$, and $(x_1^m/x_2^n)P=\infty$ if $m/n<r$. I leave it to you as an
exercise to verify that this defines a unique place $P$ of $K(x_1,x_2)|
K$ with the desired value group and extending our given homomorphism.

Observe that so far every value group was finitely generated, namely by
two elements. Now we come to the groups of rational rank 1. If such a
group is finitely generated, then it is simply isomorphic to $\Z$. How
do we get places $P$ on $K(x_1,x_2)$ with value group $\Z$? A place with
value group $\Z$ is called a \bfind{discrete place}. The idea is to
first construct a place on the subfield $K(x_1)$. We know from
Example~\ref{examp1} that every place of $K(x_1)|K$ (if it is not
trivial on $K(x_1)$) will have value group $\Z$ (cf.\
Theorem~\ref{prelBour}). Then we can try to extend this place from
$K(x_1)$ to $K(x_1,x_2)$ in such a way that the value group doesn't
change.

There are many different ways how this can be done. One possibility is
to send the fraction $x_1/x_2$ to an element $z$ which is transcendental
over $K$. You may verify that there is a unique place which does
this and extends the given homomorphism; it has value group $\Z$ and
residue field $K(z)$. If, as in this case, a place $P$ of $F|K$ has the
property that $\trdeg FP|K=\trdeg F|K -1$, then $P$ is called a
\bfind{prime divisor} and $v_P^{ }$ is called a \bfind{divisorial
valuation}. The places $Q,\ovl{Q}$ were prime divisors, one of $F$, the
other one of $FQ$.

But maybe we don't want a residue field which is transcendental over
$K$? Maybe we even insist on having $K$ as a residue field? Well, then
we can employ another approach. Having already constructed our place
$P$ on $K(x_1)$ with residue field $K$, we can consider the completion
of $(K(x_1),P)$. The \bfind{completion} of an arbitrary valued field
$(L,v)$ is the completion of $L$ with respect to the topology induced by
$v$. Both $v$ and the associated place $P$ extend canonically to this
completion, whereby value group and residue field remain unchanged. Let
us give a more concrete representation of this completion.

Let $t$ be any transcendental element over $K$. We consider the unique
place $P$ of $F|K$ with $tP=0$. The associated valuation is called the
\bfind{$t$-adic valuation}, denoted by $v_t\,$. It is the unique
valuation $v$ on $K(t)$ (up to equivalence) which is trivial on $K$ and
satisfies that $vt>0$. We want to write down the completion of
$(K(t),v_t)$. We define the \bfind{field of formal Laurent series}
(I prefer \bfind{power series field}) over $K$. It is denoted
by $K((t))$ and consists of all formal sums of the form
\begin{equation}
\sum_{i=n}^{\infty} c_i t^i \;\;\;\mbox{ with $n\in\Z$ and $c_i\in
K$}\;.
\end{equation}
I suppose I don't have to tell you in which way the set $K((t))$ can be
made into a field. But I tell you how $v_t$ extends from $K(t)$ to
$K((t))$: we set
\begin{equation}
v_t\sum_{i=n}^{\infty} c_i t^i\>=\>n \;\;\;\mbox{ if $c_n\ne 0$}\;.
\end{equation}
One sees immediately that $v_t K((t))=v_tK(t)=\Z$. For
$b=\sum_{i=n}^{\infty} c_i t^i$ with $c_n\ne 0$, we have that
$bv_t=\infty$ if $m<0$, $bv_t=0$ if $m>0$, and $bv_t=c_0\in K$ if $m=0$.
So we see that $K((t))v_t=K(t)v_t=K$. General valuation theory shows
that $(K((t)),v_t)$ is indeed the completion of $(K(t),v_t)$.

It is also known that the transcendence degree of $K((t))|K(t)$ is
uncountable. If $K$ is countable, this follows directly from the fact
that $K((t))$ then has the cardinality of the continuum. But it is
quite easy to show that the transcendence degree is at least one, and
already this suffices for our purposes here. The idea is to take any
$y\in K((t))$, transcendental over $K(t)$; then $x_1 \mapsto t$,
$x_2\mapsto y$ induces an isomorphism $K(x_1,x_2)\rightarrow K(t,y)$. We
take the restriction of $v_t$ to $K(t,y)$ and pull it back to
$K(x_1,x_2)$ through the isomorphism. What we obtain on $K(x_1,x_2)$ is
a valuation $v$ which extends our valuation $v_P^{ }$ of $K(x_1)$. As is
true for $v_t$, also this extension still has value group $\Z=v_P^{ }
K(x_1)$ and residue field $K=K(x_1)P$. The desired place of $K(x_1,x_2)$
is the place associated with this valuation $v$.

We have now constructed essentially all places on $K(x_1,x_2)$ which
extend the given homomorphism of $K[x_1,x_2]$ and have a finitely
generated value group (up to certain variants, like exchanging the role
of $x_1$ and $x_2$). The somewhat shocking experience to every
``newcomer'' is that on this rather simple rational function field,
there are also places extending the given homomorphism and having
a value group which is not finitely generated. For instance, the value
group can be $\Q$. (In fact, it can be any subgroup of $\Q$.) We
postpone the construction of such a place till Section~\ref{sectbad}.
\end{example}

After we have become acquainted with places and how one obtains them
from homomorphisms of coordinate rings, it is time to formulate our
problem of local desingularization. Instead of looking for a
desingularization ``at a given point'' of our variety $V$, we will
look for a desingularization at a given place $P$ of the function
field $F|K$ (we forget about the variety from which $F$ originates).
Suppose we have any $V$ such that $K(V)=F$, that is, we have generators
$x_1,\ldots,x_\ell$ of $F|K$ and the coordinate ring $K[x_1,\ldots,
x_\ell]$ of $V$. If we talk about the \bfind{center of $P$ on $V$}, we
always tacitly assume that $x_1,\ldots,x_\ell\in {\cal O}_P\,$, so that
the restriction of $P$ is a homomorphism on $K[x_1,\ldots,x_\ell]$.
With this provision, the center of $P$ on $V$ is the point
$(x_1P,\ldots, x_\ell P)$ (or, if we so want, the induced homomorphism).
We also say that $P$ is \bfind{centered on $V$ at $(x_1P,\ldots, x_\ell
P)$}. If $V$ is a variety defined over $K$ with function field $F$, then
we call $V$ a \bfind{model of $F|K$}. Our problem now reads:
\mn
{\bf (LU)} \ \ {\it Take any function field $F|K$ and a place $P$ of
$F|K$. Does there exist a model of $F|K$ on which $P$ is centered at a
simple point?}
\mn
This was answered in the positive by Oscar Zariski in [Z] for
the case of $K$ having characteristic $0$. Instead of ``local
desingularization'', he called this principle \bfind{local
uniformization}.

%
%
\section{Local uniformization and the Implicit Function Theorem}
Let's think about what we mean by ``simple point''. I don't really have
to tell you, so let me pick the most valuation theoretic definition,
which will show us our way on our excursion. It is the Jacobi criterion:
Given our variety $V$ defined by $f_1,\ldots,f_n\in K[X_1,\ldots,
X_\ell]$ and having function field $F$, then a point $a=(a_1,\ldots,
a_\ell)$ of $V$ is called \bfind{simple} (or \bfind{smooth}) if $\trdeg
F|K=\ell-r$, where $r$ is the rank of the Jacobi matrix
\[\left(\frac{\partial f_i}{\partial X_j}(a)\right)_{1\leq i\leq n
\atop 1\leq j\leq\ell}\]
But wait --- I have seen the Jacobi matrix long before I learned
anything about algebraic geometry. Now I remember: I saw it in my
first year calculus course in connection with the \bfind{Implicit
Function Theorem}.
Let's have a closer look. First, let us assume that
we don't have too many $f_i$'s. Indeed, when looking for a local
uniformization we will construct varieties $V$ defined by $\ell-\trdeg
F|K$ many polynomial relations, whence $n=\ell-\trdeg F|K$. In this
situation, if $a$ is a simple point, then $n$ is equal to $r$ and after
a suitable renumbering we can assume that for $k:=\ell-n =\trdeg F|K$,
the submatrix
\[\left(\frac{\partial f_i}{\partial X_j}(a)\right)_{1\leq i\leq n
\atop k+1\leq j\leq\ell}\]
is invertible. Then, assuming that we are working over the reals, the
Implicit Function Theorem tells us that for every $(a'_1,\dots,a'_k)$ in
a suitably small neighborhood of $(a_1,\dots,a_k)$ there is a unique
$(a'_{k+1},\ldots,a'_\ell)$ such that $(a'_1,\ldots,a'_\ell)$ is a point
of $V$. Working in the reals, the existence is certainly interesting,
but for us here, the main assertion is the uniqueness. Let's look at a
very simple example.

\begin{example}                             \label{exampy2x3}
I leave it to you to draw the graph of the function $y^2=x^3$ in $\R^2$.
It only exists for $x\geq 0$. Starting from the origin, it has two
branches, one positive, one negative. Now assume that we are sitting
on one of these branches at a point $(x,y)$, away from the origin.
If somebody starts to manipulate $x$ then we know exactly which
way we have to run (depending on whether $x$ increases or decreases).
But if we are sitting at the origin and somebody increases $x$, then
we have the freedom of choice into which of the two branches we want
to run. So we see that everywhere but at the origin, $y$ is an implicit
function of $x$ in a sufficiently small neighborhood. Indeed, with
$f(x,y)=x^3-y^2$, we have that $\frac{\partial f}{\partial x}(x,y)
=3x^2$. If $x\ne 0$, then this is non-zero, whence $r=1$ while
$\trdeg F|K =1$ and $\ell=2$, so for $x\ne 0$, $(x,y)$ is a simple
point. On the other hand, $\frac{\partial f}{\partial x}(0,0)=0$ and
$\frac{\partial f}{\partial y}(0,0)=0$, so $(0,0)$ is singular.
\end{example}

We have now seen the connection between simple points and the
Implicit Function Theorem. ``Wait!'' you will interrupt me. ``You have
used the topology of $\R$. What if we don't have such a topology at
hand? What then do you mean by `neighborhood'?'' Good question. So let's
look for a topology. My luck, that the Implicit Function Theorem is also
known in valuation theory. Indeed, we have already remarked in
connection with the notion ``completion of a valued field'' that every
valuation induces a topology. And since we have our place on $F$, we
have the topology right at hand. That is why I said that the
Jacobi criterion renders the most valuation theoretical definition of
``simple''.

But now this makes me think: haven't I seen the Jacobi matrix in
connection with an even more famous valuation theoretical theorem,
one of central importance in valuation theory? Indeed: it appears in the
so-called ``multidimensional version'' of \bfind{Hensel's Lemma}. This
brings us to our next sightseeing attraction on our excursion.

%
%
\section{Hensel's Lemma}
Hensel's Lemma is originally a lemma proved by Kurt Wilhelm Sebastian
Hensel for the field of $p$-adic numbers $\Qp\,$. It was then extended
to all complete discrete valued fields and later to all maximal fields
(see Corollary~\ref{maxhens} below). A valued field $(L,v)$ is called
\bfind{maximal} (or \bfind{maximally complete}) if it has no proper
extensions for which value group and residue field don't change. A
complete field is not necessarily maximal, and if it is not of rank 1
(i.e., its value group is not archimedean), then it also does not
necessarily satisfy Hensel's Lemma. However, complete discrete valued
fields are maximal. In particular, $(K((t)),v_t)$ is maximal.

In modern valuation theory (and its model theory), Hensel's Lemma is
rather understood to be a property of a valued field. The nice thing is
that, in contrast to ``complete'' or ``maximal'', it is an elementary
property (I will tell you in Section~\ref{sectmta} what this means). We
call a valued field \bfind{henselian} if it satisfies Hensel's Lemma.
Here is one version of Hensel's Lemma for a valued field with
valuation ring ${\cal O}_v\,$:
\sn
{\bf (Hensel's Lemma)} \ {\it For every polynomial $f\in {\cal O}_v [X]$
the following holds: if $b\in {\cal O}_v$ satisfies
\begin{equation}                            \label{hhl}
vf(b)\>>\>0 \mbox{ \ \ and \ \ } vf'(b)\>=\>0\;,
\end{equation}
then $f$ admits a root $a\in {\cal O}_v$ such that $v(a-b)>0$.}
\sn
Here, $f'$ denotes the derivative of $f$. Note that a more classical
version of Hensel's Lemma talks only about monic polynomials.

For the multidimensional version, we introduce some notation. For
polynomials $f_1,\ldots,f_n$ in variables $X_1,\ldots, X_n\,$, we
write $f=(f_1,\ldots,f_n)$ and denote by $J_f$ the Jacobian matrix
$\left(\frac{\partial f_i}{\partial X_j}\right)_{i,j}\,$. For $a\in
L^n$, $J_f(a)=\left(\frac{\partial f_i}{\partial X_j}(a)
\right)_{i,j}\,$.

\sn
{\bf (Multidimensional Hensel's Lemma)} \ {\it Let $f= (f_1,\ldots,f_n)$
be a system of polynomials in the variables $X= (X_1,\ldots,X_n)\,$
and with coefficients in ${\cal O}_v$. Assume that there exists
$b=(b_1,\ldots,b_n)\in {\cal O}_v^{\,n}$ such that
\begin{equation}                            \label{hmhl}
vf_i(b)\>>\>0\mbox{ for }1\leq i\leq n\mbox{ \ \ and \ \ }v\det
J_f(b)\>=\>0\;.
\end{equation}
Then there exists a unique $a=(a_1,\ldots,a_n)\in {\cal O}_v^{\,n}$ such
that $f_i(a)=0\,$ and that $\,v(a_i-b_i)>0\,$ for all $i$.}
\mn
And here is the valuation theoretical Implicit Function Theorem:
\sn
{\bf (Implicit Function Theorem)} \ {\it Take $f_1,\ldots,f_n\in
L[X_1,\ldots,X_\ell]$ with $n<\ell$. Set
\begin{equation}
\tilde{J}\>:=\>\left(\begin{array}{ccc}
\frac{\partial f_1}{\partial X_{\ell-n+1}} &\ldots&
\frac{\partial f_1}{\partial X_{\ell}}\\
\vdots & & \vdots\\
\frac{\partial f_n}{\partial X_{\ell-n+1}} &\ldots&
\frac{\partial f_n}{\partial X_{\ell}}
\end{array}\right)\;\;.
\end{equation}
Assume that $f_1,\ldots,f_n$ admit a common zero $a=(a_1,\ldots,a_\ell)
\in L^\ell$ and that $\det \tilde{J}(a)\ne 0$. Then there is some
$\alpha\in vL$ such that for all $(a'_1,\ldots,a'_{\ell-n})\in
L^{\ell-n}$ with $v(a_i-a'_i)>2\alpha$, $1\leq i\leq \ell-n$, there
exists a unique $(a'_{\ell-n+1},\ldots,a'_{\ell})\in L^n$ such that
$(a'_1,\ldots, a'_{\ell})$ is a common zero of $f_1, \ldots,f_n\,$,
and $v(a_i-a'_i)>\alpha$ for $\ell-n< i\leq\ell$.}
\sn
It is (not all too well) known that Hensel's Lemma holds in $(L,v)$
if and only if the Multidimensional Hensel's Lemma holds in $(L,v)$, and
this in turn is true if and only if the Implicit Function Theorem holds
in $(L,v)$. For a proof, see [K2] or [PZ]. The latter paper is
particularly interesting since it shows the connection between the
Implicit Function Theorem in henselian fields and
the ``real'' Implicit Function Theorem in $\R$.

There are many more versions of Hensel's Lemma which all are equivalent
to the above (the classical Hensel's Lemma for monic polynomials,
Krasner's Lemma, Newton's Lemma, Hensel--Rychlik,...). See [R2] or [K2]
for a listing of them. It is indeed often very useful to have the
different versions at hand. One particularly important is given in the
following lemma:
\begin{lemma}                               \label{hensuniq}
A valued field $(L,v)$ is henselian if and only if the extension of
$v$ to the algebraic closure $\tilde{L}$ of $L$ is unique.
\end{lemma}
Since any valuation of any field can always be extended to any extension
field (cf.\ [V], \S5), the following is an easy consequence of this
lemma:
{\it $(L,v)$ is henselian if and only if $v$ admits a unique extension
to every algebraic extension field.\/} Also, we immediately obtain:
\begin{corollary}                           \label{algexth}
Every algebraic extension of a henselian field is again henselian.
\end{corollary}
This is hard to prove if you use Hensel's Lemma instead of the unique
extension property in the proof. On the other hand, the next lemma is
hard to prove using the unique extension property, while it is immediate
if you use Hensel's Lemma:
\begin{lemma}
Take a henselian field $(L,v)$ and a relatively algebraically closed
subfield $L'$ of $L$. Then also $(L',v)$ is henselian.
\end{lemma}

\parm
Let us take a short break to see how Hensel's Lemma can be applied. The
following two examples will later have important applications.
\begin{example}                             \label{exampHL1}
Assume that $\chara L=p>0$. A polynomial $f(X)=X^p-X-c$ with $c\in L$ is
called an \bfind{Artin-Schreier polynomial} (over $L$). If $\vartheta$
is a root of $f$ in some extension of $L$, then $\vartheta,\vartheta+1,
\ldots,\vartheta+p-1$ are the distinct roots of $f$. Hence if $f$ is
irreducible over $L$, then $L(\vartheta)|L$ is a Galois extension of
degree $p$. It is called an \bfind{Artin-Schreier extension}.
Conversely, {\it every\/} Galois extension of degree $p$ in
characteristic $p$ is generated by a root of a suitable Artin-Schreier
polynomial, i.e., is an Artin-Schreier extension (see [K2] for a proof).

Let us prove our assertion about the roots of $f$. We note that in
characteristic $p>0$, the map $x\mapsto x^p$ is a ring homomorphism (the
\bfind{Frobenius}). Therefore, the polynomial $\wp(X):=X^p-X$ is an
\bfind{additive polynomial}. A polynomial $g$ is called additive if
$g(a+b)=g(a)+g(b)$ for all $a,b$ (for details, cf.\ [L2], VIII, \S 11).
Thus, if $i\in\Fp\,$, then $f(\vartheta+i)=\wp(\vartheta+i)-c=
\wp(\vartheta)-c+\wp(i)=0+i^p-i= i-i=0$ since $i^p=i$ for every
$i\in\Fp\,$.

Now assume that $(L,v)$ is henselian. Suppose first that $vc>0$. Take
$b=0\in {\cal O}_v\,$. Then $vf(b)=vc>0$. On the other hand, $f'(X)=
pX^{p-1}-1=-1$ since $p=0$ in characteristic $p$. Hence, $vf'(b)=v(-1)
=0$. Therefore, Hensel's Lemma shows that $f$ admits a root in $L$,
which by our above observation about the roots of $f$ means that $f$
splits completely over $L$.

Suppose next that $vc=0$. Then for $b\in {\cal O}_v$ we have that
$v(b^p-b-c)>0$ if and only if $0=(b^p-b-c)v=(bv)^p-bv-cv$. Hence,
$v(b^p-b-c)>0$ if and only if $bv$ is a root of the Artin-Schreier
polynomial $X^p-X-cv\in Lv[X]$. If $cv=0$, which is our previous case
where $vc>0$, then $0$ is a root of $X^p-X-cv=X^p-X$ and we can choose
$b=0$. But in our present case, $cv\ne 0$, and everything depends on
whether $X^p-X-cv$ has a root in $Lv$ or not. If it has a root $\eta$ in
$Lv$, then we choose $b\in {\cal O}_v$ such that $bv=\eta$. We obtain
that $(b^p-b-c)v=(bv)^p-bv-cv=\eta^p-\eta-cv=0$, hence $vf(b)>0$. Then
by Hensel's Lemma, $X^p-X-c$ has a root $a\in {\cal O}_v$ with
$v(a-b)>0$, hence $av=bv=\eta$. Conversely, if $X^p-X-c$ has a root $a$
in $L$, then one easily shows that $a\in {\cal O}_v\,$, and $0=0v=
(a^p-a-c)v =(av)^p-av-cv$ yields that $X^p-X-cv$ has a root in $Lv$.

The only remaining case is that of $vc<0$. In this case, $X^p-X-c\notin
{\cal O}_v[X]$, so Hensel's Lemma doesn't give us any immediate
information about whether $f$ has a root in $L$ or not.
\end{example}

\begin{example}
Take a field $K$ of characteristic $p>0$. In the field $(K((t)),v_t)\,$
(which is henselian, cf.\ Corollary~\ref{maxhens} below), the
Artin-Schreier polynomial
\begin{equation}
X^p\>-\>X\>-\>t
\end{equation}
has the root
\begin{equation}
a\>=\>\sum_{i=0}^{\infty} (-t)^{p^i}
\end{equation}
since
\[a^p\,-\,a\>=\>\sum_{i=0}^{\infty} (-t)^{p^{i+1}}\>-\>
\sum_{i=0}^{\infty} (-t)^{p^i}\>=\>\sum_{i=1}^{\infty} (-t)^{p^i}\>-\>
\sum_{i=0}^{\infty} (-t)^{p^i}\>=\>t\;.\]
\end{example}

Take any polynomial $f\in {\cal O}_v[X]$. By $fv$ we mean the
\bfind{reduction of the polynomial $f$ modulo $v$}, that is, the
polynomial we obtain from $f$ by replacing every coefficient $c_i$ of
$f$ by its residue $c_iv$. As the residue map is a homomorphism on
${\cal O}_v\,$, we have that $f'v=(fv)'$. Suppose there is some $b\in L$
such that $vf(b)>0$ and $vf'(b)=0$. This is equivalent to $f(b)v=0$ and
$f'(b)v\ne 0$. But $f(b)v=fv(bv)$ and $f'(b)v=(fv)'(bv)$, so the latter
is equivalent to $bv$ being a simple root of $fv$. Conversely, if $fv$
has a simple root $\zeta$, find some $b$ such that $bv=\zeta$, and you
will have that $vf(b)>0$ and $vf'(b)=0$. Hence, Hensel's Lemma is also
equivalent to the following version:
\sn
{\bf (Hensel's Lemma, Simple Root Version)}
{\it For every polynomial $f\in {\cal O}_v [X]$ the following holds: if
$fv$ has a simple root $\zeta$ in $Lv$, then $f$ admits a root $a\in
{\cal O}_v$ such that $av=\zeta$.}

\begin{example}                             \label{exampHL2}
Take a henselian valued field $(L,v)$ and a
relatively algebraically closed subfield $L'$ of $L$. Assume there is
an element $\zeta$ of the residue field $Lv$ which is algebraic over
$L'v$, and denote its minimal polynomial over $L'v$ by $h\in L'v[X]$.
Find a monic polynomial $f\in ({\cal O}_v\cap L')[X]$ such that $fv=h$.

If $\zeta$ is separable over $L'v$, then $\zeta$ is a simple root of
$h$. As $\zeta\in Lv$ and $(L,v)$ is henselian by assumption, the Simple
Root Version of Hensel's Lemma tells us then that there is some $a\in L$
such that $h(a)=0$ and $av=\zeta$. But as $a$ is algebraic over $L'$ we
have that $a\in L'$, so that $\zeta=av\in L'v$. If on the other hand
$\zeta$ is not separable over $L'v$, then it is quite possible that
$\eta\notin L'v$. But we have proved:
\begin{lemma}                               \label{HLrel1}
If $(L,v)$ is henselian and $L'$ is relatively algebraically closed
in $L$, then $L'v$ is \bfind{relatively separable-algebraically closed}
in $Lv$, i.e., every element of $Lv$ already belongs to $L'v$ if it is
separable-algebraic over $L'v$.
\end{lemma}

Something similar can be shown for the value groups, provided that
$Lv=L'v$. Pick an element $\delta\in vL$ such that for some $n>0$,
$n\delta\in vL'$. Choose some $d\in L$ such that $vd=\delta$. Hence,
$vd^n=nvd\in vL'$ and we can choose some $d'\in L'$ such that
$vd'd^n=0$. Assuming that $Lv=L'v$, we can also pick some $d''\in L'$
such that $(d'd''d^n)v=1$.

An element $u$ with $uv=1$ is called a \bfind{1-unit}. We consider the
polynomial $X^n-u$. Its reduction modulo $v$ is simply the polynomial
$X^n-1$. Obviously, $1$ is a root of that polynomial, but is it a simple
root? The answer is: $1$ is a simple root of $X^n-1$ if and only if the
characteristic of $Lv$ does not divide $n$. Hence in that case, Hensel's
Lemma shows that there is a root $a\in L$ of the polynomial $X^n-u$
such that $av=1$. This proves:
\begin{lemma}                               \label{1-unit}
Take a 1-unit $u$ in the henselian field $(L,v)$ and $n\in\N$ such that
the characteristic of $Lv$ does not divide $n$. Then there is a unique
1-unit $a\in L$ such that $a^n=u$.
\end{lemma}

In our present case, this provides an element $a\in L$ such that $a^n=
d'd''d^n$. We find that $(a/d)^n=d'd''\in L'$. Since $a/d\in L$ and $L'$
is relatively algebraically closed in $L$, this implies that $a/d\in
L'$. On the other hand, $v(a/d)^n=vd'd''=vd'=n\alpha$ so that
$v(a/d)=\alpha$. This proves that $\alpha\in vL'$. We have proved:
\begin{lemma}                               \label{HLrel2}
If $(L,v)$ is henselian and $L'$ is relatively algebraically closed
in $L$ and $Lv=L'v$, then the torsion subgroup of $vL/vL'$ is trivial if
$\chara Lv=0$, and it is a $p$-group if $\chara Lv=p>0$.
\end{lemma}
It can be shown that the assertion is in general not true without the
assumption that $Lv=L'v$.
\end{example}

\pars
Let's return to our variety $V$ which is defined by $f_1,\ldots,
f_n\in K[X_1,\ldots, X_\ell]$ and has coordinate ring $K[x_1,\ldots,
x_\ell]$. We have seen that a point $a=(a_1,\ldots,a_\ell)$ of $V$ is
simple if and only if after a suitable renumbering, the submatrix
\begin{equation}
\tilde{J}\>=\>\left(\frac{\partial f_i}{\partial X_j}(a)\right)_{1\leq
i\leq n \atop k+1\leq j\leq\ell}
\end{equation}
of $J_f(a)$ is invertible, where $k:=\ell-n=\trdeg F|K$. That means that
$f_1,\ldots,f_n$ and $a$ satisfy the assumptions of the Implicit
Function Theorem.

Since we are interested in the question whether the center of $P$ on $V$
is simple, we have to look at $a=(x_1P,\ldots,x_\ell P)$. As $P$ is a
homomorphism on ${\cal O}_P$ and leaves the coefficients of the $f_i$
invariant, we see that
\begin{equation}
\left(\frac{\partial f_i}{\partial X_j}(x_1P,\ldots,x_\ell P)
\right)\>=\> \left(\frac{\partial f_i}{\partial X_j}(x_1,\ldots,x_\ell)
\right) P\;.
\end{equation}
We have omitted the indices since this holds for {\it every} submatrix
of $J_f\,$. Again because $P$ is a homomorphism, it commutes with taking
determinants (since this operation remains inside the ring
${\cal O}_P$). Hence,
\begin{equation}
\det\tilde{J}(x_1P,\ldots,x_\ell P)\>=\>
(\det\tilde{J}(x_1,\ldots,x_\ell))P\;.
\end{equation}
Therefore, $\det\tilde{J}(x_1P,\ldots,x_\ell P)\ne 0$ is equivalent to
$v\det\tilde{J}(x_1,\ldots,x_\ell)=0$. This condition also appears in
the Multidimensional Hensel's Lemma, but with $J_f$ in the place of
$\tilde{J}$. So we are led to the question: what is the connection?
It is obvious that we have some variables too many for the case of
the Multidimensional Hensel's Lemma. But they are exactly $\trdeg F|K$
too many, and on the other hand, at least the basic Hensel's Lemma
obviously talks about algebraic elements (we will see that this is also
true for the Multidimensional Hensel's Lemma). So why don't we just
take $x_1,\ldots,x_k$ as a transcendence basis of $F|K$ and view
$f_1,\ldots,f_n$ as polynomial relations defining the remaining
$x_{k+1},\ldots,x_\ell\,$, which are algebraic over $K(x_1,\ldots,x_k)$?
But then, we should write every $f_i$ as a polynomial $\tilde{f}_i$ in
the variables $X_{k+1},\ldots,X_\ell$ with coefficients in
$K(x_1,\ldots,x_k)$, or actually, in $K[x_1,\ldots,x_k]$. Then we have
that
\begin{eqnarray*}
f_i(x_1,\ldots,x_\ell)P & = & f_i(x_1P,\ldots,x_\ell P)\\
& = & \tilde{f}_i P(x_{k+1}P,\ldots,x_\ell
P)\>=\>\tilde{f}_i(x_{k+1},\ldots,x_\ell)P\;.
\end{eqnarray*}
With $\tilde{f}:=(\tilde{f}_1,\ldots,\tilde{f}_n)$ and
$\tilde{f}P:=(\tilde{f}_1P,\ldots,\tilde{f}_nP)$, it follows that
\begin{equation}
\det\tilde{J}(x_1P,\ldots,x_\ell P)\>=\>
\det J_{\tilde{f}P}(x_{k+1}P,\ldots,x_\ell P)\>=\>
(\det J_{\tilde{f}}(x_{k+1},\ldots,x_\ell))P\;.
\end{equation}
Hence,
$\det\tilde{J}(x_1P,\ldots,x_\ell P)\ne 0$ is equivalent to $\>v\det
J_{\tilde{f}}(x_{k+1},\ldots,x_\ell)=0\,$, which means that the
polynomials $\tilde{f}_1,\ldots,\tilde{f}_n$ and the elements $x_{k+1},
\ldots,x_\ell$ satisfy the assumption (\ref{hmhl}) of the
Multidimensional Hensel's Lemma. Indeed, we have that
$v\tilde{f}_i(x_{k+1},\ldots,x_\ell)=\infty>0$ since
$\tilde{f}_i(x_{k+1},\ldots,x_\ell)=0$. So we see:
\mn
{\it To find a model of $F|K$ on which $P$ is centered at a simple
point means to find generators $x_1,\ldots,x_\ell\in {\cal O}_P$ such
that $x_1,\ldots,x_k$ form a transcendence basis of $F|K$ and
$x_{k+1},\ldots,x_\ell$ together with the polynomials which define them
over $K[x_1,\ldots,x_k]$ satisfy the assumption of the Multidimensional
Hensel's Lemma.}
\sn
(Since $f_i$ and $\tilde{f}_i$ are the same polynomial, just written in
two different ways, we will later use ``$f_i$'' instead of
``$\tilde{f}_i$''; cf.\ the definition of relative uniformization in
Section~\ref{sectrlu}. There, we will also prefer ``$\,(\det
J_{\tilde{f}}(x_{k+1},\ldots,x_\ell))P\ne 0\,$'' over ``$\,v\det
J_{\tilde{f}}(x_{k+1},\ldots,x_\ell)=0\,$''.)

\parm
What we have derived now is still quite vague, and before we can make
more out of it, I'm sorry, you have to go to a course again.

%
%
\section{A crash course in ramification theory}       \label{sectram}
Throughout, we assume that $L|K$ is an algebraic extension, not
necessarily finite, and that $v$ is a {\it non-trivial} valuation on $K$.
If $w$ is a valuation on $L$ which extends $v$, then there is a natural
embedding of the value group $vK$ of $v$ in the value group $wL$ of
$w$. Similarly, there is a natural embedding of the residue field $Kv$
of $v$ in the residue field $Lw$ of $w$. If both embeddings are onto
(which we just express by writing $vK=wL$ and $Kv=Lw$), then the
extension $(L,w)|(K,v)$ is called {\bf immediate}. WARNING: It may
happen that $vK\isom wL$ or $Kv\isom Lw$ although the corresponding
embedding is not onto and therefore, the extension is not immediate.
For example, every finite extension of the $p$-adics $(\Q_p,v_p)$ will
again have a value group isomorphic to $\Z$, but $v_pp$ may not be
anymore the smallest positive element in this value group.

\pars
We choose an arbitrary extension of $v$ to the algebraic closure
$\tilde{K}$ of $K$. Then for every $\sigma\in\Aut(\tilde{K}|K)$, the map
\begin{equation}
\tilde{v}\sigma=\tilde{v}\circ\sigma:\;L\ni a \>\mapsto\>
\tilde{v}(\sigma a)\in \tilde{v}\tilde{K}
\end{equation}
is a valuation of $L$ which extends $v$.
\begin{theorem}                             \label{allext}
The set of all extensions of $v$ from $K$ to $L$ is
\[\{\tilde{v}\sigma\mid \sigma \mbox{ an embedding of $L$ in
$\tilde{K}$ over $K$}\}\;.\]
\end{theorem}
(We say that ``all extensions of $v$ from $K$ to $L$ are
\bfind{conjugate}''.)
\begin{corollary}                              \label{cae}
If $L|K$ is finite, then the number {\rm g} of distinct extensions of
$v$ from $K$ to $L$ is smaller or equal to the extension degree $[L:K]$.
More precisely, {\rm g} is smaller or equal to the degree of the maximal
separable subextension of $L|K$. In particular, if $L|K$ is purely
inseparable, then $v$ has a unique extension from $K$ to $L$.
\end{corollary}
\begin{theorem}                             \label{fie}
Assume that $n:=[L:K]$ is finite, and denote the extensions of $v$ from
$K$ to $L$ by $v_1,\ldots,v_{\rm g}\,$. Then
for every $i\in\{1,\ldots,{\rm g}\}$, the \bfind{ramification index}
${\rm e}_i=(v_iL:vK)$ and the \bfind{inertia degree} ${\rm f}_i=
[Lv_i:Kv]$ are finite, and we have the \bfind{fundamental inequality}
\begin{equation}                             \label{fundineq}
n\>\geq\>\sum_{i=1}^{\rm g} {\rm e}_i {\rm f}_i  \;\>.
\end{equation}
\end{theorem}

\parb
From now on, let us assume that $L|K$ is normal. Hence, the set of all
extensions of $v$ from $K$ to $L$ is given by $\{\tilde{v}\sigma
\mid\sigma\in\Aut(L|K)\}$. For simplicity, we denote the restriction of
$\tilde{v}$ to $L$ again by $v$. The valuation ring of $v$ on $L$ will
be denoted by ${\cal O}_L\,$. We define distinguished subgroups of
$G:=\Aut(L|K)$. The subgroup
\begin{equation}                            \label{decgrp}
G\dec :=G\dec (L|K,v):=\{\sigma\in G\mid\forall x\in
{\cal O}_L:\;v\sigma x\geq 0\}
\end{equation}
is called the \bfind{decomposition group} {\bf of} $(L|K,v)$. It is easy
to show that $\sigma$ sends ${\cal O}_L$ into itself if and only if
the valuations $v$ and $v\sigma$ agree on $L$. Thus,
\begin{equation}                            \label{decgrp1}
G\dec =\{\sigma\in G\mid v\sigma=v\mbox{ on }L\}\;.
\end{equation}
Further, the \bfind{inertia group} is defined to be
\begin{equation}                            \label{ingrp}
G^i:=G^i (L|K,v) := \{\sigma\in \Aut (L|K)\mid \forall x\in
{\cal O}_L: v(\sigma x - x)>0\}\;,
\end{equation}
and the \bfind{ramification group} is
\begin{equation}                            \label{ramgrp}
G^r:=G^r (L|K,v) :=\{\sigma\in \Aut(L|K)\mid
\forall x\in {\cal O}_L: v(\sigma x - x)>vx\}\;.
\end{equation}
Let S denote the maximal separable extension of $K$ in $L$ (we call
it the \bfind{separable closure of $K$ in $L$}). The fixed fields of
$G\dec$, $G^i$ and $G^r$ in S are called \bfind{decomposition
field}, \bfind{inertia field} and \bfind{ramification field} {\bf of}
$(L|K,v)$. For simplicity, let us abbreviate them by Z, T and V. (These
letters refer to the german words ``Zerlegungsk\"orper'',
``Tr\"ag\-heits\-k\"orper'' and ``Verzweigungsk\"orper''.)
\begin{remark}                            
In contrast to the classical definition used by other
authors, we take decomposition field, inertia field and ramification
field to be the fixed fields of the respective groups {\it in the
maximal separable subextension}. The reason for this will become clear
in Section~\ref{sectdef}.
\end{remark}

By our definition, V, T and Z are separable-algebraic extensions of $K$,
and S$|$V, S$|$T, S$|$Z are (not necessarily finite) Galois extensions.
Further,
\begin{equation}
1\subset G^r\subset G^i\subset G\dec\subset G\;\mbox{ and thus, }\;
{\rm S}\supset {\rm V}\supset {\rm T}\supset {\rm Z}\supset K\;.
\end{equation}
(For the inclusion $G^i\subset G\dec$ note that $vx\geq 0$ and
$v(\sigma x-x)>0$ implies that $v\sigma x\geq 0$.)

\begin{theorem}
$G^i$ and $G^r$ are normal subgroups of $G\dec$, and $G^r$ is a normal
subgroup of $G^i$. Therefore, ${\rm T}|{\rm Z}$, ${\rm V}|{\rm Z}$ and
${\rm V}|{\rm T}$ are (not necessarily finite) Galois extensions.
\end{theorem}

First, we consider the decomposition field {\rm Z}. In some sense, it
represents all extensions of $v$ from $K$ to $L$.
\begin{theorem}                             \label{Z}
a) \ $v\sigma=v\tau$ on $L$ if and only if $\sigma\tau^{-1}$ is trivial
on {\rm Z}.\n
b) \ $v\sigma=v$ on {\rm Z} if and only if $\sigma$ is trivial
on {\rm Z}.\n
c) \ The extension of $v$ from {\rm Z} to $L$ is unique.\n
d) \ The extension $({\rm Z}|K,v)$ is immediate.
\end{theorem}
WARNING: It is in general not true that $v\sigma\ne v\tau$ holds already
on {\rm Z} if it holds on $L$.\sn
a) and b) are easy consequences of the definition of $G\dec$.
c) follows from b) by Theorem~\ref{allext}. For d), there is a simple
proof using a trick which is mentioned in the paper [AX] by James Ax.

\parm
Now we turn to the inertia field T. Let ${\cal M}_L$ denote the
valuation ideal of $v$ on $L$ (the unique maximal ideal of
${\cal O}_L$). For every $\sigma\in G\dec (L|K,v)$ we have that
$\sigma {\cal O}_L={\cal O}_L$, and it follows that $\sigma {\cal M}_L
={\cal M}_L$. Hence, every such $\sigma$ induces an automorphism
$\ovl{\sigma}$ of ${\cal O}_L/{\cal M}_L=Lv$ which
satisfies $\ovl{\sigma}\,\ovl{a}= \ovl{\sigma a}$. Since
$\sigma$ fixes $K$, it follows that $\ovl{\sigma}$ fixes $Kv$.

\begin{lemma}
Since $L|K$ is normal, the same is true for $Lv|Kv$. The map
\begin{equation}                            \label{redofauto}
G\dec (L|K,v)\ni\sigma\;\mapsto\;\ovl{\sigma}\in\Aut (Lv|Kv)
\end{equation}
is a group homomorphism.
\end{lemma}

\begin{theorem}                             \label{T}
a) \ The homomorphism (\ref{redofauto}) is onto and induces an
isomorphism
\begin{equation}
\Aut({\rm T}|{\rm Z})\>=\>G\dec/G^i\>\isom\>
\Aut({\rm T}v|{\rm Z}v)\;.
\end{equation}
b) \ For every finite subextension $F|{\rm Z}$ of ${\rm T}|{\rm Z}$,
\begin{equation}
[F:{\rm Z}]\>=\>[Fv:{\rm Z}v]\;.
\end{equation}
c) \ We have that $v{\rm T}=v{\rm Z}=vK$. Further, ${\rm T}v$ is the
separable closure of $Kv$ in $Lv$, and therefore,
\begin{equation}
\Aut({\rm T}v|{\rm Z}v)\>=\>\Aut(Lv|Kv)\;.
\end{equation}
\end{theorem}
If $F|{\rm Z}$ is normal, then b) is an easy consequence of a). From
this, the general assertion of b) follows by passing from $F$ to the
normal hull of the extension $F|{\rm Z}$ and then using the
multiplicativity of the extension degree. c) follows from b) by use of
the fundamental inequality.

\parm
We set $p:=\chara Kv$ if this is positive, and $p:=1$ if $\chara Kv=0$.
Given any extension $\Delta\subset \Delta'$ of abelian groups, the
\bfind{$p'$-divisible closure of $\Delta$ in $\Delta'$} is defined to be
the subgroup $\{\alpha\in\Delta'\mid\exists n\in\N:\;(p,n) =1\,
\wedge\,n\alpha\in\Delta\}$ of all elements in $\Delta'$
whose order modulo $\Delta$ is prime to $p$.
\begin{theorem}                             \label{V}
a) \ There is an isomorphism
\begin{equation}
\Aut ({\rm V}|{\rm T})\>=\> G^i/G^r\>\isom\>
\mbox{\rm Hom}\left(v{\rm V}/v{\rm T}\,,\,({\rm T}v)^\times\right)\;,
\end{equation}
where the character group on the right hand side is the full
character group of the abelian group $v{\rm V}/v{\rm T}$.
Since this group is abelian, ${\rm V}|{\rm T}$ is an abelian Galois
extension.\n
b) \ For every finite subextension $F|{\rm T}$ of ${\rm V}|{\rm T}$,
\begin{equation}
[F:{\rm T}]\>=\>(vF:v{\rm T})\;.
\end{equation}
c) \ ${\rm V}v={\rm T}v$, and $v{\rm V}$ is the $p'$-divisible closure
of $vK$ in $vL$.
\end{theorem}
b) follows from a) since for a finite extension $F|{\rm T}$, the group
$vF/v{\rm T}$ is finite and thus there exists an isomorphism of
$vF/v{\rm T}$ onto its full character group. The equality ${\rm V}v=
{\rm T}v$ follows from b) by the fundamental inequality. The second
assertion of c) follows from the next theorem and the fact that the
order of all elements in $({\rm T}v)^\times$ and thus also of all
elements in $\mbox{\rm Hom} \left(v{\rm V}/v{\rm T}\,,\,
({\rm T}v)^\times \right)$ is prime to $p$.

\begin{theorem}                             \label{S}
The ramification group $G^r$ is a $p$-group and therefore,
${\rm S}|{\rm V}$ is a $p$-extension. Further, $vL/v{\rm V}$ is a
$p$-group, and the residue field extension $Lv|{\rm V}v$ is
purely inseparable. If $\chara Kv=0$, then ${\rm V}={\rm S}=L$.
\end{theorem}
We note:
\begin{lemma}                              \label{RemASe}
Every $p$-extension is a tower of Galois extensions of degree $p$. In
characteristic $p$, all of them are Artin--Schreier--extensions, as we
have mentioned in Example~\ref{exampHL1}.
\end{lemma}

\pars
From Theorem~\ref{S} it follows that there is a canonical isomorphism
\begin{equation}                            \label{cg}
\mbox{\rm Hom}\left(v{\rm V}/v{\rm T}\,,\,({\rm T}v)^\times\right)
\>\isom\>\mbox{\rm Hom}\left(vL/vK\,,\,(Lv)^\times\right)\;.
\end{equation}
\parm
We summarize our main results in the following table:\n
\setlength{\unitlength}{0.0015\textwidth}
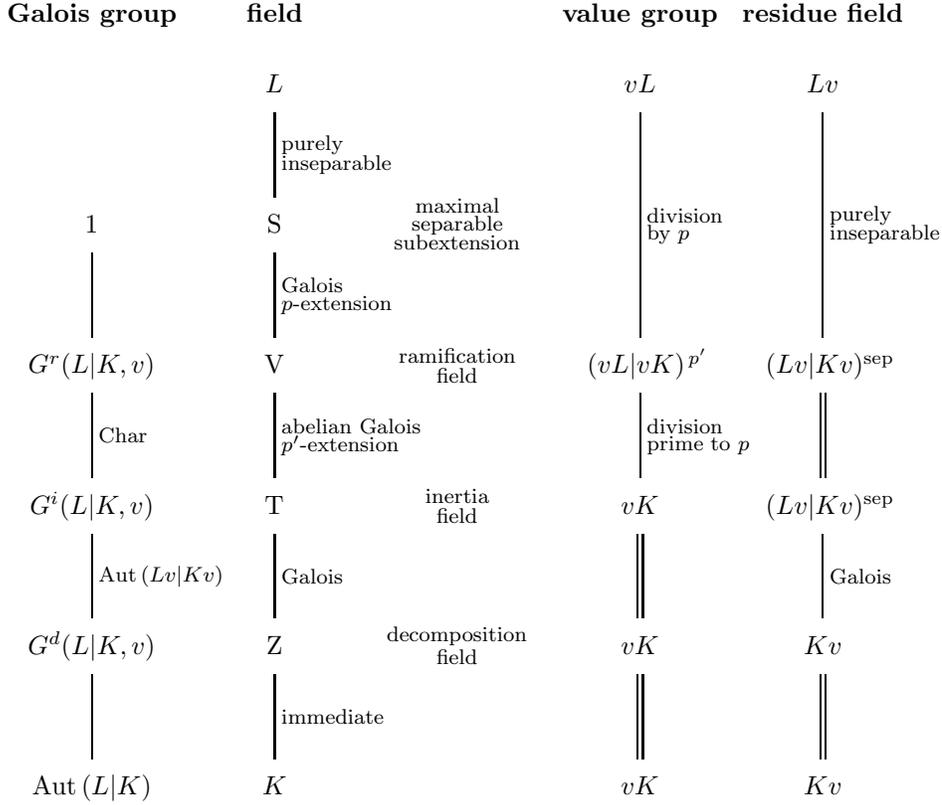
\begin{picture}(650,640)(10,0)              \label{ramtable}
\put(70,50){\bbox{$\Aut (L|K)$}}
\put(70,150){\bbox{$G\dec (L|K,v)$}}
\put(70,250){\bbox{$G^i (L|K,v)$}}
\put(70,350){\bbox{$G^r (L|K,v)$}}
\put(70,450){\bbox{$1$}}
\put(70,550){\bbox{$$}}
\put(70,600){\bbox{\bf Galois group}}
\put(200,50){\bbox{$K$}}
\put(200,150){\bbox{Z}}
\put(200,250){\bbox{T}}
\put(200,350){\bbox{V}}
\put(200,450){\bbox{S}}
\put(200,550){\bbox{$L$}}
\put(200,600){\bbox{\bf {field}}}
\put(460,50){\bbox{$vK$}}
\put(460,150){\bbox{$vK$}}
\put(460,250){\bbox{$vK$}}
\put(460,350){\bbox{$\;\;(vL|vK)^{\,p'}$}}
\put(460,550){\bbox{$vL$}}
\put(460,600){\bbox{\bf {value group}}}
\put(590,50){\bbox{$Kv$}}
\put(590,150){\bbox{$Kv$}}
\put(590,250){\bbox{$\;\;(Lv|Kv)\sep$}}
\put(590,350){\bbox{$\;\;(Lv|Kv)\sep$}}
\put(590,550){\bbox{$Lv$}}
\put(590,600){\bbox{\bf {residue field}}}
\put(70,130){\line(0,-1){60}}    
\put(70,230){\line(0,-1){60}}    
\put(70,330){\line(0,-1){60}}    
\put(70,430){\line(0,-1){60}}    
\put(200,130){\line(0,-1){60}}    
\put(200,230){\line(0,-1){60}}    
\put(200,330){\line(0,-1){60}}    
\put(200,430){\line(0,-1){60}}    
\put(200,530){\line(0,-1){60}}    
\put(458,130){\line(0,-1){60}}    
\put(462,130){\line(0,-1){60}}    
\put(458,230){\line(0,-1){60}}    
\put(462,230){\line(0,-1){60}}    
\put(460,330){\line(0,-1){60}}    
\put(460,530){\line(0,-1){160}}    
\put(588,130){\line(0,-1){60}}    
\put(592,130){\line(0,-1){60}}    
\put(590,230){\line(0,-1){60}}    
\put(588,330){\line(0,-1){60}}    
\put(592,330){\line(0,-1){60}}    
\put(590,530){\line(0,-1){160}}    
\put(75,200){\makebox(0,0)[l]{\footnotesize\rm $\Aut (Lv|Kv)$}}%
\put(75,300){\makebox(0,0)[l]{\footnotesize\rm Char}}%
\put(205,100){\makebox(0,0)[l]{\footnotesize\rm immediate}}%
\put(205,200){\makebox(0,0)[l]{\footnotesize\rm Galois}}%
\put(205,307){\makebox(0,0)[l]{\footnotesize\rm abelian Galois}}%
\put(205,293){\makebox(0,0)[l]{\footnotesize\rm $p'$-extension}}%
\put(205,407){\makebox(0,0)[l]{\footnotesize\rm Galois}}%
\put(205,393){\makebox(0,0)[l]{\footnotesize\rm $p$-extension}}%
\put(205,507){\makebox(0,0)[l]{\footnotesize\rm purely}}%
\put(205,493){\makebox(0,0)[l]{\footnotesize\rm inseparable}}%
\put(465,457){\makebox(0,0)[l]{\footnotesize\rm division}}%
\put(465,443){\makebox(0,0)[l]{\footnotesize\rm by $p$}}%
\put(465,307){\makebox(0,0)[l]{\footnotesize\rm division}}%
\put(465,293){\makebox(0,0)[l]{\footnotesize\rm prime to $p$}}%
\put(595,200){\makebox(0,0)[l]{\footnotesize\rm Galois}}%
\put(595,457){\makebox(0,0)[l]{\footnotesize\rm purely}}%
\put(595,443){\makebox(0,0)[l]{\footnotesize\rm inseparable}}%
\put(330,157){\makebox(0,0)[c]{\footnotesize\rm decomposition}}%
\put(330,143){\makebox(0,0)[c]{\footnotesize\rm field}}%
\put(330,257){\makebox(0,0)[c]{\footnotesize\rm inertia}}%
\put(330,243){\makebox(0,0)[c]{\footnotesize\rm field}}%
\put(330,357){\makebox(0,0)[c]{\footnotesize\rm ramification}}%
\put(330,343){\makebox(0,0)[c]{\footnotesize\rm field}}%
\put(330,464){\makebox(0,0)[c]{\footnotesize\rm maximal}}%
\put(330,450){\makebox(0,0)[c]{\footnotesize\rm separable}}%
\put(330,438){\makebox(0,0)[c]{\footnotesize\rm subextension}}%
\end{picture}
\n
where $(vL|vK)^{\,p'}$ denotes the $p'$-divisible closure of
$vK$ in $vL$, $(Lv|Kv)\sep$ denotes the separable closure of $Kv$ in
$Lv$, and Char denotes the character group (\ref{cg}).

\pars
We state two more useful theorems from ramification theory. If we have
two subfields $K,L$ of a field $M$ (in our case, we will have the
situation that $L\subset\tilde{K}$) then $K.L$ will denote the smallest
subfield of $M$ which contains both $K$ and $L$; it is called the
\bfind{field compositum of $K$ and $L$}.
\begin{theorem}                             \label{liftZTV}
If $K\subseteq K'\subseteq L$, then the decomposition field of the
normal extension $(L|K',v)$ is {\rm Z}$.K'$, its inertia field is
{\rm T}$.K'$, and its ramification field is {\rm V}$.K'\,$.
\end{theorem}
\begin{theorem}                             \label{downZTV}
If $E|K$ is a normal subextension of $L|K$, then the decomposition field
of $(E|K,v)$ is {\rm Z}$\cap E$, its inertia field is {\rm T}$\cap E$,
and its ramification field is {\rm V}$\cap E$.
\end{theorem}

\parm
If we take for $L|K$ the normal extension $\tilde{K}|K$, then we speak
of \bfind{absolute rami\-fication theory}. The fixed fields $K\dec$,
$K^i$ and $K^r$ of $G\dec(\tilde{K}|K,v)$, $G^i(\tilde{K}|K,v)$ and
$G^r(\tilde{K}|K,v)$ in the separable-algebraic closure $K\sep$ of $K$
are called \bfind{absolute decomposition field}, \bfind{absolute inertia
field} and \bfind{absolute ramification field} {\bf of} $(K,v)$ (with
respect to the given extension of $v$ from $K$ to its algebraic closure
$\tilde{K}$). If $\chara Kv=0$, then by Theorem~\ref{S},
$K^r=K\sep=\tilde{K}$.
\begin{lemma}                               \label{Kimaxur}
Fix an extension of $v$ from $K$ to $\tilde{K}$. Then the absolute
inertia field of $(K,v)$ is the unique maximal extension of $(K,v)$
within the absolute ramification field having the same value group as
$K$.
\end{lemma}
\begin{proof}
Let $(L|K,v)$ be any extension within the absolute ramification field
s.t.\ $vL=vK$. Then $vL^i=vL=vK=vK^i$. By Theorem~\ref{liftZTV},
$L^i=L.K^i$. Further, $L\subseteq K^r$ yields that $L.K^i\subseteq K^r$.
If the subextension $L^i|K^i$ of $K^r|K^i$ were proper, it contained
a proper finite subextension $L_1|K^i$, and by part b) of
Theorem~\ref{V} we had that $vK^i\subsetuneq vL_1\subseteq vL^i$. As
this contradicts the fact that $vL^i=vK^i$, we find that $L^i=K^i$,
that is, $L\subseteq K^i$.
\end{proof}

From part c) of Theorem~\ref{Z} we infer that the extension of $v$ from
$K\dec$ to $\tilde{K}$ is unique. On the other hand, if $L$ is any
extension field of $K$ within $K\dec$, then by Theorem~\ref{liftZTV},
$K\dec=L\dec$. Thus, if $L\ne K\dec$, then it follows from part b) of
Theorem~\ref{Z} that there are at least two distinct extensions of $v$
from $L$ to $K\dec$ and thus also to $\tilde{K}=\tilde{L}$. This proves
that the absolute decomposition field $K\dec$ is a minimal algebraic
extension of $K$ admitting a unique extension of $v$ to its algebraic
closure. So it is the minimal algebraic extension of $K$ which is
henselian (cf.\ Lemma~\ref{hensuniq}). We call it the
\bfind{henselization of $(K,v)$ in $(\tilde{K},v)$}. Instead of $K\dec$,
we also write $K^h$. A valued field is henselian if and only if it is
equal to its henselization. Henselizations have the following universal
property:
\begin{theorem}                             \label{hensuniqemb}
Let $(K,v)$ be an arbitrary valued field and $(L,v)$ any henselian
extension field of $(K,v)$. Then there is a unique embedding of
$(K^h,v)$ in $(L,v)$ over $K$.
\end{theorem}

From the definition of the henselization as a decomposition field,
together with part d) of Theorem~\ref{Z}, we obtain another very
important property of the henselization:
\begin{theorem}                             \label{immhens}
The henselization $(K^h,v)$ is an immediate extension of $(K,v)$.
\end{theorem}
\begin{corollary}                           \label{maxhens}
Every maximal valued field is henselian. In particular, $(K((t)),v_t)$
is henselian.
\end{corollary}

Finally, we employ Theorem~\ref{liftZTV} to obtain:
\begin{theorem}                             \label{hensfcomp}
If $K'|K$ is an algebraic extension, then the henselization of
$K'$ is $K'.K^h\,$.
\end{theorem}

%
%
\section{A valuation theoretical interpretation of local uniformization}
We return to where we stopped before entering the crash course in
ramification theory. The first question is: what does it mean that
$x_{k+1},\ldots,x_\ell$ together with the polynomials which define them
over $K[x_1,\ldots,x_k]$ satisfy the assumption of the Multidimensional
Hensel's Lemma? First of all, general valuation theory tells us
that a rational function field $K(x_1,\ldots,x_k)$ is much too small to
be henselian (unless the valuation is trivial). But we could pass to
the henselization of $(K(x_1,\ldots,x_k),P)$. So does it mean that
$x_{k+1},\ldots,x_\ell$ lie in this henselization? If we look closely,
there is something fishy in the way we have satisfied the assumption
of the Multidimensional Hensel's Lemma. Instead of talking about a
so-called ``approximative root'' $b=(b_1,\ldots,b_n)$ which lies in the
henselian field we wish to work in, we have talked already about the
actual root, and we do not know where it lies. Let us modify our
Example~\ref{exampy2x3} a bit to see that it does not always lie in the
henselization of $(K(x_1,\ldots,x_k),v)$.
\begin{example}
Let us consider the function field $\Q(x,y)$ where $y^2=x^3$. Take the
place given by $xP=2$, $yP=2\sqrt{2}$. The minimal polynomial of $y$
over $\Q(x)$ is $f(Y)=Y^2-x^3$. As $f(y)=0$, we have that $v_P^{ }f(y)=
\infty>0$. As $f'(Y)=2Y$, we have that $v_P^{ }f'(y)=v_P^{ }2y=0\>$
(since $2yP=4\sqrt{2}\ne 0$). Hence, $f$ and $y$ satisfy the assumption
(\ref{hhl}) of Hensel's Lemma. But $y$ does not lie in the henselization
of $(\Q(x),P)$. Indeed, $P$ on $\Q(x)$ is just the place coming from the
evaluation homomorphism given by $x\mapsto 2$; hence, $\Q(x)^hP=\Q(x)P=
\Q$. But $\Q(x,y)P\ne\Q$ since $yP =2\sqrt{2}\notin\Q$.
\end{example}

So we see that extensions of the residue field can play a role. We could
try to suppress them by requiring that $K$ be algebraically closed.
This works for those $P$ for which $FP|K$ is algebraic, but if this is
not the case, then we have no chance to avoid them. At least, we can
show that they are the only reason why $x_{k+1},\ldots,x_\ell$ may
not lie in the henselization of $(K(x_1,\ldots,x_k),P)$.
\begin{theorem}
If $x_{k+1},\ldots,x_\ell$ together with the polynomials $f_i$ which
define them over $K[x_1,\ldots,x_k]$ satisfy the assumption (\ref{hmhl})
of the Multidimensional Hensel's Lemma, then $x_{k+1},
\ldots, x_\ell$ lie in the absolute inertia field of $(K(x_1,\ldots,
x_k),P)$, and the extension $FP|K(x_1P,\ldots,x_kP)$ is
separable-algebraic. If in addition $P$ is a rational place, then
$x_{k+1}, \ldots, x_\ell$ lie in the henselization of
$(K(x_1,\ldots,x_k),P)$.
\end{theorem}
\begin{proof}
Denote by $(L,P)$ the absolute inertia field of $(K(x_1,\ldots,x_k),P)$.
First,
\begin{equation}
\det J_{\tilde{f}P}(x_{k+1}P,\ldots,x_\ell P)\>=\>\det J_{\tilde{f}}
(x_{k+1},\ldots,x_\ell)P\>\ne\> 0
\end{equation}
and the fact that the $f_iP$ are polynomials over $K(x_1P,\ldots,x_k P)$
imply that $x_{k+1} P,\ldots,x_\ell P$ are separable algebraic over
$K(x_1P,\ldots,x_k P)$ (cf.\ [L2], Chapter X, \S7, Proposition 8).
On the other hand, $LP$ is the separable-algebraic closure of $K(x_1,
\ldots,x_k)P$. Therefore, there are elements $b_1,\ldots,b_n$ in $L$
such that $b_iP=x_{k+i}P$. Since $(L,P)$ is henselian, the
Multidimensional Hensel's Lemma now shows the existence of a common root
$(b'_1,\ldots,b'_n)\in L^n$ of the $f_i$ such that $b'_iP=b_iP=x_{k+i}
P$. But the uniqueness assertion of the Multidimensional Hensel's Lemma
also holds in the algebraic closure $\tilde{L}$ of $L$ (which is also
henselian). So we find that $(b'_1,\ldots,b'_n)=(x_{k+1},\ldots,
x_\ell)$. Hence, $x_{k+1},\ldots,x_\ell$ are elements of $L$.

\pars
If we have in addition that $P$ is a rational place, then $x_{k+1}P,
\ldots, x_\ell P\in K$. In this case, we can choose $b_1,\ldots,b_n$
and $b'_1,\ldots,b'_n$ already in the henselization of
$(K(x_1,\ldots,x_k),P)$, which implies that also
$x_{k+1},\ldots,x_\ell$ lie in this henselization.
\end{proof}

Since the absolute inertia field is a separable-algebraic extension
and every rational function field is separable, we obtain:
\begin{corollary}
If the place $P$ of $F|K$ admits local uniformization, then $F|K$ is
separable.
\end{corollary}

We see that we are slowly entering the \bfind{structure theory of
valued function fields}, that is, the algebraic theory of function
fields $F|K$ equipped with a valuation (which may or may not be trivial
on $K$). Later, we will see some main results from this theory
(Theorems~\ref{ai} and~\ref{stt3}).

Given a place $P$ of $F$, not necessarily trivial on $K$, we will
say that $(F|K,P)$ is \bfind{inertially generated} if there is a
transcendence basis $T$ of $F|K$ such that $(F,P)$ lies in the absolute
inertia field of $(K(T),P)$. Similarly, $(F|K,P)$ is \bfind{henselian
generated} if there is a transcendence basis $T$ of $F|K$ such that
$(F,P)$ lies in henselization of $(K(T),P)$. Now we see a valuation
theoretical interpretation of local uniformization:
\begin{theorem}                             \label{MT5}
If the place $P$ of $F|K$ admits local uniformization, then
$(F|K,P)$ is inertially generated. If in addition $FP=K$, then
$(F|K,P)$ is henselian generated.
\end{theorem}
So if local uniformization holds in arbitrary characteristic for every
$F|K$ with perfect $K$, then for every place $P$ of $F|K$, the valued
function field $(F|K,P)$ is inertially generated. In the context of
valuation theory, at least to me, this is a quite surprising assertion.
Here is our first open problem:
\mn
{\bf Open Problem 1:} \ Is the converse also true, i.e., if
$(F|K,P)$ is inertially generated, does it then admit
local uniformization?
\mn
I will discuss this question in Section~\ref{sectrlu}. A partial
answer to this question is given in the papers [K5] and [K6].
What we see is that in order to get local uniformization, one has to
avoid ramification. Indeed, ramification is the valuation theoretical
symptom of branching, the violation of the Implicit Function
Theorem at a point of the variety. Let us look again at our simple
Example~\ref{exampy2x3}:
\begin{example}
Consider the function field $\R(x,y)$ where $y^2=x^3$. Take the
place given by $xP=0=yP$. As $P$ on $K(x)$ originates from the
evaluation homomorphism given by $x\mapsto 0$, we have that
$v_P^{ } K(x)=\Z$, with $v_P^{ }x=1$ the smallest positive element in
the value group. Now compute $v_P^{ }y$. We have that $y^2=x^3$, whence
$2vy=vy^2=vx^3=3$. It follows that $vy=3/2\notin\Z$, that is, the
extension $(K(x,y)|K(x),P)$ is \bfind{ramified}, or in other words,
$(K(x,y),P)$ does not lie in the absolute inertia field of $(K(x),P)$.
We see that we have ramification at the singular point $(0,0)$. As an
exercise, you may check that $(K(x,y),Q)$ lies in the absolute inertia
field of $(K(x),Q)$ whenever $xQ\ne 0$.
\end{example}

%
%
\section{Inertial generation and Abhyankar places}    \label{sectig}
We may now ask ourselves: How could we show that for a given
place $P$ of $F|K$, the valued function field $(F|K,P)$ is
inertially generated?
\begin{example}                             \label{exampig}
Let us start with the most simple case,
where $\trdeg F|K=1$. Assuming that $P$ is not trivial
on $F$ (if it is trivial, then local uniformization is trivial if
$F|K$ is separable), we pick some $z\in F$ such that $zP=0$. As we have
seen in Example~\ref{examp1}, $v_P^{ }K(z)=\Z$. Since $z\notin K$ and
$\trdeg F|K=1$, we know that $F|K(z)$ is algebraic; since $F|K$ is
finitely generated, it follows that $F|K$ is finite. From
Theorem~\ref{fie} we infer that the ramification index $(v_P^{ }F:
v_P^{ }K(z))$ is finite. Therefore, $v_P^{ }F$ is again isomorphic to
$\Z$ and we can pick some $x\in F$ such that $x\in {\cal O}_P$ and
$v_P^{ }F=\Z v_P^{ }x$.

We have achieved that $v_P^{ }F=v_P^{ }K(x)$. If
$\chara FP=\chara K$ is 0, then we know from
Lemma~\ref{Kimaxur} that the absolute inertia field $K(x)^i$ is the
unique maximal extension still having the same value group as $K(x)$. In
this case, we find that $F$ must lie in this absolute inertia field, and
we have proved that $(F|K,P)$ is inertially generated. But we are lost,
it seems, if the characteristic is $p>0$, since in this case, the
absolute inertia field is not necessarily the maximal algebraic
extension of $K(x)$ having the same value group. To solve this case, we
yet have to learn some additional tools.
\end{example}

In this example, the fact that $v_P^{ }F$ was finitely generated played a
crucial role. As we have shown, this is always the case if $\trdeg F|K
=1$. But in general, we can't expect this to hold. We will give
counterexamples in Section~\ref{sectbad}. But prior to the negative, we
want to start with the positive, i.e., criteria for the value group to
be finitely generated.

The following theorem has turned out
in the last years to be amazingly universal in many different
applications of valuation theory. It plays an important role in
algebraic geometry as well as in the model theory of valued fields, in
real algebraic geometry, or in the structure theory of exponential Hardy
fields (= nonarchimedean ordered fields which encode the asymptotic
behaviour of real-valued functions including $\exp$ and $\log$, cf.\
[KK]). For more details and the easy proof of the theorem, see
[V], Theorem 5.5, or [B], Chapter VI, \S10.3, Theorem~1, or [K2].
\begin{theorem}                                \label{prelBour}
Let $(L|K,P)$ be an extension of valued fields. Take $x_i,y_j
\in L$, $i\in I$, $j\in J$, such that the values $v_P^{ }x_i\,$, $i\in I$,
are rationally independent over $v_P^{ }K$, and the residues $y_jP$, $i\in
J$, are algebraically independent over $KP$. Then the elements
$x_i,y_j$, $i\in I$, $j\in J$, are algebraically independent over $K$,
the value of each polynomial in $K[x_i,y_j\mid i\in I,j\in J]$ is
equal to the least of the values of its monomials, and
%
%
\begin{eqnarray}
v_P^{ }K(x_i,y_j\mid i\in I,j\in J) & = & v_P^{ }K\oplus\bigoplus_{i\in I}
\Z v_P^{ }x_i\\
K(x_i,y_j\mid i\in I,j\in J)P & = & KP\,(y_jP\mid j\in J)\;.
\end{eqnarray}
Moreover, the valuation $v_P^{ }$ on $K(x_i,y_j\mid i\in
I,j\in J)$ is uniquely determined by its restriction to $K$, the values
$v_P^{ }x_i$ and the residues $y_jP$.
\end{theorem}

For the proof of the following corollary, see [V] or [K2].
\begin{corollary}                              \label{fingentb}
Let $(L|K,P)$ be an extension of valued fields of finite transcendence
degree. Then
\begin{equation}                            \label{wtdgeq}
\trdeg L|K \>\geq\> \trdeg LP|KP \,+\, \rr (v_P^{ }L/v_P^{ }K)\;.
\end{equation}
If in addition $L|K$ is a function field, and if equality holds in
(\ref{wtdgeq}), then the extensions $v_P^{ }L| v_P^{ }K$ and $LP|KP$
are finitely generated. In particular, if $P$ is trivial on $K$, then
$v_P^{ }L$ is a product of finitely many (namely, $\rr v_P^{ }L$)
copies of $\Z$, and $LP$ is again a function field over $K$.
\end{corollary}
%

\pars
If $P$ is a place of $F|K$, then (\ref{wtdgeq}) reads as follows:
\begin{equation}                            \label{Abhie}
\trdeg F|K \>\geq\> \trdeg FP|K \,+\, \rr v_P^{ }F\;.
\end{equation}
The famous \bfind{Abhyankar inequality} is a generalization of this
inequality to the case of noetherian local rings (see [V]). We call
$P$ an \bfind{Abhyankar place} if equality holds in (\ref{Abhie}).

\pars
The rank of an ordered abelian group is always smaller or equal to its
rational rank. This is seen as follows. If $G_1$ is a subgroup of
$G$, then its divisible hull $\Q \otimes G_1$ lies in the
convex hull of $G_1$ in $\Q\otimes G$. Hence if $G_1$ is
a proper convex subgroup of $G$, then $\Q\otimes G_1$ is a
proper convex subgroup of $\Q\otimes G$ and thus, $\dim_\Q
\Q\otimes G_1<\dim_\Q \Q\otimes G$. It follows that if
$\{0\}=G_0\subsetuneq G_1\subsetuneq\ldots \subsetuneq
G_n= G$ is a chain of convex subgroups of $G$, then
$\rr G=\dim_\Q \Q\otimes G\geq n$. In view of (\ref{Abhie}),
this proves that the rank of a place $P$ of a function field $F|K$
cannot exceed $\trdeg F|K$ and thus is finite. We say that $P$ is of
\bfind{maximal rank} if the rank is equal to $\trdeg F|K$.

If $\trdeg F|K=1$, then every place $P$ of $F|K$ is an Abhyankar place.
It is of maximal rank if and only if it is non-trivial. Indeed, if
$v_P^{ }F$ is not trivial, then $\rr v_P^{ }F\geq 1$, and it follows from
(\ref{Abhie}) that $\trdeg F|K=1= \rr v_P^{ }F$. Then also the rank is 1
since a group of rational rank 1 is a non-trivial subgroup of $\Q$. If
on the other hand $v_P^{ }F$ is trivial, then $P$ is an isomorphism on $F$
so that $\trdeg F|K=1=\trdeg FP|K$.

\parb
Using Corollary~\ref{fingentb}, we can now generalize our construction
given in Example~\ref{exampig}. Let $P$ be an arbitrary place of $F|K$.
We set $\rho=\rr v_P^{ }F$ and $\tau=\trdeg FP|K$. We take
elements $x_1,\ldots,x_{\rho}\in F$ such that $v_P^{ }x_1,\ldots,v_P^{ }
x_{\rho}$ are rationally independent elements in $v_P^{ } F$. Further,
we take elements $y_1,\ldots, y_{\tau}\in F$ such that $y_1P,\ldots,
y_{\tau} P$ are algebraically independent over $K$. Then by
Theorem~\ref{prelBour}, $x_1,\ldots,x_{\rho},y_1,\ldots,y_{\tau}$ are
algebraically independent over $K$. The restriction of $P$ to
$K(x_1,\ldots,x_{\rho},y_1,\ldots,y_{\tau})$ is an Abhyankar place.
We fix this situation for later use. We call
\begin{equation}                            \label{Asff}
\left.\begin{array}{l}
F_0\>:=\>K(x_1,\ldots,x_{\rho},y_1,\ldots,y_{\tau})\mbox{ with}\\
\rho=\rr v_P^{ }F\mbox{\ \ and\ \ }\tau=\trdeg FP|K,\\
v_P^{ }x_1,\ldots,v_P^{ } x_{\rho}\mbox{ rationally independent in
$v_P^{ } F$, and}\\
y_1P,\ldots,y_{\tau} P\mbox{ algebraically independent over $K$}
\end{array}\right\}
\end{equation}
an \bfind{Abhyankar subfunction field} of $(F|K,P)$.

Now let us assume in addition that $P$ is an Abhyankar place of $F|K$.
That is, $\rho+\tau=\trdeg F|K$. It follows that $x_1,\ldots,x_{\rho},
y_1,\ldots, y_{\tau}$ is a transcendence basis of $F|K$. We refine our
choice of these elements as follows. From Corollary~\ref{fingentb} we
know that $v_P^{ }F$ is product of $\rho$ copies of $\Z$. So we can choose
$x_1,\ldots,x_{\rho}\in {\cal O}_P$ in such a way that $v_P^{ }F=\Z
v_P^{ }x_1 \oplus\ldots\oplus\Z v_P^{ }x_{\rho}$, which implies that
$v_P^{ }F=v_P^{ }K(x_1, \ldots,x_{\rho})$. From Corollary~\ref{fingentb}
we also know that $FP|K$ is finitely generated. We shall also assume
that $FP|K$ is separable. Then it follows that there is a separating
transcendence basis for $FP|K$. We choose $y_1,\ldots,y_{\tau}\in
{\cal O}_P$ in such a way that $y_1P,\ldots,y_{\tau}P$ is such a
separating transcendence basis. Now we can choose some $a\in FP$ such
that $FP=K(y_1P, \ldots, y_\tau P,a)$. We take a monic polynomial $f$
with coefficients in the valuation ring of $(F_0,P)$ such that its
reduction $fv_P^{ }$ is the minimal polynomial of $a$ over $F_0P=
K(y_1P,\ldots, y_\tau P)$. Since $a\in FP$ is separable-algebraic over
$K(y_1P, \ldots, y_\tau P)$, by Hensel's Lemma (Simple Root Version)
there exists a root $\eta$ of $f$ in the henselization of $(F,P)$ such
that $\eta P=a$. Take $\sigma\in \Aut(\tilde{F}_0|F_0)$ such that
$v(\sigma x-x)>0$ for all $x$ in the valuation ring of $P$ on
$\tilde{F}$. Then in particular, $v(\sigma\eta-\eta)>0$. But if
$\sigma\eta\ne\eta$, then it follows from $\deg (f)=\deg (fv_P^{ })$
that $(\sigma\eta)P\ne\eta P$, i.e., $v(\sigma\eta-\eta)=0$. Hence,
$\sigma\eta=\eta$, which shows that $\eta$ lies in the absolute inertia
field of $F_0\,$.

Now the field $F_0(\eta)$ has the same value group and residue field as
$F$, and it is contained in the henselization of $F\,$. Hence by
Theorem~\ref{immhens},
\begin{equation}                            \label{extens}
(F^h|F_0(\eta)^h,P)
\end{equation}
is an immediate algebraic extension. As $\eta$ lies in the absolute
inertia field of $F_0$ and this field is henselian, we have that
$F_0(\eta)^h$ is a subfield of this absolute inertia field. If we
could show that $F^h=F_0(\eta)^h$, then $F$ itself would lie in this
absolute inertia field, which would prove that $(F|K,P)$ is inertially
generated. If the residue characteristic $\chara FP=\chara K$ is 0, then
again Lemma~\ref{Kimaxur} tells us that the absolute inertia field of
$(F_0,P)$ is the unique maximal extension having the same value group as
$F_0\,$; so $F^h$ must be a subfield of it. Hence in characteristic 0 we
have now shown that $(F|K,P)$ is inertially generated. But what happens
in positive characteristic? Can the extension (\ref{extens}) be
non-trivial? To answer this question, we have to take a closer look at
the main problem of valuation theory in positive characteristic.

%
%
\section{The defect}                        \label{sectdef}

Assume that $(K,v)$ is henselian and $(L|K,v)$ is a finite extension of
degree $n$. Then we have to deal only with a single ramification index e
and a single inertia degree f. Hence, the fundamental inequality now
reads as
\begin{equation}
n\>\geq\> {\rm e}\,{\rm f}\;.
\end{equation}
If $L$ is contained in $K^i$ then by Theorem~\ref{T}, $n={\rm f}$.
If $K=K^i$ and $L$ is contained in $K^r$, then by Theorem~\ref{V},
$n={\rm e}$. Putting these observations together (using
Theorems~\ref{liftZTV} and~\ref{downZTV} and the fact that
extension degree, ramification index and inertia degree are
multiplicative), one finds:
\begin{lemma}
If $(K,v)$ is henselian and $L|K$ is a finite subextension of $K^r|K$,
then it satisfies the \bfind{fundamental equality}
\begin{equation}
n={\rm e}\,{\rm f}\;.
\end{equation}
\end{lemma}
Hence, an inequality can only result from some part of the extension
which lies beyond the absolute ramification field. So Theorem~\ref{S}
shows that the missing factor can only be a power of $p$. In this way,
one proves the important \bfind{Lemma of Ostrowski}:
\begin{theorem}                             \label{LofO}
Set $p:=\chara Kv$ if this is positive, and $p:=1$ if $\chara Kv=0$.
If $(K,v)$ is henselian and $L|K$ is of degree $n$, then
\begin{equation}
n\>=\>{\rm d}\,{\rm e}\,{\rm f}\;,
\end{equation}
where {\rm d} is a power of $p$. In particular, if $\chara Kv=0$,
then we always have the fundamental equality $n={\rm e}\,{\rm f}$.
\end{theorem}

The integer d$\>\geq 1$ is called the \bfind{defect} of the extension
$(L|K,v)$. This can also be taken as a definition for the defect if
$(K,v)$ is not henselian, but the extension of $v$ from $K$ to $L$ is
unique. We note:
\begin{corollary}                           \label{hidef}
If $(K,v)$ is henselian and $(L|K,v)$ is a finite immediate extension,
then the defect of $(L|K,v)$ is equal to $[L:K]$.
\end{corollary}

A henselian field is called \bfind{defectless} if every finite
extension has trivial defect d$\>=1$. In rigid analysis, this is also
called \bfind{stable}. A not necessarily henselian field is
called defectless if for every finite extension of it, equality holds in
the fundamental inequality (\ref{fundineq}) (if the field is henselian,
this coincides with our first definition). A proof of the next
theorem can be found in [K2] and, partially, also in [E].
\begin{theorem}                             \label{dlfhens}
A valued field is a defectless field if and only if its henselization
is.
\end{theorem}

We also note the following fact, which is easy to prove:
\begin{lemma}                               \label{dlfinext}
Every finite extension field of a defectless field is again a
defectless field.
\end{lemma}

The following are examples of defectless fields:
\sn
{\bf (DF1)} \ All valued fields with residue characteristic 0. This is a
direct consequence of the Lemma of Ostrowski.
\sn
{\bf (DF2)} \ Every discretely valued field of characteristic 0. An easy
argument shows that every finite extension with non-trivial defect of a
discretely valued field must be inseparable. In particular, the field
$(\Q_p,v_p)$ of $p$-adic numbers with its $p$-adic valuation, and all of
its subfields, are defectless fields.
\sn
{\bf (DF3)} \ All maximal fields (and hence also all power series
fields, see Section~\ref{sectmax}) are defectless fields. For the proof,
see [R1] or [K2].

\parb
We have seen that the extensions beyond the absolute ramification
field are responsible for non-trivial defects. To get this picture,
we have chosen a modified approach to ramification theory (cf.\ our
remark in Section~\ref{sectram}). We have shifted the purely inseparable
extensions to the top (cf.\ our table). In fact, that is where the
purely inseparable extensions belong, because from the ramification
theoretical point of view, they can be nasty, and in this respect, they
have much in common with the extension S$|$V.

\parm
The defect, appearing only for positive residue characteristic, is
essentially the cause of the problems that we have in algebraic geometry
as well as in the model theory of valued fields, in positive
characteristic. Therefore, it is very important that you get a feeling
for what the defect is. Let us look at three main examples. The
first one is the most basic and was probably already known to most of
the early valuation theorists. But it seems reasonable to attribute it
to F.~K.~Schmidt.
\begin{example}
We consider the power series field $K((t))$ with its $t$-adic valuation
$v=v_t$. We have already remarked in Example~\ref{examp2} that
$K((t))|K(t)$ is transcendental. So we can choose an element $s\in
K((t))$ which is transcendental over $K(t)$. Since $vK((t))=\Z=vK(t)$
and $K((t))v=K=K(t)v$, the extension $(K((t))|K(t),v)$ is immediate.
The same must be true for the subextension $(K(t,s)|K(t),v)$ and thus
also for $(K(t,s)|K(t,s^p),v)$. The latter extension is purely
inseparable of degree $p$ (since $s,t$ are algebraically independent
over $K\,$, the extension $K(s)|K(s^p)$ is linearly disjoint from
$K(t,s^p)|K(s^p)\,$). Hence, Corollary~\ref{cae} shows that there is
only one extension of the valuation $v$ from $K(t,s^p)$ to $K(t,s)$.
Consequently, its defect is $p$.
\end{example}

To give an example of a henselian field which is not defectless, we
build on the foregoing example.
\begin{example}                            \label{exampdef2}
By Theorem~\ref{hensuniqemb}, there is a henselization $(K(t,s),v)^h$
of the field $(K(t,s),v)$ in the henselian field $K((t))$ and a
henselization $(K(t,s^p),v)^h$ of the field $(K(t,s^p),v)$ in
$(K(t,s),v)^h$. We find that the extension $K(t,s)^h|K(t,s^p)^h$
is again purely inseparable of
degree $p$. Indeed, $K(t,s)| K(t, s^p)$ is linearly disjoint from the
separable extension $K(t,s^p)^h|K(t,s^p)$, and by virtue of
Corollary~\ref{hensfcomp}, $K(t,s)^h= K(t,s).K(t,s^p)^h$. Also for this
extension, we have that ${\rm e}={\rm f}=1$ and again, the defect is
$p$. Note that by what we have said earlier, an extension of degree $p$
with non-trivial defect over a discretely valued field like
$(K(t,s^p),v)^h$ can only be purely inseparable.
\end{example}

\pars
Now we will give an example of a finite {\it separable} extension
with non-trivial defect. It seems to be the generic example for
our purposes since its importance is also known in algebraic
geometry.
\begin{example}                             \label{exampASd1}
Take an arbitrary field $K$ of characteristic $p>0$, and $t$ to be
transcendental over $K$. On $K(t)$ we take the $t$-adic valuation
$v=v_t\,$. We set $L:=K(t^{1/p^i}\mid i\in\N)$.
This is a purely inseparable extension of $K(t)$; if $K$ is perfect,
then it is the perfect hull of $K(t)$. By Corollary~\ref{cae}, $v$
has a unique extension to $L$. We set $L_k:=K(t^{1/p^k})$ for every
$k\in\N$; so $L=\bigcup_{k\in\N} L_k\,$. We observe that
$1/p^k=vt^{1/p^k}\in vL_k\,$, so $(vL_k:vK(t))\geq p^k$. Now the
fundamental inequality shows that $(vL_k:vK(t))=p^k$ and that
$L_kv=K(t)v=K$. The former shows that $vL_k=\frac{1}{p^k}\Z$. We obtain
that $vL=\bigcup_{k\in\N}^{} vL_k=\frac{1}{p^\infty}\Z$ and that
$Lv=\bigcup_{k\in\N}^{} L_kv=K$.
We consider the extension $L(\vartheta)|L$ generated by a root
$\vartheta$ of the Artin--Schreier--polynomial $X^p-X-\frac{1}{t}$.
We set
\begin{equation}
\vartheta_k\>:=\>\sum_{i=1}^{k} t^{-1/p^i}
\end{equation}
and compute
\begin{eqnarray*}
\vartheta_k^p-\vartheta_k-\frac{1}{t} & = & \sum_{i=1}^{k}
t^{-1/p^{i-1}}\,-\,\sum_{i=1}^{k} t^{-1/p^i}\,-\,t^{-1}\\
& = & \sum_{i=0}^{k-1} t^{-1/p^{i}}\,-\,\sum_{i=1}^{k}
t^{-1/p^i}\,-\,t^{-1} \>=\>-t^{-1/p^k}\;.
\end{eqnarray*}
Therefore,
\begin{eqnarray*}
(\vartheta-\vartheta_k)^p\,-\,(\vartheta-\vartheta_k) & = &
\vartheta^p-\vartheta-\frac{1}{t}\,-\,\left(\vartheta_k^p-\vartheta_k-
\frac{1}{t}\right)\\
 & = & 0+t^{-1/p^k}\>=\>t^{-1/p^k}\;.
\end{eqnarray*}
If we have the equation $b^p-b=c$ and know that $vc<0$, then
we can conclude that $vb<0$ since otherwise, $v(b^p-b)\geq 0>vc$, a
contradiction. But since $vb<0$, we have that $vb^p=pvb<vb$, which
implies that $vc=v(b^p-b)=pvb$. Consequently, $vb=\frac{vc}{p}$. In our
case, we obtain that
\begin{equation}
v(\vartheta-\vartheta_k)\>=\>\frac{vt^{-1/p^k}}{p}\>=\>
-\frac{1}{p^{k+1}}\;.
\end{equation}
We see that $1/p^{k+1}\in vL_k(\vartheta)$, so that $(vL_k(\vartheta):
vL_k)\geq p$. Since $[L_k(\vartheta):L_k]\leq p$, the fundamental
inequality shows that $[L_k(\vartheta):L_k]=p$, $vL_k(\vartheta)=
1/p^{k+1}\Z$, $L_k(\vartheta)v=L_kv=K$ and that the extension of the
valuation $v$ from $L_k$ to $L_k(\vartheta)$ is unique. As
$L(\vartheta)= \bigcup_{k\in\N}^{}L_k(\vartheta)$, we obtain that
$vL(\vartheta)= \bigcup_{k\in\N} 1/p^{k+1}\Z=\frac{1}{p^\infty}\Z=vL$
and that $L(\vartheta)v=\bigcup_{k\in\N} L_kv=K=Lv$. We have thus proved
that the extension $(L(\vartheta)|L,v)$ is immediate. Since $\vartheta$
has degree $p$ over every $L_k\,$, it must also have degree $p$ over
their union $L$. Further, as the extension of $v$ from $L_k$ to
$L_k(\vartheta)$ is unique for every $k$, also the extension of $v$ from
$L$ to $L(\vartheta)$ is unique. So we have $n=p$, ${\rm e}= {\rm f}=
{\rm g}=1$, and we find that the defect of $(L(\vartheta)|L,v)$ is $p$.

To obtain a defect extension of a henselian field, we show that $L$ can
be replaced by its henselization. By Theorem~\ref{hensfcomp},
$L^h(\vartheta)=L^h.L(\vartheta)=
L(\vartheta)^h$. By Theorem~\ref{immhens}, $vL(\vartheta)^h=vL
(\vartheta)= vL=vL^h$ and $L(\vartheta)^hv=L(\vartheta)v=Lv=L^hv$.
Hence, also the extension $(L^h(\vartheta)|L^h,v)$ is immediate. We only
have to show that it is of degree $p$. This follows from the general
valuation theoretical fact that if an extension $L'|L$ admits a unique
extension of the valuation $v$ from $L$ to $L'$, then $L'|L$ is linearly
disjoint from $L^h|L$. But we can also give a direct proof. Again by
Theorem~\ref{hensfcomp}, $L_k^h(\vartheta)=L_k(\vartheta)^h$, and by
Theorem~\ref{immhens}, $vL_k(\vartheta)^h=vL_k (\vartheta)$ and $vL_k^h
=vL_k\,$. Therefore, $(vL_k^h(\vartheta):vL_k^h)=(vL_k(\vartheta):
vL_k)=p$, showing that $[L_k^h(\vartheta):L_k^h]=p$ for every $k$.
Again by Theorem~\ref{hensfcomp}, $L^h=L.L_1^h=(\bigcup_{k\in\N} L_k).
L_1^h=\bigcup_{k\in\N} L_k.L_1^h=\bigcup_{k\in\N} L_k^h$. By the same
argument as before, it follows that $[L^h(\vartheta):L^h]=p$.

Hence, we have found an immediate Artin--Schreier extension of
degree $p$ and defect $p$ of a henselian field which is only of
transcendence degree 1 over $K$.
\end{example}

A valued field $(K,v)$ is called \bfind{algebraically maximal} if it
admits no proper immediate algebraic extension, and it is called
\bfind{separable-algebraically maximal} if it admits no proper immediate
separable-algebraic extension. Since the henselization is an immediate
separable-algebraic extension by Theorem~\ref{immhens}, every
separable-algebraically maximal field is henselian. The converse is not
true, since the field $(L^h,v)$ of our foregoing example is henselian
but not separable-algebraically maximal. Corollary~\ref{hidef} shows
that every henselian defectless field is algebraically maximal.
The converse is not true, as was shown by Fran\c{c}oise Delon [DEL]
(cf.\ also [K2]).

\begin{example}                             \label{exampASd2}
In the foregoing example, we may replace $K(t)$ by $K((t))$, taking
$L$ to be the field $K((t)) (t^{1/p^k}\mid k\in\N)$. It is not hard
to show (by splitting up the power series in a suitable way) that
$K((t^{1/p^k}))=K((t))[t^{1/p^k}]$, which is algebraic over $K((t))$.
Hence, $L=\bigcup_{k\in\N} K((t^{1/p^k}))$, a union of an ascending
chain of power series fields. By Lemma~\ref{unihens} below, $(L,v)$
is henselian, and $(L(\vartheta)|L,v)$ gives an instant example of an
immediate extension of a henselian field. But this $L$ is ``very
large'': it is of infinite transcendence degree over $K$. On the other
hand, this version of our example shows that an infinite algebraic
extension of a maximal field (or a union over an ascending chain of
maximal fields) is in general not even defectless (and hence also not
maximal). The example can also be transformed into the $p$-adic
situation, showing that there are infinite extensions of
$(\Q_p,v_p)$ which are not defectless fields (cf.\ [K2]).
\end{example}

%
%
\section{Maximal immediate extensions}      \label{sectmax}
Based on our examples, we can observe another obstruction in positive
characteristic. In many applications of valuation theory, one is
interested in the embedding of a given valued field in a power series
field which, if possible, should have the same value group and residue
field. (We give an example relevant for algebraic geometry in the next
section.) Then this power series field would be an immediate extension
of our field, and since every power series field is maximal, it would be
a \bfind{maximal immediate extension} of our field. So we see that we
are led to the problem of determining maximal immediate extensions, in
particular, whether maximal immediate extensions of a given valued field
are unique up to valuation preserving isomorphism. It was shown by
Wolfgang Krull [KR] that maximal immediate extensions exist for every
valued field. The proof uses Zorn's Lemma in combination with an upper
bound for the cardinality of valued fields with prescribed value group
and residue field. Krull's deduction of this upper bound is hard to
read; later, Kenneth A.\ H.\ Gravett [GRA] gave a nice and simple
proof.

The uniqueness problem for maximal immediate extensions was considered
by Irving Kaplansky in his important paper [KA1]. He showed that if the
so-called \bfind{hypothesis A} holds, then the field has a unique
maximal immediate extension (up to valuation preserving isomorphism).
For a Galois theoretic interpretation of hypothesis A and more
information about the uniqueness problem, see [KPR]. Let us mention a
problem which was only partially solved in [KPR] and in [WA1]:
\mn
{\bf Open Problem 2:} \ If a valued field does not satisfy
Kaplansky's hypothesis A, does it then admit two non-isomorphic
maximal immediate extensions?
\mn

We can give a quick example of a valued field with two non-isomorphic
maximal immediate extensions.
\begin{example}                             \label{exampmie}
In the setting of Example~\ref{exampASd2},
suppose that $K$ is not \bfind{Artin--Schrei\-er closed},
that is, there is an element $c\in K$ such that $X^p-X-c$ is
irreducible over $K$. Take $\vartheta_c$ to be a root of
$X^p-X-(\frac{1}{t}+c)$; note that $v(\frac{1}{t}+c)=v\frac{1}{t}<0$
since $vc=0>v\frac{1}{t}$. Then in exactly the same way as for
$\vartheta$, one shows that the extension $(L(\vartheta_c)|L,v)$
is immediate of degree $p$ and defect $p$. So we have two distinct
immediate extensions of $L$. We take $(M_1,v)$ to be a maximal immediate
extension of $L(\vartheta)$, and $(M_2,v)$ to be a maximal immediate
extension of $L(\vartheta_c)$. Then $(M_1,v)$ and $(M_2,v)$ are also
maximal immediate extensions of $(L,v)$. If they were isomorphic over
$L$, then $M_1$ would also contain a root of $X^p-X-(\frac{1}{t}+c)$;
w.l.o.g., we can assume that it is the one called $\vartheta_c\,$.
Now we compute:
\[(\vartheta_c-\vartheta)^p - (\vartheta_c-\vartheta)\>=\>
\vartheta_c^p-\vartheta_c - (\vartheta^p-\vartheta)\>=\>
\frac{1}{t}+c -\frac{1}{t}\>=\> c\;.\]
Since $vc=0$, we also have that $v(\vartheta_c-\vartheta)=0$ (you may
prove this along the lines of an argument given earlier). Applying
the residue map to $\vartheta_c-\vartheta$, we thus obtain a root
of $X^p-X-c$. But by our assumption, this root is not contained in
$K=Lv$. Consequently, $M_1v\ne Lv$, contradicting the fact that
$(M_1,v)$ was an immediate extension of $(L,v)$. This proves that
$(M_1,v)$ and $(M_2,v)$ cannot be isomorphic over $(L,v)$.

We have used that $K$ is not Artin--Schreier closed. And in fact,
one of the consequences of hypothesis A for a valued field $(L,v)$ is
that its residue field be Artin--Schreier closed (see
Section~\ref{secttw}).
\end{example}

We will need a generalization of the field of formal Laurent series,
called \bfind{(generalized) power series field}. Take any field
$K$ and any ordered abelian group $G$. We take $K((G))$ to be the set
of all maps $\mu$ from $G$ to $K$ with well-ordered \bfind{support}
$\{g\in G\mid \mu(g)\ne 0\}$. You can visualize the elements of
$K((G))$ as formal power series $\sum_{g\in G}^{} c_g t^g$ for
which the support $\{g\in G\mid c_g\ne 0\}$ is well-ordered. Using
this condition one shows that $K((G))$ is a field in a similar way as it
is done for $K((t))$. Also, one uses it to introduce the valuation:
\begin{equation}
v\sum_{g\in G}^{} c_g t^g\>=\>\min\{g\in G\mid c_g\ne 0\}
\end{equation}
(the minimum exists because the support is well-ordered). This
valuation is often called the \bfind{canonical valuation of $K((G))$},
and sometimes called the \bfind{minimum support valuation}.
With this valuation, $K((G))$ is a maximal field.

\pars
The fields $L$ constructed in Examples~\ref{exampASd1}
and~\ref{exampASd2} are subfields of $K((\Q))$ in a canonical way. It is
interesting to note that the element $\vartheta$ is an element of
$K((\Q))$:
\begin{equation}                            \label{expan}
\vartheta\>=\>\sum_{i\in\N} t^{-1/p^i}\>=\>
t^{-1/p}+t^{-1/p^2}+\ldots+t^{-1/p^i}+\ldots\;.
\end{equation}
Indeed,
\begin{eqnarray*}
\vartheta^p-\vartheta-\frac{1}{t} & = & \sum_{i\in\N} t^{-1/p^{i-1}}
\,-\,\sum_{i\in\N} t^{-1/p^i}\,-\,t^{-1}\\
& = & \sum_{i=0}^{\infty} t^{-1/p^i}
\,-\,\sum_{i=1}^{\infty}t^{-1/p^i}\,-\,t^{-1}\>=\>0\;.
\end{eqnarray*}
Note that the values $vt^{-1/p^n}=-1/p^n$ converge from below to 0.
Therefore, $\vartheta$ does not even lie in the completion of $L$. In
fact, there cannot be a root of $X^p-X-1/t$ in the completion; if $a$
would be such a root, then there would be some $b\in L$ such that
$v(a-b)>0$. We would have that
\begin{equation}                            \label{ASsurg}
(a-b)^p-(a-b)\>=\>a^p-a -(b^p-b)\>=\>\frac{1}{t}-(b^p-b)\;.
\end{equation}
Because of $v(a-b)>0$, the left hand side and consequently also
the right hand side has value $>0$. But as we have seen in
Example~\ref{exampHL1}, the polynomial $X^p-X-c$ splits over every
henselian field containing $c$ if $vc>0$. Hence, in the cases where $L$
is henselian, there exists a root $a'\in L$ of $X^p-X-1/t+b^p-b$.
It follows that $(a'+b)^p-(a'+b)-1/t=1/t-(b^p-b)+b^p-b-1/t=0$.
As $a'+b\in L$, this would imply that $X^p-X-1/t$ splits over $L$,
a contradiction.

\parb
Let us illustrate the influence of the defect by considering an object
which is well-known in algebraic geometry.

%
%
\section{A quick look at Puiseux series fields}
Recall that $K((\Q))$ is the field of all formal sums $\sum_{q\in\Q} c_q
t^q$ with $c_q\in K$ and well-ordered support. The subset
\begin{equation}                            \label{PSF}
{\rm P}(K)\>:=\>\left\{\sum_{i=n}^{\infty} c_i t^{i/k}\mid c_i\in K,\,
n\in\Z,\, k\in\N\right\}\>=\>\bigcup_{k\in\N}^{} K((t^{1/k}))\>\subset\>
K((\Q))
\end{equation}
is itself a field, called the \bfind{Puiseux series field over $K$}.
Here, the valuation $v$ on $K((t^{1/k}))$ is again the minimum support
valuation, in particular, we have that $vt^{1/k}=1/k$. In this way, the
valuation $v$ on every $K((t^{1/k}))$ is an extension of the $t$-adic
valuation $v_t$ of $K((t))$ and of the valuation of every subfield
$K((t^{1/m}))$ where $m$ divides $k$.
%
%

${\rm P}(K)$ can also be written as a union of an ascending chain of
power series fields in the following way. We take $p_i$ to be the $i$-th
prime number and set $m_k:=\prod_{i=1}^{k}p_i^k$. Then $m_k|m_{k+1}$ and
thus $K((t^{1/m_k}))\subset K((t^{1/m_{k+1}}))$ for every $k\in\N$,
and every natural number will divide $m_k$ for large enough $k$.
Therefore,
\begin{equation}
{\rm P}(K)\>=\>\bigcup_{k\in\N}^{} K((t^{1/m_k}))\;.
\end{equation}

If one does not want to work in the power series field $K((\Q))$, then
one simply has to choose a compatible system of $k$-th roots $t^{1/k}$
of $t$ (that is, for $k=\ell m$ we must have $(t^{1/k})^\ell=t^{1/m}$;
this is automatic for the elements $t^{1/k}\in K((\Q))$ by definition of
the multiplication in this field). Then (\ref{PSF}) can serve as a
definition for the Puiseux series field over $K$.

\begin{lemma}                             \label{Psf}
The Puiseux series field P$(K)$ is an algebraic extension of $K((t))$,
and it is henselian with respect to its canonical valuation $v$. Its
residue field is $K$ and its value group is $\Q$.
\end{lemma}
\begin{proof}
For every $k\in \N$, the element $t^{1/k}$ is algebraic over $K((t))$.
Similarly as in Example~\ref{exampASd2}, we have that $K((t^{1/k}))=
K((t)) [t^{1/k}]$, which is algebraic over $K((t))$. Consequently, also
the union P$(K)$ of the $K((t^{1/k}))$ is algebraic over $K((t))$. By
Corollary~\ref{maxhens}, $K((t))$ is henselian w.r.t.\ its canonical
valuation $v_t\,$. As the canonical valuation $v$ of P$(K)$ is an
extension of $v_t\,$, Corollary~\ref{algexth} yields that
$(\mbox{P}(K),v)$ is henselian.

The value group of every $(K((t^{1/k})),v)$ is $\Z vt^{1/k}=\Z
\frac{vt}{k}=\frac{1}{k}\Z$, so the union over all $K((t^{1/k}))$ has
value group $\bigcup_{k\in\N}\frac{1}{k}\Z=\Q$. The residue field
of every $(K((t^{1/k})),v)$ is $K$, hence also the residue field of
their union is $K$.
\end{proof}

\begin{theorem}
Let $K$ be a field of characteristic $0$. Then $({\rm P}(K),v)$ is a
defectless field. Further, P$(K)$ is the algebraic closure of
$K((t))$ if and only if $K$ is algebraically closed.
\end{theorem}
\begin{proof}
The residue field of $({\rm P}(K),v)$ is $K$, hence if $\chara K=0$,
then $({\rm P}(K),v)$ is a defectless field by {\bf (DF1)} in
Section~\ref{sectdef}.

For the second assertion, we use the following valuation
theoretical fact (try to prove it, it is not hard):
\sn
{\it Let $(L,v)$ be a valued field and choose any extension of $v$ to
the algebraic closure $\tilde{L}$. Then $v\tilde{L}$ is the divisible
hull of $vL$, and $\tilde{L}v$ is the algebraic closure of $Lv$.}
\sn
Hence, $v\widetilde{K((t))}=\Q=v{\rm P}(K)$ and $\widetilde{K((t))}P
=\tilde{K}$. Thus if $\widetilde{K((t))}={\rm P}(K)$, then
$\tilde{K}={\rm P}(K)=K$, which shows that $K$ must be algebraically
closed. For the converse, note that by the foregoing lemma, ${\rm P}
(K) \subseteq \widetilde{K((t))}$. Assume that $\tilde{K}=K$. Then the
extension $(\widetilde{K((t))}|{\rm P}(K),v)$ is immediate. But since
$({\rm P}(K), v)$ is henselian (by the foregoing lemma) and defectless,
every finite subextension must be trivial by Theorem~\ref{LofO}. This
proves that $\widetilde{K((t))}= {\rm P}(K)$, i.e., ${\rm P}(K)$ is
algebraically closed.
\end{proof}

The assertion of this theorem does not hold if $K$ has positive
characteristic:

\begin{example}                             \label{exampASd3}
In Example~\ref{exampASd1}, we can replace $L_k$ by $K((t^{1/k}))$
for every $k\in\N\>$ (as opposed to $K((t^{1/p^k}))$, which we used
in Example~\ref{exampASd2}). Still, everything works the same,
producing the henselian Puiseux series field $L={\rm P}(K)$ with an
immediate Artin--Schreier extension $(L(\vartheta)|L,v)$ of degree $p$
and defect $p$.

By construction, ${\rm P}(K)$ is a subfield of $K((\Q))$. Hence, the
arguments at the end of the last section show that there is no root of
$X^p-X-1/t$ in the completion of ${\rm P}(K)$. The arguments of
Example~\ref{exampmie} show that ${\rm P}(K)$ has non-isomorphic
maximal immediate extensions if $K$ is not Artin--Schreier closed.
\end{example}

Our example proves:
\begin{theorem}
Let $K$ be a field of characteristic $p>0$. Then $({\rm P}(K),v)$ is
not defectless. In particular, ${\rm P}(K)$ is not algebraically
closed, even if $K$ is algebraically closed. Not even the completion
of ${\rm P}(K)$ is algebraically closed.
\end{theorem}

There is always a henselian defectless field extending $K((t))$ and
having residue field $K$ and divisible value group, even if $K$ has
positive characteristic. We just have to take the power series field
$K((\Q))$. But in contrast to the Puiseux series field, this field is
``very large'': it has uncountable transcendence degree over $K((t))$.
Nevertheless, having serious problems with the Puiseux series field in
positive characteristic, we tend to replace it by $K((\Q))$. But
this seems problematic since it might not be the unique maximal
immediate extension of the Puiseux series field. However, if
$K$ is perfect and does not admit a finite extension whose degree
is divisible by $p$ (and in particular if $K$ is algebraically closed),
then Kaplansky's uniqueness result shows that the maximal immediate
extension is unique. On the other hand, our example shows that the
assumption ``$K$ is perfect'' alone is not sufficient, since there are
perfect fields which are not Artin--Schreier closed.

%
%
\section{The tame and the wild valuation theory}         \label{secttw}
Before we carry on, let us describe some advanced ramification theory
based on the material of Sections~\ref{sectram} and~\ref{sectdef}.
Throughout, we let $(K,v)$ be a henselian non-trivially valued field. We
set $p=\chara Kv$ if this is positive, and $p=1$ otherwise. If $(L|K,v)$
is an algebraic extension, then we call $(L|K,v)$ a \bfind{tame
extension} if for every finite subextension $L'|K$ of $L|K$,
\sn
1) \ $(vL':vK)$ is not divisible by the residue characteristic
$\chara Kv$,\n
2) \ $L'v|Kv$ is separable,\n
3) \ $[L':K]=(vL':vK)[L'v:Kv]$, i.e., $(L'|K,v)$ has trivial defect.
\sn
From the ramification theoretical facts presented in
Section~\ref{sectram}, one derives:
\begin{theorem}
If $(K,v)$ is henselian, then its absolute ramification field $(K^r,v)$
is the unique maximal tame extension of $(K,v)$, and its absolute
inertia field $(K^i,v)$ is the unique maximal tame extension of $(K,v)$
having the same value group as $K$.
\end{theorem}

An extension $(L|K,v)$ is called \bfind{purely wild} if $L|K$ is
linearly disjoint from $K^r|K$. An ordered group $G$ is called
$p$-divisible if for every $\alpha\in G$ and $n\in\N$ there is
$\beta\in G$ such that $p^n\beta=\alpha$. The \bfind{$p$-divisible hull}
of $G$, denoted by $\frac{1}{p^\infty} G$, is the smallest subgroup
of the divisible hull $\Q\otimes G$ which contains $G$ and is
$p$-divisible; it can be written as $\{\alpha/p^n\mid \alpha\in G\,,\,
n\in\N\}$. The following was proved by Matthias Pank (cf.\ [KPR]):
\begin{theorem}                             \label{rfmte}
If $(K,v)$ is henselian, then there exists a field complement $W$ to
$K^r$ over $K$, that is, $W.K^r= \tilde{K}$ and $W\cap K^r=K$. The
degree of every finite subextension of $W|K$ is a power of $p$. Further,
$vW$ is the $p$-divisible hull $\frac{1}{p^\infty} vK$ of $vK$, and $Wv$
is the perfect hull of $Kv$.
\end{theorem}
So $(W,v)$ is a maximal purely wild extension of $(K,v)$. It was shown
by Pank and is shown in [KPR] via Galois theory that $W$ is unique up to
isomorphism over $K$ if $Kv$ does not admit finite separable extensions
whose degree are divisible by $p$. On the other hand, if $vK$ is
$p$-divisible and $Kv$ is perfect, then $(W|K,v)$ is an immediate
extension, and since every subextension of $K^r|K$ has trivial defect,
it follows that the field complements $W$ of $K^r$ over $K$ are
precisely the maximal immediate algebraic extensions of $(K,v)$.

It was shown by George Whaples [WH2] and by Francoise Delon [D] that
Kaplansky's original hypothesis A consists of the following three
conditions:
\sn
1) \ $Kv$ does not admit finite separable extensions whose degree are
divisible by $p$,\n
2) \ $vK$ is $p$-divisible,\n
3) \ $Kv$ is perfect.
\sn
So if $(K,v)$ satisfies Kaplansky's hypothesis A, then it follows from
what we said above that the maximal immediate algebraic extensions of
$(K,v)$ are unique up to isomorphism over $K$. But this is the kernel of
the uniqueness problem for the maximal immediate extensions: using
Theorem~2 of [KA1], one can easily show that the maximal immediate
extensions are unique as soon as the maximal immediate algebraic
extensions are.

\pars
Since all finite tame extensions have trivial defect, the defect
is located in the purely wild extensions $(W|K,v)$. So we are interested
in their structure. Here is one amazing result, due to Florian Pop
(for the proof, see [K2], and for the notion of ``additive
polynomial'', see Example~\ref{exampHL1}):
\begin{theorem}
Let $(L|K,v)$ be a minimal purely wild extension, i.e., there is no
subextension $L'|K$ of $L|K$ such that $L\ne L'\ne K$. Then there is an
additive polynomial ${\cal A}\in K[X]$ and some $c\in K$ such that
$L|K$ is generated by a root of ${\cal A}(X)+c$.
\end{theorem}
The degree of ${\cal A}$ is a power of $p$ (since it is additive), and
in general it may be larger than $p$.

\pars
Now we shall quickly develop the theory of tame fields. The henselian
field $(K,v)$ is said to be a \bfind{tame field} if all of its algebraic
extensions are tame extensions. By Theorem~\ref{rfmte}, this holds if
and only if $K^r$ is algebraically closed. Similarly, $(K,v)$ is said to
be a \bfind{separably tame field} if all of its separable-algebraic
extensions are tame extensions. This holds if and only if $K^r$ is
separable-algebraically closed.

By Theorem~\ref{liftZTV}, $\tilde{K}=K^r.W$ is the absolute ramification
field of $W$. If $W'|K$ is a proper subextension, then
$\tilde{K}\ne K^r.W'$. This proves:
\begin{lemma}                               \label{Wtf}
Every maximal purely wild extension $W$ is a tame field. No proper
subextension of $W|K$ is a tame field. The maximal separable
subextension is a separably tame field.
\end{lemma}

By Theorem~\ref{S}, $K\sep|K^r$ is a $p$-extension. Hence if $\chara Kv
=0$, then this extension is trivial. Since then also $\chara K=0$, it
follows that $K^r=K\sep=\tilde{K}$. Therefore,
\begin{lemma}                               \label{tamerc0}
Every henselian field of residue characteristic 0 is a tame field.
\end{lemma}
Suppose that $K_1|K$ is an algebraic extension. Then
$K^r\subseteq K_1^r$ by Theorem~\ref{liftZTV}. Hence if $K^r$ is
algebraically closed, then so is $K_1^r\,$, and if $K^r$ is
separable-algebraically closed, then so is $K_1^r\,$. This proves:
\begin{lemma}
Every algebraic extension of a tame field is again a tame field.
Every algebraic extension of a separably tame field is again a
separably tame field.
\end{lemma}
If $K^r=\tilde{K}$, then every finite extension of $(K,v)$ is  a tame
extension and thus has trivial defect, which shows that $(K,v)$ is a
defectless field. If $K^r=K\sep$, then every finite separable extension
has trivial defect. So we note:
\begin{lemma}                               \label{tamehdp}
Every tame field is henselian defectless and perfect. Every separably
tame field is henselian and all of its finite separable extensions have
trivial defect.
\end{lemma}

We give a characterization of tame and separably tame
fields (for the proof, see [K2]):
\begin{lemma}                    \label{tame}
A valued field $(K,v)$ is tame if and only if it is algebraically
maximal, $vK$ is $p$-divisible and $Kv$ is perfect. If $\chara K= \chara
Kv$ then $(K,v)$ is tame if and only if it is algebraically maximal and
perfect.

A non-trivially valued field $(K,v)$ is separably tame if and only if it
is separable-algebraically maximal, $vK$ is $p$-divisible and $Kv$ is
perfect.
\end{lemma}

This lemma together with Lemma~\ref{tamehdp} shows that for perfect
valued fields $(K,v)$ with $\chara K=\chara Kv$, the two properties
``algebraically maximal'' and ``henselian and defectless'' are
equivalent.

\begin{corollary}                           \label{cortame}
Assume that $\chara K=\chara Kv$. Then every maximal immediate algebraic
extension of the perfect hull of $(K,v)$ is a tame field (and no proper
subextension of it is a tame field). If $\chara Kv= 0$ then already the
henselization $(K^h,v)$ is a tame field.
\end{corollary}

\pars
The following is a crucial lemma in the theory of tame fields. For its
proof, see [K1] or [K2].
\begin{lemma}                                     \label{trac}
Let $(L,v)$ be a tame field and $K\subset L$ a relatively algebraically
closed subfield. If in addition $Lv|Kv$ is an algebraic extension, then
$(K,v)$ is also a tame field and moreover, $vK$ is pure in
$vL$ and $Kv=Lv$. The same holds for ``separably tame'' in the
place of ``tame''.
\end{lemma}

The break we took for the development of the theory of tame
fields is at the same time a good occasion to do some model
theoretic preparation for later sections.

%
%
\section{Some notions and tools from model theoretic
algebra}                                    \label{sectmta}
The basic idea of model theoretic algebra is to analyze the assertions
that an algebraist wants to prove, and to apply principles that are
valid for certain types of assertions. Such principles prove once and
for all facts that otherwise are proved over and over again in different
settings (as a little example, see Lemma~\ref{union} below). To state
and apply such principles, it is necessary to make it precise what it is
that we are talking about, and in which mathematical language we are
talking. The reader may interpose that mathematicians are talking about
mathematical structures, which are fixed by definitions, that is, by
axioms. If for instance we are talking about a group, then we talk about
a set of elements $G$ and a binary function $G\times G\,\rightarrow\,G$
which associates with every two elements a third one. So besides the
underlying set, we are using a \bfind{function symbol} for this
function, which is $+$ if we write the group additively, and $\cdot$ if
we write it multiplicatively. But a group also has a unit element, for
which we use the \bfind{constant symbol} $0$ in the additive and $1$ in
the multiplicative case. If we talk about ordered groups, then for
expressing the ordering we need a further symbol, which might be $<$ or
$\leq$. Although it is also binary like $+$ or $\cdot\,$, it is not a
function symbol. Since the ordering is a relation between the elements
of the group, it is called a \bfind{relation symbol}. The description of
a mathematical object may need more than one constant, function or
relation symbol. For a field, we need two binary functions, $+$ and
$\cdot\,$, and two constants, $0$ and $1$. Further, we may need function
symbols or relation symbols of any (fixed) number of entries.

A \bfind{language} is defined to be
\[{\cal L}={\cal F}\cup {\cal C}\cup {\cal R}\]
where\sn
$\bullet$\ \ ${\cal F}$ is a set of function symbols,\n
$\bullet$\ \ ${\cal C}$ is a set of constant symbols,\n
$\bullet$\ \ ${\cal R}$ is a set of relation symbols.
\sn
For example,
\[{\cal L}_{\rm G}:=\{+,-,0\}\]
is the \bfind{language of groups} (where $-$ is a function symbol with
one entry), ${\cal L}_{\rm OG} :=\{+,-,0,<\}$ is the \bfind{language of
ordered groups},
\[{\cal L}_{\rm F}:=\{+,\cdot\,,-,{ }^{-1},0,1\}\]
is the \bfind{language of fields}, and
\[{\cal L}_{\rm VF}:=\{+,\cdot\,,-,{ }^{-1},0,1,{\cal O}\}\]
is the \bfind{language of valued fields}, where ${\cal O}$ is a relation
symbol with one entry.

\pars
For a given language ${\cal L}$, an \bfind{${\cal L}$-structure} is a
quadruple\glossary{${\eu A}$}
\[{\eu A}=(A,{\cal R}_{\eu A},{\cal F}_{\eu A},{\cal C}_{\eu A})\]
where\sn
$\bullet$\ \ $A$ is a set, called the \bfind{universe} {\bf of}
${\eu A}$,\n
$\bullet$\ \ ${\cal F}_{\eu A}=\{f_{\eu A}\mid f\in {\cal F}\}$ such
that every $f_{\eu A}$ is a function on $A$ of the same arity as the
function symbol $f$,\n
$\bullet$\ \ ${\cal C}_{\eu A}=\{c_{\eu A}\mid c\in {\cal C}\}$ such
that every $c_{\eu A}^{ }$ is an element of $A$ (called a
\bfind{constant}),\n
$\bullet$\ \ ${\cal R}_{\eu A}=\{R_{\eu A}\mid R\in {\cal R}\}$ such
that every $R_{\eu A}$ is a relation on $A$ of the same arity as the
relation symbol $R$.
\sn
We call $R_{\eu A}$ the \bfind{interpretation} {\bf of} $R$ {\bf on}
$A$, and similarly for the functions $f_{\eu A}$ and the constants
$c_{\eu A}^{ }\,$. Let ${\eu A}$ and ${\eu B}$ be two
${\cal L}$-structures. Then we will call ${\eu A}$ a
\bfind{substructure} of ${\eu B}$ if the universe $A$ of ${\eu A}$
is a subset of the universe $B$ of ${\eu B}$ and the restrictions to $A$
of the interpretations of the relation and function symbols and the
constant symbols on $B$ coincide with the interpretations of the same
relation, function and constant symbols on $A$.

Any set $A$ with a distinguished element $0$, a binary function
$+:\>A\times A\,\rightarrow\,A$ and a unary function $-:\>A\,\rightarrow
\,A$ will be a structure for the language ${\cal L}_{\rm G}=\{+,-,0\}$
of groups. This does not say anything about the behaviour of the
functions $+$ and $-$ and the element $0$. For instance, $0$ may not at
all behave as a neutral element for $+$. Such properties of structures
cannot be fixed by the language. They have to be described by axioms.

By an \bfind{elementary ${\cal L}$-formula}\index{elementary formula} we
mean a syntactically correct string built up using the symbols of the
language ${\cal L}$, variables, $=$, and the logical symbols $\forall$,
$\exists$, $\neg$, $\wedge$, $\vee$, $\rightarrow$, $\leftrightarrow$.
An elementary ${\cal L}$-formula is called an \bfind{elementary
${\cal L}$-sentence} if every variable is bound by some quantifier. An
elementary ${\cal L}$-sentence is called \bfind{existential} if it is of
the form $\exists X_1\ldots \exists X_n \varphi(X_1,\ldots,X_n)$ where
$\varphi(X_1,\ldots,X_n)$ is a quantifier free ${\cal L}$-formula and
$X_1,\ldots,X_n$ are the only variables appearing in $\varphi$. Hence an
existential sentence is a sentence which only talks about the existence
of certain elements. An elementary ${\cal L}$-sentence is called
\bfind{universal existential} if it is of the form $\forall X_1\ldots
\forall X_k\exists X_{k+1}\ldots \exists X_n \varphi(X_1,\ldots,X_n)$
where $\varphi$ is as above ($k=0$ or $k=n$ are admissible, so
existential is also universal existential).

For example, the usual sentences expressing associativity,
commutativity, the fact that $0$ is a neutral element, and the
existence of inverses are universal existential elementary
${\cal L}_{\rm G}$-sentences. They form an
elementary axiom system for the class of abelian groups.
Similarly, we have elementary axiom systems for the classes of
ordered abelian groups, fields, valued fields. It is not necessary that
an axiom system consists of only finitely many axioms. For instance,
properties like ``algebraically closed'' or ``real closed'' can be
axiomatized by an infinite scheme of universal existential elementary
axioms. One can quantify over all possible polynomials of fixed degree
$n$ by quantifying over their $n+1$ coefficients. But in order to
express that {\it all\/} polynomials have a root in a field $K$ we need
countably many axioms talking about polynomials of increasing degree. In
a similar way, the property ``henselian'' is axiomatized by an infinite
scheme of universal existential elementary axioms in the language of
valued fields. In contrast to this, properties like ``complete'' or
``maximal'' have no elementary axiomatization in ${\cal L}_{\rm VF}$; we
would have to quantify over subsets of the universe, which is impossible
in elementary sentences. This shows that it makes sense to replace
``complete'' or ``maximal'' by ``henselian'', wherever possible.

The following lemma expresses (once and for all) a fact that is
(intuitively) known to every good mathematician.
\begin{lemma}                               \label{union}
The union over an ascending chain of ${\cal L}$-structures
${\eu A}_i\,$, $i\in\N$, satisfies all universal existential
elementary sentences which are satisfied in all of the ${\eu A}_i\,$.
\end{lemma}
This proves, for instance:
\begin{lemma}                               \label{unihens}
The union of an ascending chain of henselian valued fields
$(L_i,v)$, $i\in\N$, is again a henselian valued field.
\end{lemma}

Any set ${\cal T}$ of elementary ${\cal L}$-sentences
is called an \bfind{elementary axiom system} (or an
\bfind{${\cal L}$-theory}). If an ${\cal L}$-structure satisfies all
axioms in ${\cal T}$, then we call it a \bfind{model of ${\cal T}$}. So
if we have an ${\cal L}_{\rm G}$-structure which is a model of the axiom
system of groups, then we know that $+,-,0$ behave as we expect them to
do. The axiom system for valued fields expresses that $\{x\mid
{\cal O}(x) \mbox{\rm \ holds}\}$ is a valuation ring. Since $v$ is not
a function from the field into itself, we cannot simply take a function
symbol for $v$ into the language. However, we can express ``$vx\leq
vy$'' by the elementary sentence ``$\,{\cal O}(yx^{-1})\,\vee\,
x=y=0\,$''.

We say that two ${\cal L}$-structures ${\eu A},{\eu B}$ are
\bfind{elementarily equivalent}, denoted by ${\eu A}\equiv {\eu B}$,
if ${\eu A}$ and ${\eu B}$ satisfy the same elementary
${\cal L}$-sentences. An axiom system is \bfind{complete} if and only if
all of its models are elementarily equivalent. Syntactically, that means
that every elementary ${\cal L}$-sentence or its negation can be deduced
from that axiom system. For example, the axiom system of divisible
ordered abelian groups and the axiom system of algebraically closed
fields of fixed characteristic are complete. It was shown by Abraham
Robinson [RO] that the axiom system of algebraically closed
non-trivially valued fields of fixed characteristic and fixed residue
characteristic is complete.

The completeness of an axiom system yields a \bfind{Transfer Principle}
for the class axiomatized by it. For example, the completeness of the
axiom system for algebraically closed fields of fixed characteristic
tells us that every elementary ${\cal L}_{\rm F}$-sentence which holds
in one algebraically closed field will also hold in all other
algebraically closed fields of the same characteristic. This reminds of
the \bfind{Lefschetz Principle} which was stated and partially proved by
Weil and Lefschetz and says (roughly speaking) that algebraic geometry
over all algebraically closed fields of a fixed characteristic is the
same (``there is but one algebraic geometry in characteristic $p$''). As
for the elementary ${\cal L}_{\rm F}$-sentences of algebraic geometry,
this indeed follows from the completeness. However, Lefschetz and Weil
had in mind more than just the elementary sentences. That is why Weil
worked with so-called ``universal domains'' which are algebraically
closed and of infinite transcendence degree over their prime field. So
the assertion was that there is but one algebraic geometry over
universal domains of characteristic $p$. A satisfactory formalization
and model theoretic proof by use of an \bfind{infinitary language}
is due to Paul Eklof [EK]. Infinitary languages admit the
conjunction of infinitely many elementary sentences. With such
infinitary sentences, one can also express the fact that a field has
infinite transcendence degree over its prime field. This cannot be
done by elementary sentences. Indeed, algebraically closed fields
are elementarily equivalent to the algebraic closure of their prime
field, even if they have infinite transcendence degree.

\pars
In the theory of valued fields, we consider two important invariants:
value groups and residue fields. If two valued fields are equivalent
(in the language of valued fields), then so are their value groups (in
the language of ordered groups) and their residue fields (in the
language of fields). We are interested in a converse of this
implication: under which additional assumptions do we have the
so-called \bfind{Ax--Kochen--Ershov principle}:
\begin{equation}                            \label{AKEprinc}
vK\equiv vL \mbox{ and } Kv\equiv Lv\>\Rightarrow\>(K,v)\equiv
(L,v)
\end{equation}
It follows from the fact that the henselization is an immediate
extension that this principle can only hold for henselian fields.
And indeed, it was shown by James Ax and Simon Kochen [AK1] and
independently by Yuri Ershov [ER2] that this principle holds for all
henselian fields of residue characteristic 0. This is the famous
\bfind{Ax--Kochen--Ershov Theorem}. Ax and Kochen
proved it in order to deduce a correct version of \bfind{Artin's
conjecture} about the fields $\Qp$ of $p$-adic numbers ([AK1]; cf.\
[CK] or [K2]). From the Ax--Kochen--Ershov Theorem, one obtains an
equivalence of two ultraproducts:
\begin{equation}                            \label{equivultra}
\prod_{\mbox{\small\rm $p$ prime}} \Qp \,/{\cal D}\>\equiv\>
\prod_{\mbox{\small\rm $p$ prime}} \Fp((t))\,/{\cal D}
\end{equation}
because both fields carry a canonical henselian valuation of residue
characteristic 0 and have equal value groups and residue fields. Since a
product like $\prod_p \Qp$ would not even be a field, one has to take
the product modulo a non-principal ultrafilter ${\cal D}$ on the set of
all primes. Let us quickly give the main facts about ultraproducts.

A filter ${\cal D}$ on a set $I$ is called an \bfind{ultrafilter} if
\begin{equation}
J\subset I\>\wedge\> J\notin {\cal D}
\;\;\Longrightarrow\;\; I\setminus J\in {\cal D}\;,
\end{equation}
and it is called \bfind{non-principal} if it is not of the form
$\{J\subset I\mid i\in J\}$, for no $i\in I$. If ${\eu A}_i\,$,
$i\in I$ are ${\cal L}$-structures, then the \bfind{ultraproduct}
$\prod_{i\in I}^{}{\eu A}_i\,/{\cal D}$ is defined by setting, for all
$(a_i), (b_i)\in \prod_{i\in I}^{}{\eu A}_i\,$,
\begin{equation}
(a_i)\equiv (b_i) \mbox{ modulo } {\cal D} \;\;\Longleftrightarrow\;\;
\{i\in I\mid a_i=b_i\}\>\in\>{\cal D}\;.
\end{equation}
The following theorem is due to J.~\L os. For a proof, see [K2] or [CK].
\begin{theorem} {\bf (Fundamental Theorem of Ultraproducts)}\label{FTU}
\n For every ${\cal L}$-sentence $\varphi$, $\prod_{i\in I}^{}{\eu
A}_i\,/ {\cal D}$ satisfies $\varphi$ if and only if
\begin{equation}
\{i\in I\mid {\eu A}_i\mbox{\rm\ satisfies }\varphi\}\in {\cal D}\;.
\end{equation}
\end{theorem}
Thus, elementary sentences which are true for all
$\Fp((t))$ can be transferred to $\prod_p \Fp((t))/{\cal D}$, from there
via (\ref{equivultra}) to $\prod_p \Qp /{\cal D}$, and from there, by
varying over all possible ultrafilters on the set of primes, to almost
all $\Qp$. The elementary sentences we are interested in are deduced
from the analogue of Artin's conjecture which holds for all $\Fp((t))$,
as proved by Serge Lang [L1].

Since the proof of the Ax--Kochen--Ershov Theorem, the
Ax--Kochen--Ershov principle (\ref{AKEprinc}) has also been proved for
other classes of valued fields, like $p$-adically closed fields, or
algebraically maximal fields satisfying Kaplansky's hypothesis A
(``algebraically maximal'' is stronger than ``henselian'' if the residue
characteristic is positive, as we have seen in Example~\ref{exampdef2}).
The proofs used, more or less explicitly, that those fields have unique
maximal immediate extensions. But this is not necessary for the validity
of the Ax--Kochen--Ershov principle (\ref{AKEprinc}). Using instead the
Generalized Stability Theorem (Theorem~\ref{ai} below) and the Henselian
Rationality of Immediate Function Fields (Theorem~\ref{stt3} below), I
proved that Ax--Kochen--Ershov principle (\ref{AKEprinc}) also holds for
all tame fields ([K1], [K2]).

There is another version of (\ref{AKEprinc}) which will bring us closer
to applications in algebraic geometry. There is a notion which in the
past years has turned out to be more basic and flexible than that of
elementary equivalence. Through general tools of model theory (like
Theorem~\ref{RT} below), notions like elementary
equivalence can often be reduced to it. Take ${\eu B}$ to be an
${\eu L}$-structure, and ${\eu A}$ a substructure of ${\eu B}$. We form
a language ${\cal L}(A)$ by adjoining the universe $A$ of ${\eu A}$ to
the language ${\cal L}$. That is, in the language ${\cal L}(A)$ we have
a constant symbol for every element of the structure ${\eu A}$, so we
can talk about every single element.
We say that \bfind{${\eu A}$ is existentially closed in ${\eu B}$},
denoted by ${\eu A}\ec {\eu B}$, if every existential
${\cal L}(A)$-sentence holds already in ${\eu A}$ if it holds
in ${\eu B}$. (The other direction is trivial: if something exists in
${\eu A}$, then it also exists in ${\eu B}$.) Let us illustrate the use
of this notion by three important examples.

\begin{example}
Take a field extension $L|K$. If $K\ec L$ in the language of fields,
then $K$ is relatively algebraically closed in $L$. To see this, take
$a\in L$ algebraic over $K$. Take $f=X^n+c_{n-1}X^{n-1}+\ldots+c_0\in
K[X]$ to be the minimal polynomial of $a$ over $K$. Since $f(a)=0$, we
know that the existential ${\cal L}_{\rm F}(K)$-sentence
\[\exists X\, X^n+c_{n-1}X^{n-1}+\ldots+c_0=0\]
holds in $L$. (``$X^n$'' is
an abbreviation for ``$X\cdot\ldots\cdot X$'' where $X$ appears $n$
times. Observe that we need the constants from $K$ since we use the
coefficients $c_i$ in our sentence.) Since $K\ec L$, it must also hold
in $K$. That means, that $f$ also has a root in $K$. But as a minimal
polynomial, $f$ is irreducible over $K$. This shows that $f$ must be
linear, i.e., $a\in K$. One can also show that $L|K$ must be separable.

Similarly, let $G\subset H$ be an extension of abelian groups such that
$G\ec H$ in the language of groups. Take $\alpha\in H$ such that
$n\alpha\in G$ for some integer $n>0$. Set $\beta=n
\alpha$. Then the existential ${\cal L}_{\rm G}(G)$-sentence
``$\exists X\, nX=\beta$'' holds in $H$. (Here, ``$nX$'' is just
an abbreviation for ``$X+\ldots+ X$'' where $X$ appears $n$
times.) Hence, it must also hold in $G$. That is, $\alpha=\beta/n\in G$.
Hence, we have:
\begin{lemma}                               \label{KecL}
If $L|K$ is an extension of fields such that $K\ec L$ in the language of
fields, then $K$ is relatively algebraically closed in $L$ and $L|K$ is
separable. If $G\subset H$ is an extension of abelian groups such that
$G\ec H$ in the language of groups, then $H/G$ is torsion free.
\end{lemma}
\end{example}

\begin{example}                             \label{examperp}
Take a function field $F|K$ such that $K\ec F$ in the language of
fields. Since $F|K$ is separable by Lemma~\ref{KecL}, we
can choose a separating transcendence basis $t_1,\ldots,t_k$ of
$F|K$ and an element $z\in F$ such that $F=K(t_1,\ldots,t_k,z)$
with $z$ separable-algebraic over $K(t_1,\ldots,t_k)$. Take
$f\in K[X_1,\ldots,X_k,Z]$ to be the irreducible polynomial of
$z$ over $K[t_1,\ldots, t_k]$ (obtained from the minimal polynomial
by multiplication with the common denominator of the coefficients
from $K(t_1,\ldots,t_k)$). We have that $f(t_1,\ldots,t_k,z)=0$.
Since $z$ is a simple root of $f$, we also have that
$\frac{\partial f}{\partial Z}(t_1,\ldots,t_k,z)\ne 0$. Further,
we take $n$ arbitrary non-zero elements $z_1,\ldots,z_n\in F$ which we
write as $g_1/h_1,\ldots,g_n/h_n$ with non-zero elements $g_i,h_i\in
K[t_1,\ldots,t_k,z]$. Now the existential ${\cal L}_{\rm F}(K)$-sentence
\begin{eqnarray*}
&& \mbox{``}\>\exists Y_1\ldots\exists Y_k\exists Y\; f(Y_1,\ldots,
Y_k,Y)=0 \>\wedge\>\frac{\partial f}{\partial Z}(Y_1,\ldots,Y_k,Y)
\ne 0\>\wedge\\ &&
\;\;\;\;\wedge\>g_1(Y_1,\ldots,Y_k,Y)\ne 0\>\wedge\ldots\wedge\>
g_n(Y_1,\ldots,Y_k,Y)\ne 0 \>\wedge\\ &&
\;\;\;\;\wedge\>h_1(Y_1,\ldots,Y_k,Y)\ne 0\>\wedge\ldots\wedge\>
h_n(Y_1,\ldots,Y_k,Y)\ne 0\>\mbox{''}
\end{eqnarray*}
holds in $F$. Hence it must also hold in $K$, that is, there are
$c_1,\ldots,c_k,d\in K$ such that $f(c_1,\ldots,c_k,b)=0$,
$\frac{\partial f}{\partial Z}(c_1,\ldots,c_k,d)\ne 0$
and
\[g_i(c_1,\ldots,c_k,b)/h_i(c_1,\ldots,c_k,d)\ne 0,\infty\;.\]

On $K[t_1,\ldots,t_k]$, we have the homomorphism given by $t_i\mapsto
c_i\,$, or equivalently, by $t_i-c_i\mapsto 0$. As in
Example~\ref{examp2}, we can construct a place $P$ of maximal rank of
$K(t_1,\ldots,t_k)= K(t_1-c_1,\ldots,t_k-c_k)$ which extends this
homomorphism. Its residue field is $K$. Now we consider the polynomial
$g(Z)= f(t_1,\ldots,t_k,Z)$. Its reduction modulo $P$ is the polynomial
$f(c_1,\ldots,c_k,Z)$ which admits $d$ as a simple root. Hence by
Hensel's Lemma, $g(Z)$ has a root $z'$ in the henselization
$(K(t_1,\ldots,t_k)^h,P)$ of $(K(t_1,\ldots,t_k),P)$. Thus, the
assignment $z\mapsto z'$ defines an embedding of $F$ over
$K(t_1,\ldots,t_k)$ in $K(t_1,\ldots,t_k)^h$, and pulling the place $P$
from the image of this embedding back to $F$, we obtain on $F$ a place
$P$ with residue field $K$ and having maximal rank. In addition,
$z_iP\ne 0,\infty$. We have proved:
\begin{lemma}                               \label{exrapl}
Take a function field $F|K$ such that $K\ec F$ in the language of
fields. Take non-zero elements $z_1,\ldots,z_n\in F$. Then there exists
a rational place $P$ of $F|K$ of maximal rank and such that $z_iP\ne
0,\infty$ for $1\leq i\leq n$, and a model of $F|K$ on which $P$ is
centered at a smooth point.
\end{lemma}
Note that it was crucial for our proof that $F|K$ is finitely generated
(because elementary sentences can only talk about finitely many
elements). If a field extension $L|K$ is not finitely generated, then
there may not exist a place $P$ of $L$ such that $LP=K$, even if
$K\ec L$.
\end{example}

\begin{example}
A field $K$ is called a \bfind{large field} (cf.\ [POP]) if every smooth
curve over $K$ has infinitely many $K$-rational points, provided it has
at least one $K$-rational point. For the proof of the following theorem,
see [K2] or [K3].
\begin{theorem}                             \label{largef}
A field $K$ is large if and only if $K\ec K((t))$ (in the language of
fields).
\end{theorem}
\end{example}

The use of ``$\ec$'' gives us another version of the Ax--Kochen--Ershov
principle:
\begin{equation}                            \label{AKEec}
vK\ec vL \mbox{ and } Kv\ec Lv\>\Rightarrow\>(K,v)\ec (L,v)\;.
\end{equation}
This principle also holds for the classes of valued fields that we
mentioned above: henselian fields of residue characteristic 0,
$p$-adically closed fields, algebraically maximal fields satisfying
Kaplansky's hypothesis A, tame fields. A short proof for the case of
henselian fields of residue characteristic 0 is given in the appendix of
[KP]. This form of the Ax--Kochen--Ershov principle is applied in the
proof of Theorem~\ref{MKP} below.

%
%
\section{Saturation and embedding lemmas}
%
How can a principle like (\ref{AKEec}) be proved? In fact, nice model
theoretic results often just represent a good algebraic structure
theory. Indeed, using a very useful model theoretic tool, we can easily
transfer ``$\ec$'' to an algebraic fact. The tool is that of a
\bfind{$\kappa$-saturated model} (where $\kappa$ is a cardinal number).
Saturation is a property which is not elementary, quite similar to
``complete'' or ``maximal'', but still different (in fact, ``maximal''
and ``saturated'' are to some extent mutually exclusive). Before
defining ``$\kappa$-saturated'',
I want to illustrate its meaning by an example which plays a remarkable
role in the theory of ordered structures. We take fields, but the same
can be done for ordered abelian groups and other ordered structures.
\begin{example}
Take any ordinal number $\alpha$ and an ordered field $(K,<)$. For
$A,B\subset K$ we will write $A<B$ if every element of $A$ is smaller
than every element of $B$. Now $(K,<)$ is said to be an
\bfind{$\eta_\alpha$-field} if for every two subsets $A,B\subset K$ of
cardinality less than $\aleph_\alpha$ (= the $\alpha$-th cardinal
number) such that $A<B$, there is some $c\in K$ such that $A<\{c\}<B$,
i.e., $c$ lies between $A$ and $B$. Note that because of the restriction
of the cardinality of $A$ and $B$, this does not mean that $(K,<)$ is
cut-complete (in fact, the only cut-complete field is $\R$, while there
is an abundance of $\eta_\alpha$-fields).

Given $A,B\subset K$ such that $A<B$, we consider the collection of
elementary sentences (in the language of ordered fields with constants
from $K$) ``$a<X$'', $a\in A$, and ``$X<b$'', $b\in B$. It is clear that
if we take any finite subset of these, then there is some element in $K$
that we can insert for $X$ so that all of these finitely many
sentences hold. That is, our collection of sentences is \bfind{finitely
realizable} in $(K,<)$. Now if $(K,<)$ is $\kappa$-saturated, then this
tells us the following: if the cardinality of $A\cup B$ is smaller
than $\kappa$, then there is an element $c\in K$ which simultaneously
satisfies {\it all} of our above sentences (we say that $c$
\bfind{realizes} the above set of elementary sentences). But this means
that $A<\{c\}<B$. So we see that every $\aleph_\alpha$-saturated ordered
field is an $\eta_\alpha$-field.
\end{example}

Let us extract a definition from our example. An ${\cal L}$-structure
${\eu A}$ will be called \bfind{$\kappa$-saturated} if for every subset
$S$ of its universe $A$ of cardinality less than $\kappa$, every set of
elementary ${\cal L}(S)$-sentences is realizable in ${\eu A}$, provided
that it is finitely realizable in ${\eu A}$. To express the fact that
there are enough $\kappa$-saturated ${\cal L}$-structures, we need one
further notion. Given an ${\cal L}$-structure ${\eu B}$ with
substructure ${\eu A}$, we say that \bfind{${\eu B}$ is an elementary
extension of ${\eu A}$} and write ${\eu A}\prec {\eu B}$ if
{\it every} ${\cal L}(A)$-sentence holds in ${\eu A}$ if and only if it
holds in ${\eu B}$. (So in contrast to ``existentially closed'', here
we do not restrict the scope to existential sentences.) For example, an
algebraically closed field $K$ is existentially closed in
every extension field, and every algebraically closed extension field
of $K$ is an elementary extension of $K$. If $K\prec L$, then $K$ is
existentially closed in every intermediate field.

We are going to state the theorem which provides us with sufficiently
many $\kappa$-saturated structures. It is a consequence of one of the
basic theorems of model theory:
\begin{theorem}                             \label{CT}
\ {\bf (Compactness Theorem)} \
A set of elementary ${\cal L}$-sentences has a model if and only if each
of its finite subsets has a model.
\end{theorem}

For the proof of the next theorems, see [CK] or [K2].
\begin{theorem}                             \label{exsat}
For every infinite ${\cal L}$-structure ${\eu A}$ and every large enough
$\kappa$ there exists a $\kappa$-saturated elementary extension of
${\eu A}$.
\end{theorem}
Here ``large enough'' means: larger than the cardinality of the language
(which, if infinite, will in most cases be the cardinality of the set of
constants appearing in the language), and larger than the cardinality of
the universe of ${\eu A}$. Now the reduction of ``$\ec$'' to an
algebraic statement is done as follows:
\begin{theorem}                            \label{ecemb}
Take an ${\cal L}$-structure ${\eu B}$ with substructure ${\eu A}$. Take
$\kappa$ larger than the cardinality of ${\cal L}$ and the cardinality
of the universe of ${\eu B}$. Further, choose a $\kappa$-saturated
elementary extension ${\eu A}^*$ of ${\eu A}$. Then ${\eu A}\ec
{\eu B}$ holds if and only if there is an embedding of
${\eu B}$ over ${\eu A}$ in ${\eu A}^*$.
\end{theorem}
If there is an embedding of ${\eu B}$ over ${\eu A}$ in ${\eu A}^*$,
then every existential sentence holding in ${\eu B}$ will carry over
to the image of ${\eu B}$ in ${\eu A}^*$, from where it goes up to
${\eu A}^*$. Since ${\eu A}\prec {\eu A}^*$, it then also holds in
${\eu A}$.

A nice additional feature of saturation is the following:
\begin{theorem}                             \label{ecembfg}
There is an embedding of ${\eu B}$ over ${\eu A}$ in ${\eu A}^*$ already
if for every finitely generated subextension ${\eu A}\subset {\eu B}'$
of ${\eu A}\subset {\eu B}$ there is an embedding of ${\eu B}'$ over
${\eu A}$ in ${\eu A}^*$.
\end{theorem}

So if we have a field extension $L|K$ and want to prove that $K\ec L$,
we take a $\kappa$-saturated elementary extension $K^*$ of $K$, for
$\kappa$ larger than the cardinality of $L$, and seek to embed $L$ over
$K$ in $K^*$. By the last theorem, we only have to show that every
finitely generated subextension $F|K$ of $L|K$ embeds in $K^*$. But
a finitely generated extension is a function field (in view of
Lemma~\ref{KecL} we can exclude the case where $F|K$ is algebraic).

If $K$ is an algebraically closed field, then so is $K^*$ because it is
an elementary extension of $K$. The assumption that $K^*$ is
$\kappa$-saturated with $\kappa$ larger than the cardinality of $L$
implies that the transcendence degree of $K^*|K$ is at least as large
as that of $L|K$. So we see that $L$ embeds over $K$ in $K^*$ (even
without employing the last theorem). This proves that every
algebraically closed field is existentially closed in every extension
field.

\parm
Let's see how we can prove a principle like (\ref{AKEec}) with the above
tools. We take a $\kappa$-saturated elementary extension $(K,v)^*=
(K^*,v^*)$ of $(K,v)$ (with respect to the language of valued fields).
Then it is easy to prove that $v^*K^*$ is a $\kappa$-saturated
elementary extension of $vK$ (with respect to the language of ordered
groups) and that $K^*v^*$ is a $\kappa$-saturated elementary extension
of $Kv$ (with respect to the language of fields). Thus, we see from
Theorem~\ref{ecemb} that $vK\ec vL$ implies that $vL$ embeds over $vK$
in $v^*K^*$, and that $Kv\ec Lv$ implies that $Lv$ embeds over $Kv$ in
$K^*v^*$. So Theorem~\ref{ecemb} shows that we can prove (\ref{AKEec})
by an \bfind{embedding lemma} of the form: {\it If $vL$ embeds over $vK$
in $v^*K^*$ and $Lv$ embeds over $Kv$ in $K^*v^*$ and} (additional
assumptions) {\it then $(L,v)$ embeds over $K$ in $(K,v)^*$ (as a valued
field)}. See Example~\ref{mtGGST} and Example~\ref{mtHR} below for two
different cases and a sketch of the proof of (\ref{AKEec}) for tame
fields.

\parm
To conclude this section, let us come back to elementary extensions.
An ${\cal L}$-theory ${\cal T}$ is called \bfind{model complete} if for
every two models $\eu A$ and $\eu B$ of ${\cal T}$ such that $\eu A$ is
a substructure of $\eu B$ we have that ${\eu A}\prec {\eu B}$. This is
closely connected to the relation ${\eu A}\ec {\eu B}$ through the
following important criterion (cf.\ [K2] or [CK]):
\begin{theorem}                             \label{RT}
\ {\bf (Robinson's Test)} \n
Assume that for every two models $\eu A$ and $\eu B$ of ${\cal T}$ such
that $\eu A$ is a substructure of $\eu B$ we have that ${\eu A}\ec
{\eu B}$. Then ${\cal T}$ is model complete.
\end{theorem}
For example, this theorem together with the fact that every
algebraically closed field is existentially closed in every extension
field shows that the axiom system of algebraically closed fields is
model complete. Furthermore, with this theorem together with
Theorem~\ref{ecemb}, Theorem~\ref{prelBour}, Theorem~2 of [KA1] and
Theorem~\ref{allext}, it is not hard to prove the following theorem of
Abraham Robinson:
\begin{theorem}                             \label{mcvfAR}
The elementary axiom system of non-trivially valued algebraic\-ally
closed fields is model complete.
\end{theorem}

Observe that we do not need ``side conditions'' about the value groups
and the residue fields here (because they are divisible and
algebraically closed, respectively). But there is also an
Ax--Kochen--Ershov principle with $\prec$ in the place of $\ec$ that
again holds for the classes of valued fields which I mentioned above.

\pars
\begin{example}
Another simple but useful example for a fact proved by an embedding
lemma is the following:
\begin{theorem}                             \label{AAf}
If $(K,v)$ is henselian and $(L,v)$ is a separable extension of $(K,v)$
within its completion, then $(K,v)\ec (L,v)$.
\end{theorem}
This fact can be seen as the (much simpler) ``field version of Artin
Approximation''. It was observed in the 1960s by Yuri Ershov; for a
proof, see [K2]. Together with Theorem~\ref{largef}, a modification
of Theorem~\ref{ecwtd} and the transitivity of $\ec$, this theorem
(applied to $(K(t),v)^h$) can be used to prove (cf.\ [K2], [K3]):
\begin{theorem}
If the field $K$ admits a henselian valuation, then $K\ec K((t))$,
i.e., $K$ is a large field.
\end{theorem}
\end{example}

%
%
\section{The Generalized Grauert--Remmert Stability
Theorem}                                    \label{sectGST}
Let us return to our problem of inertial generation as considered at the
end of Section~\ref{sectig}. Our problem was to show that the finite
immediate extension (\ref{extens}) of henselian fields is trivial. If it
is not, then by Corollary~\ref{hidef} it has non-trivial defect (which
then is equal to its degree). So we would like to show that the field
$F_0(\eta)^h$ is a defectless field. The reason for this would have to
lie in the special way we have constructed this field.

At this point, let us invoke a deep and important theorem from the
theory of valued function fields ([K1], [K2]). For historical reasons, I
call it the \bfind{Generalized Grauert--Remmert Stability Theorem}
although I do not like the notion ``stable''. It is one of those words
in mathematics which is very often used in different contexts, but in
most cases does not reflect its meaning. I replace it by ``defectless''.

If $(F|K,v)$ is an extension of valued fields of finite transcendence
degree, then by inequality (\ref{wtdgeq}) of Corollary~\ref{fingentb},
$\trdeg F|K -\trdeg Fv|Kv -\rr (vF/vK)$ is a non-negative integer.
We call it the \bfind{transcendence defect} of $(F|K,v)$.
We say that $(F|K,v)$ is \bfind{without transcendence defect} if
the transcendence defect is $0$.

\begin{theorem}                           \label{ai}
Let $(F|K,v)$ be a valued function field without transcendence defect.
If $(K,v)$ is a defectless field, then also $(F,v)$ is a defectless
field.
\end{theorem}
\sn
This theorem has a long and interesting history. Hans Grauert and
Reinhold Remmert [GR] first proved it in a very restricted case, where
$(K,v)$ is a complete discrete valued field and $(F,v)$ is discrete too.
There are generalizations by Laurent Gruson [GRU], Michel Matignon, and
Jack Ohm [OH]. All of these generalizations are restricted to the case
$\trdeg F|K= \trdeg Fv|Kv$, the case of \bfind{constant reduction}. The
classical origin of it is the study of curves over number fields and the
idea to reduce them modulo a $p$-adic valuation. Certainly, the
reduction should again render a curve, this time over a finite field.
This is guaranteed by the condition $\trdeg F|K=\trdeg Fv|Kv$, where $F$
is the function field of the curve and $Fv$ will be the function field
of its reduction. Naturally, one seeks to relate the genus of $F|K$ to
that of $Fv|Kv\,$. Several authors proved \bfind{genus inequalities}. To
illustrate the use of the defect, we cite an inequality proved by Barry
Green, Michel Matignon and Florian Pop in [GMP1]. Let $F|K$ be a
function field of transcendence degree 1 and assume that $K$ coincides
with the constant field of $F|K$ (the relative algebraic closure of $K$
in $F$). Let $v_1,\ldots,v_s$ be distinct constant reductions of $F|K$
which have a common restriction to $K$. Then:
\begin{equation}
1-g_F\leq 1-s+\sum_{i=1}^{s} \delta_i e_i r_i (1-g_i)
\end{equation}
where $g_F$ is the genus of $F|K$ and $g_i$ the genus of $Fv_i|Kv_i\,$,
$r_i$ is the degree of the constant field of $Fv_i|Kv_i$ over $Kv_i\,$,
$\delta_i$ is the defect of $(F^{h(v_i)}|K^{h(v_i)},v_i)$ where
``$.^{h(v_i)}$'' denotes the henselization with respect to $v_i$,
and $e_i= (v_iF:v_iK)$ (which is always finite in the constant
reduction case by virtue of Corollary~\ref{fingentb}). It follows that
constant reductions $v_1,v_2$ with common restriction to $K$ and
$g_1=g_2= g_F\geq 1$ must be equal. In other words, for a fixed
valuation on $K$ there is at most one extension $v$ to $F$ which is a
\bfind{good reduction}, that is, (i) $g_F=g_{Fv}\,$, (ii) there exists
$f\in F$ such that $vf=0$ and $[F:K(f)]= [Fv:Kv(fv)]$, (iii) $Kv$ is the
constant field of $Fv|Kv\,$. An element $f$ as in (ii) is called a
\bfind{regular function}.

More generally, $f$ is said to have the \bfind{uniqueness property} if
$fv$ is transcendental over $Kv$ and the restriction of $v$ to $K(f)$
has a unique extension to $F$. In this case, $[F:K(f)]=\delta\, e\,
[Fv:Kv(fv)]$ where $\delta$ is the defect of $(F^h|K^h,v)$ and
$e=(vF:vK(f))=(vF:vK)$. If $K$ is algebraically closed, then $e=1$, and
it follows from the Stability Theorem that $\delta=1$; hence in this
case, every element with the uniqueness property is regular.

It was proved in [GMP2] that $F$ has an element with the uniqueness
property already if the restriction of $v$ to $K$ is henselian. The
proof uses Theorem~\ref{mcvfAR}
and ultraproducts of function fields. Elements with the
uniqueness property also exist if $vF$ is a subgroup of $\Q$ and $Kv$
is algebraic over a finite field. This follows from work in [GMP4] where
the uniqueness property is related to the \bfind{local Skolem property}
which gives a criterion for the existence of algebraic $v$-adic integral
solutions on geometrically integral varieties.

\pars
As an application to rigid analytic spaces, the Stability Theorem
is used to prove that the quotient field of the free Tate algebra
$T_n(K)$ is a defectless field, provided that $K$ is. This in turn is
used to deduce the \bfind{Grauert--Remmert Finiteness Theorem}, in a
generalized version due to Gruson; see [BGR].

\pars
Surprisingly, it was not before the model theory of valued fields
developed in positive characteristic that an interest in a generalized
version of the Stability Theorem arose. But a criterion like Robinson's
Test (Theorem~\ref{RT}) forces us to deal with arbitrary extensions of
arbitrarily large valued fields. For instance, it is virtually
impossible to restrict oneself to rank 1 in order to prove
model completeness or completeness of a class of valued fields. And
the extensions $(L|K,v)$ in question won't obey a restriction like
``$vL/vK$ is a torsion group''. Therefore, I had to prove the
above Generalized Stability Theorem. At that time, I had not heard of
the Grauert--Remmert Theorem, so I gave a purely valuation theoretic
proof ([K1], [K2]), not based on the original proofs of Grauert--Remmert
or Gruson like the other cited generalizations.

Later, I was amazed to see that the Generalized Stability Theorem is
also the suitable version for an application to the problem of local
uniformization. (If your valuation $v$ is trivial on the base field $K$
and you ask that $\trdeg F|K=\trdeg Fv|Kv$, then $vL/\{0\}$ is torsion,
so $vL=\{0\}$ and $v$ is also trivial on $F$; this is not quite the case
we are interested in.) So let's now describe this application. By our
assumption at the end of Section~\ref{sectig}, $P$ is an
Abhyankar place on $F$ and hence also on $F_0(\eta)$. That is,
$(F_0(\eta)|K,P)$ is a function field without transcendence defect. As
$P$ is trivial on $K$, also $v_P^{ }$ is trivial on $K$. But a trivially
valued field $(K,v)$ is always a defectless field since for every finite
extension $L|K$ we have that $[L:K]=[Lv:Kv]$. Hence by the Generalized
Stability Theorem, $(F_0(\eta),v_P^{ })$ is a defectless field. By
Theorem~\ref{dlfhens}, also $(F_0(\eta)^h, v_P^{ })$ is a defectless
field. Therefore, since $(F^h|F_0(\eta)^h,v_P^{ })$ is an immediate
extension, Corollary~\ref{hidef} shows that it must be trivial. We have
proved that $F^h=F_0(\eta)^h$. By construction, $F_0(\eta)^h$ was a
subfield of the absolute inertia field of $(F_0,P)$. Hence also $F$ is a
subfield of that absolute inertia field, showing that $(F,P)$ is
inertially generated. We have thus proved the first part of the
following theorem (I leave the rest of the proof to you as an
exercise; cf.\ [K6]):
\begin{theorem}                                 \label{hrwtd}
Assume that $P$ is an Abhyankar place of $F|K$ and that $FP|K$
is a separable extension. Then $(F|K,P)$ is inertially generated. If in
addition $FP=K$ or $FP|K$ is a rational function field, then $(F|K,P)$
is henselian generated. In all cases, if $v_P^{ }F=\Z v_P^{ }x_1\oplus
\ldots\oplus\Z v_P^{ }x_\rho$ and $y_1P,\ldots,y_\tau P$ is a separating
transcendence basis of $FP|K$, then $\{x_1,\ldots,x_\rho,y_1,\ldots,
y_\tau\}$ is a generating transcendence basis.
\end{theorem}

\begin{example}                             \label{mtGGST}
Let's now describe the
model theoretic use of the Generalized Stability Theorem.
Take $(L|K,v)$ of finite transcendence degree and without transcendence
defect. Note that then for every function field $F|K$ contained in
$L|K$, also the extension $(F|K,v)$ is without transcendence defect.
Further, we assume that $(K,v)$ is a non-trivially valued henselian
defectless field and that $vK\ec vL$ and $Kv\ec Lv$. Making again
essential use of the Generalized Stability, one proves a generalization
of Theorem~\ref{hrwtd} to the case of non-trivially valued ground
fields. Using also Theorem~\ref{prelBour}, Lemma~\ref{KecL} and
Theorem~\ref{hensuniqemb}, it is easy to show that the embeddings
of $vL$ and $Lv$ (which induce embeddings of $vF$ and $Fv$)
lift to an embedding of $(F,v)$ in $(K,v)^*$. This proves:
\begin{theorem}                             \label{ecwtd}
Take a non-trivially valued henselian defectless field $(K,v)$ and an
extension $(L|K,v)$ of finite transcendence degree, without
transcendence defect. If $vK\ec vL$ and $Kv\ec Lv$, then $(K,v)\ec
(L,v)$.
\end{theorem}
This theorem shows the advantage of the notion ``$\ec$'' since there is
no analogue for ``$\prec$''.
\end{example}

To conclude this section, we give a short sketch of a main part of
the proof of the Generalized Stability Theorem. This is certainly
interesting because very similar methods have been used by Shreeram
Abhyankar for the proof of his results in positive characteristic
(see, e.g., [A1], [A6], [A7]). We have to prove that a certain henselian
field $(L,v)$ is a defectless field. We take an arbitrary finite
extension $(L'|L,v)$ and have to show that it has trivial defect. We may
assume that this extension is separable since the case of purely
inseparable extensions can be considered separately and is much easier.
Looking at $(L'|L,v)$, we are completely lost since we have not the
slightest chance to develop a good structure theory. But we only have to
deal with the defect, and we remember that a defect only appears if
extensions beyond the absolute ramification field $L^r$ are involved. So
instead of $(L'|L,v)$ we consider the extension $(L'.L^r|L^r,v)$ which
has the same defect as $(L'|L,v)$, although it will in general not have
the same degree (the use of this fact reminds of Abhyankar's ``Going
Up'' and ``Coming Down''; cf.\ [A1]). Now we use the fact that by
Theorem~\ref{S} the
separable-algebraic closure of $L^r$ is a $p$-extension. It follows that
its subextension $L'.L^r|L^r$ is a tower of Artin--Schreier extensions
(cf.\ Lemma~\ref{RemASe}). Since the defect is multiplicative, to prove
that $(L'.L^r|L^r,v)$ has trivial defect it suffices to show that each
of these Artin--Schreier extensions has trivial defect. So we take such
an extension, generated by a root $\vartheta$ of an irreducible
polynomial $X^p-X-c$ over some field $L''$ in the tower. By what we
learned in Example~\ref{exampHL1}, $vc\leq 0$. If $vc=0$, the extension
(if it is not trivial) would correspond to a proper separable extension
of the residue field; but as we are working beyond the absolute
ramification field, our residue field is already separable-algebraically
closed. So we see that $vc<0$. If $b\in L''$, then also the element
$\vartheta-b$ generates the same extension. By the additivity of the
polynomial $X^p-X$, $\vartheta-b$ is a root of the Artin--Schreier
polynomial $X^p-X-(c-b^p+b)$. The idea now is to use this principle to
deduce a ``normal form'' for $c$ from which we can read off that the
extension has trivial defect. Still, we are quite lost if we do not make
some reductions beforehand. First, it is clear that one can proceed
by induction on the transcendence degree; so we can reduce to the case
of $\trdeg L|K=1$. Second, as $v$ may not be trivial on $K$, it may have
a very large rank. By general valuation theory, one has to reduce first
to finite rank and then to rank 1. This being done, one can show that
$c$ can be taken to be a polynomial $g\in K[x]$, where $x\in L''$ is
transcendental over $K$. Now the idea is the following: if $k=p
\cdot\ell$ and $g$ contains a non-zero summand $c_kx^k$, then we replace
it by $c_k^{1/p}x^\ell$. This is done by setting $b=c_k^{1/p}x^\ell$ in
the above computation. In this way one eliminates all $p$-th powers in
$g$, and the thus obtained normal form for $c$ will show that the
extension has trivial defect. This method (which I call
``Artin--Schreier surgery'') seems to have several applications; I used
it again to prove a quite different result (Theorem~\ref{stt3} below).
It can also be found in the paper [EPP].

Let us note that the Artin--Schreier polynomials appear in Abhyankar's
work in a somewhat disguised form. This is because the coefficients
have to lie in the local ring he is working in. For example, if
$vc<0$, we would rather prefer to have a polynomial having coefficients
in the valuation ring, defining the same extension as $X^p-X-c$.
Setting $X=cY$, we find that if $\vartheta$ is a root of $X^p-X-c$, then
$\vartheta/c$ is a root of $Y^p-c^{1-p}Y-c^{1-p}$ with $vc^{1-p}=
(1-p)vc>0$. Therefore, Abhyankar considers polynomials of the form $Z^p
-c_1Z-c_2\,$ (cf.\ e.g., [A1], page 515, [A6], Theorem (2.2), or [A7],
page 34). In an extension obtained from $L$ by adjoining a $(p-1)$th
root of $c_1$ (if $(L,v)$ is henselian, then such an extension is tame),
this polynomial can be transformed back into an Artin--Schreier
polynomial.

\parm
Having shown inertial generation for function fields with Abhyankar
places, let us return to our problem of local uniformization.

%
%
\section{Relative local uniformization}     \label{sectrlu}
Throughout, I have stressed the function field aspect of local
uniformization. I have shown that function fields with Abhyankar places
can be generated in a nice way. I have talked about the algebraic
elements satisfying the assumption of the Multidimensional Hensel's
Lemma. The logical consequence of all this is to try to build up our
function field step by step: first, choose a nice transcendence basis,
according to Theorem~\ref{hrwtd}, then find algebraic elements, one
after the other, each of them satisfying the assumptions of Hensel's
Lemma over the previously generated field. This is the origin of the
following definition.

Take a finitely generated extension $F|K$, not necessarily
transcendental, and a place $P$ on $F$, not necessarily trivial on
$K$. We write ${\cal O}_F$ for the valuation ring of $P$ on $F$, and
${\cal O}_K$ for the valuation ring of its restriction to $K$. Further,
take elements $\zeta_1,\ldots, \zeta_m \in {\cal O}_F\,$. We will say
that \bfind{$(F|K, P)$ is uniformizable with respect to
$\zeta_1,\ldots,\zeta_m$} if there are
\pars
$\bullet$ \ a transcendence basis $T=\{t_1,\ldots,t_s\}\subset
{\cal O}_F$ of $F|K$ (may be empty),\par
$\bullet$ \ elements $\eta_1,\ldots,\eta_n\in {\cal O}_F\,$,
with $\zeta_1,\ldots,\zeta_m$ among them,\par
$\bullet$ \ polynomials $f_i(X_1,\ldots,X_n)\in {\cal O}_K
[t_1,\ldots,t_s,X_1,\ldots,X_n]$, $1\leq i\leq n$,
\sn
such that $F=K(t_1,\ldots,t_s,\eta_1,\ldots,\eta_n)$,
and\pars
{\bf (U1)} \ for $i<j$, $X_j$ does not occur in $f_i\,$,\par
{\bf (U2)} \ $f_i(\eta_1,\ldots,\eta_n)=0$ for $1\leq i\leq n$,\par
{\bf (U3)} \ $(\det J_f (\eta_1,\ldots,\eta_n))P \,\ne\, 0\;$.
\mn
Assertion {\bf (U1)} implies that $J_f$ is triangular. Assertion
{\bf (U3)} says that $f_1,\ldots,f_n$ and $\eta_1,\ldots,\eta_n$ satisfy
the assumption (\ref{hmhl}) of the Multidimensional Hensel's Lemma. By
the triangularity, this implies that for each $i$, $f_i$ and $\eta_i$
satisfy the hypothesis of Hensel's Lemma over the ground field
$K(t_1,\ldots,t_s,\eta_1,\ldots,\eta_{i-1})$.

We say that \bfind{$(F|K,P)$ is uniformizable} if it is
uniformizable with respect to {\it every} choice of finitely many
elements in ${\cal O}_F\,$.

Now assume in addition that $P$ is trivial on $K$. Then ${\cal O}_K=K$,
and the $P$-residues of the coefficients are obtained by just replacing
$t_j$ by $t_jP$, for $1\leq j\leq n$. Hence if we view the polynomials
$f_i$ as polynomials in the variables $Z_1,\ldots,Z_s,X_1,\ldots,X_n\,$,
then assertion {\bf (U3)} means that the Jacobian matrix at the point
$(t_1P,\ldots,t_sP, \eta_1P,\ldots,\eta_nP)$ has maximal rank. This
assertion says that on the variety defined over $K$ by the $f_i$
(having generic point $(t_1,\ldots,t_s,\eta_1,\ldots,\eta_n)$ and
function field $F$), the place $P$ is centered at the smooth point
\[(t_1P,\ldots,t_sP,\eta_1P,\ldots,\eta_nP)\;.\]

By uniformizing with respect to the $\zeta$'s, we obtain the following
important information: if we have already a model $V$ of $F|K$ with
generic point $(z_1,\ldots,z_k)$, where $z_1,\ldots,z_k\in {\cal O}
_F\,$, then we can choose our new model ${\cal V}$ of $F|K$ in such a
way that the local ring of the center of $P$ on ${\cal V}$ contains the
local ring of the center $(z_1P,\ldots,z_kP)$ of $P$ on $V$. For this,
we only have to let $z_1,\ldots,z_k$ appear among the $\zeta$'s. In
fact, Zariski proved the following \bfind{Local Uniformization Theorem}
(cf.\ [Z]):
\begin{theorem}                             \label{ZLUT}
Suppose that $F|K$ is a function field of characteristic 0, $P$
is a place of $F|K$, and $\zeta_1,\ldots,\zeta_m$ are elements of
${\cal O}_F\,$. Then $(F|K,P)$ is uniformizable with respect to
$\zeta_1,\ldots,\zeta_m\,$.
\end{theorem}

But the $\zeta$'s also play another role. Through their presence, the
above property becomes transitive (see [K6]):
\begin{theorem} \
{\bf (Transitivity of Relative Uniformization)} \
If $E|F$ is uniformizable with respect to $\zeta_1,\ldots,\zeta_m$ and
$F|K$ is uniformizable with respect to certain finitely many elements
derived from $E|F$ and $\zeta_1,\ldots,\zeta_m\,$, then $E|K$ is
uniformizable with respect to $\zeta_1,\ldots,\zeta_m\,$. In particular,
if $E|F$ and $F|K$ are uniformizable, then so is $E|K$.
\end{theorem}

As we do not require that $P$ is trivial on $K$, we call the
property defined above \bfind{relative (local) uniformization}. Its
transitivity enables us to build up our function field step by step by
extensions which admit relative uniformization. I give examples of
uniformizable extensions; for the proofs, see [K5], [K6], [K7].
\mn
{\bf I)} \
We consider a function field $F|K$ and a place $P$ of $F$ such that
$v_P^{ }K$ is a convex subgroup of $v_P^{ }F$. The latter always holds
if $P$ is trivial on $K$ since then, $v_P^{ }K=\{0\}$. We take elements
$x_1,\ldots,x_{\rho}$ in $F$ such that $v_P^{ }x_1,\ldots, v_P^{ }
x_{\rho}$ form a maximal set of rationally independent elements in
$v_P^{ } F$ modulo $v_P^{ }K$. Further, we take elements $y_1,\ldots,
y_{\tau}$ in $F$ such that $y_1P,\ldots, y_{\tau} P$ form a
transcendence basis of $FP$ over $KP$.
\begin{proposition}                         \label{prep}
In the described situation, $(K(x_1,\ldots,x_{\rho},y_1,\ldots,
y_{\tau}) |K,P)$ is uniformizable. More precisely, the
transcendence basis $T$ can be chosen of the form $\{x'_1,\ldots,
x'_{\rho}, y_1,\ldots, y_{\tau}\}$, where $x'_1,\ldots, x'_{\rho}\in
{\cal O}_{K(x_1,\ldots,x_{\rho})}$ and for some $c\in {\cal O}_K\,$,
$c\ne 0$, (with $c=1$ if $P$ is trivial on $K$), the elements
$cx'_1,\ldots,cx'_{\rho}$ generate the same multiplicative subgroup of
$K(x_1,\ldots,x_{\rho})^\times$ as $x_1,\ldots,x_{\rho}\,$.
\end{proposition}
\sn
The proof uses Theorem~\ref{prelBour}. It also uses the following lemma,
which was proved (but not explicitly stated) by Zariski in [Z] for
subgroups of $\R$, using the \bfind{Algorithm of Perron}. I leave it as
an easy exercise to you to prove the general case by induction on the
(finite!) rank of the ordered abelian group. An instant proof
of the lemma can be found in [EL] (Theorem 2.2), and I am sure
there are several more authors who reproved the lemma, not knowing about
Zariski's application.
\begin{lemma}                               \label{perron}
Let $G$ be a finitely generated ordered abelian group. Take any
positive elements $\alpha_1,\ldots,\alpha_\ell$ in $G$. Then there
exist positive elements $\gamma_1,\ldots,\gamma_{\rho}\in G$ such
that $G=\Z \gamma_1 \oplus\ldots\oplus\Z\gamma_{\rho}$ and every
$\alpha_i$ can be written as a sum $\sum_{j}^{} n_{ij}\gamma_j$ with
non-negative integers $n_{ij}\,$.
\end{lemma}

\mn
{\bf II)} \
The next result is based on Kaplansky's work [KA1]. First, we need:
\begin{lemma}                               \label{close}
Let $(K(z)|K,P)$ be an immediate transcendental extension. Assume
further that $(K,P)$ is separable-algebraically maximal or that
$(K(z),P)$ lies in the completion of $(K,P)$. Then for every polynomial
$f\in K[X]$, the value $v_P^{ }f(a)$ is fixed for all $a\in K$
sufficiently close to $z$. That is,
\begin{equation}                            \label{trat}
\begin{array}{c}
\forall f\in K[X]\;\exists\alpha\in v_P^{ }K\;\exists\beta\in
\{v_P^{ }(z-b)\mid b\in K\}\; \forall a\in K:\\[.3cm]
v_P^{ }(z-a) \geq\beta\,\Rightarrow\, v_P^{ }f(a)=\alpha\;.
\end{array}\;\;\;
\end{equation}
\end{lemma}
Kaplansky proves that if (\ref{trat}) does not hold, then there is a
proper immediate algebraic extension of $(K,P)$. If $(K(z),P)$ does not
lie in the completion of $(K,P)$, then this can be transformed into a
proper immediate separable-algebraic extension ([K1], [K2]; the proof
uses a variant of the Theorem on the Continuity of Roots). The existence
of such an extension is excluded if $(K,v)$ is separable-algebraically
maximal. If on the other hand we assume that $(K(z),P)$ lies in the
completion of $(K,P)$, then one can show that if $f$ does not satisfy
(\ref{trat}), then $v_P^{ }f(z)=\infty$. But this means that $f(z)=0$,
contradicting the assumption that $K(z)|K$ is transcendental. Note that
by Lemma~\ref{tame}, every separably tame and hence also every tame
field $(K,P)$ satisfies the assumption of Lemma~\ref{close}.

\pars
The following proposition is somewhat complementary to
Proposition~\ref{prep}.
\begin{proposition}                               \label{ist}
Let $(K(z)|K,P)$ be an immediate transcendental extension. If $z$
satisfies (\ref{trat}), then $(K(z)|K,P)$ is uniformizable.
\end{proposition}

\mn
{\bf III)} \ The proof of the following proposition is quite easy;
for the transcendental part, it uses Lemma~\ref{close} and
Proposition~\ref{ist}.
\begin{proposition}                               \label{comp1}
Every separable extension of a valued field within its
completion is uniformizable.
\end{proposition}

The henselization of a valued field $(K,P)$ is always a
separable-algebraic extension. If $(K,P)$ has rank 1, then moreover, the
henselization lies in the completion of $(K,P)$ (since in this case the
completion is henselian). Therefore, Proposition~\ref{comp1} yields:
\begin{corollary}                               \label{hens1}
Assume that $(K,P)$ has rank 1. If $(L|K,P)$ is a finite subextension of
the henselization of $(K,P)$, then it is uniformizable.
\end{corollary}

\mn
{\bf IV)} \ The following proposition is of interest in view of the
inertial generation of function fields with Abhyankar places.
Its proof is again quite easy.
\begin{proposition}                               \label{inert1}
Let $(K,P)$ be a henselian field and $(L|K,P)$ a finite
extension within the absolute inertia field of $(K,P)$. Then
$(L|K,P)$ is uniformizable.
\end{proposition}

\parm
After these positive results, I have to talk about a serious problem.
Throughout my work in the model theory of valued fields, my experience
was that henselizations are very nice extensions and do not harm at all.
Because of their universal property (Theorem~\ref{hensuniqemb}), they
behave well in embedding lemmas. Unfortunately, for the problem of local
uniformization, this seems to be entirely different. While we have
relative uniformization if the henselization lies in the completion, we
get problems if this is not the case. And for a rank greater than 1, we
cannot expect in general that the henselization lies in the completion
(since the completion will in general not be henselian).
This leads to the following important open problem:
\mn
{\bf Open Problem 3:} \ Prove or disprove: every finite subextension in
the henselization of a valued field is uniformizable.
\mn
This is a special case of a slightly more general problem, however.
Again in view of inertial generation, we would like to know the
following.
\mn
{\bf Open Problem 4:} \ Assume that $(K,v)$ is a field which is
{\it not\/} henselian. Prove or disprove: every finite subextension
in the absolute ramification field of $(K,v)$ is uniformizable.
\mn
The obstruction is the following. Assume that $(L|K,v)$ is a finite
subextension in the absolute ramification field of $(K,v)$. Suppose
there is an intermediate field $L'$ such that $Lv=L'v$ and $[L':K]=
[L'v:Kv]$. The former yields that $L$ lies in the henselization of
$L'$. The latter yields that $(L'|K,v)$ admits relative local
uniformization, which can be proved in exactly the same way as
Proposition~\ref{inert1}, although $(K,v)$ need not be henselian. So
by the transitivity of relative uniformization, the problem would be
reduced to that of subextensions within the henselization. But such
intermediate fields $L'$ may not exist!

A closer look reveals that this problem is also the kernel of our
problem about subextensions within the henselization. Indeed, if we
have rank $>1$, then $P$ is the composition $P=Q\ovl{Q}$ of two
non-trivial places. Advanced ramification theory shows that if
$(L|K,P)$ is a subextension of the henselization of $(K,P)$, then
$(L|K,Q)$ is a subextension of the absolute inertia field of $(K,Q)$.
If we could split everything up by intermediate fields, then we could
reduce to extensions like the above $(L'|K,v)$, and extensions within
completions; this would solve our problem. But the necessary
intermediate fields may only exist after enlarging the extension $L|K$.
Nevertheless, in [K6] I prove a weak form of relative uniformization for
finite extensions within the absolute ramification field.

\parm
Let us see what we get and what we do not get from the above positive
results.

%
%
\section{Local uniformization for Abhyankar places}
Using the transitivity of relative uniformization, we can combine
Proposition~\ref{prep} with Corollary~\ref{hens1} to obtain:
\begin{theorem}                             \label{Abh1}
Take an Abhyankar place of $F|K$ such that $(F,P)$ has rank 1. If
$FP=K$ or $FP|K$ is a rational function field, then $(F|K,P)$
is uniformizable.
\end{theorem}

This works since $(F|K,P)$ is henselian generated. But as soon as $FP|K$
is not a rational function field, $(F|K,P)$ will not be henselian
generated. Then we may run into the problems described in the last
section. For the time being, our only chance is to accept to extend
$F$. Then we can prove ([K6], [K7]):
\begin{theorem}                         \label{Abh2}
Assume that $P$ is an Abhyankar place of $F|K$ and take
elements $\zeta_1,\ldots,
\zeta_m \in {\cal O}_F\,$. Then there is a finite purely inseparable
extension ${\cal K}|K$, a finite separable extension ${\cal F}|F.
{\cal K}$, and an extension of $P$ from $F$ to ${\cal F}$ such that
$({\cal F}|{\cal K},P)$ is uniformizable with respect to
$\zeta_1,\ldots,\zeta_m\,$. In addition, we have:
\n
1) \ If $(F,P)$ has rank 1, then ${\cal F}$ can be obtained from
$F.{\cal K}$ by a Galois extension.
\n
2) \ If $(F,P)$ has rank $r>1$, then ${\cal F}$ can be obtained from
$F.{\cal K}$ by a sequence of at most $r-1$ Galois extensions if $FP|K$
is algebraic, or at most $r$ Galois extensions otherwise.
\n
3) \ Alternatively, ${\cal F}$ can always be chosen to lie in the
henselization of $(F.{\cal K},P)$.
\n
4) \ If $FP|K$ is separable, then in all cases we can choose
${\cal K}=K$.
\end{theorem}

Unfortunately, ``Galois extension'' and ``lie in the henselization''
are mutually incompatible. Indeed, the normal hull of a subextension of
the henselization will in general not again lie in the henselization.

\pars
We could prove local uniformization for all Abhyankar places of $F|K$
with separable $FP|K$, if we would know a positive answer to the
following problem:
\mn
{\bf Open Problem 5:} \ Take any Abhyankar place of $F|K$. Is it
possible to choose the generating transcendence basis $T$ for the
inertial generation of $(F|K,P)$ in such a way that the extension
$(F|K(T),P)$ is uniformizable?
\mn

Here is an even more daring idea. Since we have seen that we have
problems with the henselization, why don't we try to avoid it?
\mn
{\bf Open Problem 6:} \ Take any Abhyankar place of $F|K$. Is it
possible to choose the generating transcendence basis $T$ for the
inertial generation of $(F|K,P)$ in such a way that the extension
of $P$ from $K(T)$ to $F$ is unique?
\mn
In that case, general valuation theory shows that the extension $F|K(T)$
is linearly disjoint from the extension $K(T)^h|K(T)$. If $T$ has that
property, we would say that \bfind{$T$ has the uniqueness property
for $F|K$} (this is a generalization of the definition given in
Section~\ref{sectGST}). But I warn you: this problem seems to be very
hard. It is already non-trivial to prove the existence of elements
with the uniqueness property in case of transcendence degree 1.
For arbitrary transcendence degree, the necessary algebraic geometry
is not in sight. But perhaps there is some connection with local
uniformization or resolution of singularities?

\pars
A weak form of local uniformization, without extending the function
field, can be proved for all Abhyankar places in the case of perfect
ground fields ([K6]). Also T.~Urabe [U], building on Abhyankar's
original approach, has a comparable result for the special case of
places of maximal rank.

%
%
\section{Non-Abhyankar places and the Henselian Rationality of immediate
function fields}                            \label{sectnon}
What can we do if the place $P$ of $F|K$ is {\it not} an Abhyankar
place? Still, the place may be nice. Assume for instance that $v_P^{ }F$
is finitely generated and $FP=K$. Then we can choose $x_1,\ldots,x_\rho$
such that $v_P^{ }=\Z v_P^{ }x_1\oplus\ldots\oplus\Z v_P^{ } x_\rho$,
and set $F_0:=K(x_1,\ldots,x_\rho)$. Consequently, $(F|F_0,v_P^{ })$ is
an immediate extension. If $P$ is not an Abhyankar place, then this
extension is not algebraic. But we have already stated two tools to
treat transcendental immediate extensions, namely Proposition~\ref{ist}
and Proposition~\ref{comp1}.

Let us first apply Proposition~\ref{comp1}. If $(F,v_P^{ })$ is a
separable extension within the completion of $(F_0,v_P^{ })$, then by this
proposition, the transitivity and Proposition~\ref{prep}, we find that
$(F|K,P)$ is uniformizable. Do we need the assumption that the extension
be separable? The answer is: this is automatically true. Indeed, by the
Generalized Stability Theorem~\ref{ai}, $(F_0,v_P^{ })$ is a defectless
field; thus, our assertion follows from the following lemma:
\begin{lemma}
If $(L,v)$ is a defectless field and $(L'|L,v)$ is an immediate
extension, then $L'|L$ is separable.
\end{lemma}
\begin{proof}
We have to show that $L'|L$ is linearly disjoint from every purely
inseparable finite extension $E|L$. As the extension of the valuation
from $L$ to $E$ is unique by Corollary~\ref{cae}, this implies that
$[E:L]=(vE:vL) [Ev:Lv]$. Now we consider the compositum $L'.E$ (with the
unique extension of the valuation from $L'$ to the purely inseparable
extension $L'.E$). Since $vE\subseteq v(L'.E)$, $vL'=vL\,$, $Ev
\subseteq (L'.E)v$ and $L'v=Lv$, we have that $(v(L'.E): vL')\geq
(vE:vL)$ and $[(L'.E)v:L'v]\geq [Ev:Lv]$. Hence,
\begin{eqnarray*}
[L'.E:L'] & \geq & (v(L'.E):vL')[(L'.E)v:L'v]\\
 & \geq & (vE:vL) [Ev:Lv] \>=\> [E:L]\>\geq\>[L'.E:L']\;.
\end{eqnarray*}
Thus, equality holds everywhere, showing that $L'|L$ is linearly
disjoint from $E|L\,$.
\end{proof}

In view of this result, Proposition~\ref{comp1} and the transitivity
prove:
\begin{theorem}
The assertions of Theorem~\ref{Abh1} and Theorem~\ref{Abh2} remain
true if $(F,P)$ is replaced by a finitely generated extension
within its completion.
\end{theorem}

If $P$ is a discrete rational place of $F|K$, then we only have to
choose $t\in F$ such that $v_P^{ }t$ is the smallest positive element in
$v_P^{ }F \isom\Z$; this will imply that $v_P^{ }F=\Z v_P^{ }t$. The
immediate extension $(F,P)$ of $(K(t),P)$ will automatically lie in the
completion of $(K(t),P)$ (since the completion is of the form $K((t))$
and hence maximal, and one can easily show that the completion is unique
up to isomorphism). So we obtain:
\begin{theorem}                             \label{MTdrlu}
Every discrete rational place is uniformizable.
\end{theorem}

\pars
Now let us turn to the general case. Not every immediate extension lies
in the completion, not even in rank 1.
\begin{example}
Take the valuation $v$ on the rational function field $K(x_1,x_2)$ such
that $vK(x_1,x_2)=\Z+r\Z$ with $r\in\R\setminus\Q$. We can view
$K(x_1,x_2)$ as a subfield of the power series field $K((\Z+r\Z))$,
with $x_1=t$ and $x_2=t^r$. In $\Z+r\Z$, we can choose a monotonically
increasing sequence $r_i$ converging to 0 from below. Then we take the
element $z= \sum_{i=1}^{\infty} t^{r_i}\in K((\Z+r\Z))$. It does not lie
in the completion of $(K(x_1,x_2),v)$. Nevertheless, the extension
$(K(x_1,x_2,z)|K(x_1,x_2),v)$ is immediate. I leave it to
you to show that $z$ is transcendental over $K(x_1,x_2)$.
\end{example}

So let us look at the case of an immediate extension $(F_0(z)|F_0,P)$
which does not lie in the completion of $(F_0,P)$. Then to apply
Proposition~\ref{ist}, we need to know that $z$ satisfies (\ref{trat}).
By Lemma~\ref{close}, this would hold if $(F_0,P)$ were
separable-algebraically maximal. But as the rational function field
$F_0$ is not henselian unless $P$ is trivial, it will admit its
henselization as a proper immediate separable-algebraic extension. The
only way out at this point is to pass to the immediate extension
$(F_0^h(z)|F_0^h,P)$. There, we can apply Proposition~\ref{ist} since by
the Generalized Stability Theorem, $(F_0^h,P)$ is a defectless field and
thus is algebraically maximal. But now we have extended our function
field $F$. (In the end, we will only need a finite subextension since
the statement of local uniformization only talks about finitely many
elements.) We will return to this aspect below.

Beforehand, let us think about two further problems. First, if our
extension $F|F_0$ is of transcendence degree $>1$, how do we carry on by
induction? As the Generalized Stability Theorem does not apply any more
in this situation, we do not know whether $(F_0(z),P)$ is a defectless
field (and in fact, in general it will not be). Then we have to enlarge
$F_0(z)$ even more to achieve the next induction step; in positive
characteristic, the henselization will be too small for this purpose.

Second, if for instance $F|F_0(z)$ is a proper algebraic extension, then
does it have relative uniformization? The only answers we have apply to
the case where $(F,P)$ lies in the henselization of $(F_0(z),P)\>$ (and
even in this case our discussion has shown a bunch of problems). If
$(F,P)$ does not lie in the henselization, then we know nothing. Observe
that since the extension $(F|F_0(z),P)$ is finite and immediate, $(F,P)$
does not lie in the henselization if and only if $(F^h|F_0(z)^h,P)$ has
non-trivial defect (by Theorem~\ref{immhens} and Corollary~\ref{hidef}).

So the question arises: how can we avoid the defect in the case of
immediate extensions? The answer is a theorem that I proved in [K1]
(cf.\ [K2]). As for the Generalized Stability Theorem, the proof
uses ramification theory and the deduction of normal forms for
Artin-Schreier extensions. It also uses significantly a theory
of immediate extensions which builds on Kaplansky's paper [KA1].
\begin{theorem}                \label{stt3}
{\bf (Henselian Rationality of Immediate Function Fields)} \
Let $(K,P)$ be a tame field and $(F|K,P)$ an immediate function field
of trans\-cendence degree 1. Then
\begin{equation}
\mbox{there is $x\in F$ such that }\; (F^h,P)\,=\,(K(x)^h,P)\;,
\end{equation}
that is, $(F|K,P)$ is henselian generated. The same holds over a
separably tame field $(K,P)$ if in addition $F|K$ is separable.
\end{theorem}
Since the assertion says that $F^h$ is equal to the henselization of a
rational function field, we also call $F$ \bfind{henselian rational}.
For valued fields of residue characteristic 0, the assertion is a direct
consequence of the fact that every such field is defectless. Indeed,
take any $x\in F\setminus K$. Then $K(x)|K$ cannot be algebraic since
otherwise, $(K(x)|K,P)$ would be a proper immediate algebraic extension
of the tame field $(K,P)$, a contradiction to Lemma~\ref{tame}. Hence,
$F|K(x)$ is algebraic and immediate. Therefore, $(F^h|K(x)^h,P)$
is algebraic and immediate too. But since it cannot have a non-trivial
defect, it must be trivial. This proves that $(F,P)\subset (K(x)^h,P)$.
In contrast to this, in the case of positive residue characteristic only
a very carefully chosen $x\in F\setminus K$ will do the job.

\pars
To illustrate the use of Theorem~\ref{stt3} in the model theory of
valued fields, we give an example which is ``complementary'' to
Example~\ref{mtGGST}, treating the case of immediate extensions:
\begin{example}                             \label{mtHR}
Suppose that $(K,v)$ is a tame field and that $(L|K,v)$ is an immediate
extension. Then the conditions $vK\ec vL$ and $Kv\ec Lv$ are trivially
satisfied. So do we have that $(K,v)\ec (L,v)$? Using the theory of tame
fields, in particular the crucial Lemma~\ref{trac}, one reduces the proof
to the case of transcendence degree 1. So we have to prove an embedding
lemma for immediate function fields $(F|K,v)$ of transcendence degree 1
over tame fields. If we take any $x\in F\setminus K$ then it will
satisfy (\ref{trat}) because $(K,v)$ is tame and thus also
algebraically maximal. Then with the help of Theorem~2 of [KA1]
and the saturation of $(K,v)^*$, we can find
an embedding of $(K(x),v)$ in $(K,v)^*$. But how do we carry on?
We know that $(F|K(x),v)$ is immediate, but this does not mean that
$(F,v)$ lies in the henselization of $(K(x),v)$. But if it does, we can
just use the universal property of the henselization
(Theorem~\ref{hensuniqemb}). Indeed, being a tame field, $(K,v)$ is
henselian. Since $(K,v)^*$ is an elementary extension of $(K,v)$ (in the
language of valued fields), it is also henselian. Hence if $(K(x),v)$
embeds in $(K,v)^*$, then this embedding can be extended to an embedding
of $(K(x)^h,v)$ in $(K,v)^*$. This induces the desired embedding of $F$.

If $F$ does not lie in the henselization of $(K(x),v)$, then
$(F^h|K(x)^h,v)$ has non-trivial defect, and we have no
clue how the embedding of $K(x)^h$ could be extended to an embedding of
$F^h$. Again, our enemy is the defect, and we have to avoid it.
Now this can be done by Theorem~\ref{stt3}. It tells us that there is
some $x\in F$ such that $(F,v)$ lies in the henselization of $(K(x),v)$.
So we have proved:
\begin{theorem}                             \label{ecimm}
Suppose that $(K,v)$ is a tame field and that $(L|K,v)$ is an immediate
extension. Then $(K,v)\ec (L,v)$.
\end{theorem}

Given an extension of tame fields of finite transcendence degree, then
by use of Lemma~\ref{trac}, one can separate it into an extension
without transcendence defect and an immediate extension. Both can be
treated separately by Theorem~\ref{ecwtd} and Theorem~\ref{ecimm}. As
$\ec$ is transitive, this proves (cf.\ [K1], [K2]):
\begin{theorem}                             \label{AKEtame}
The Ax--Kochen--Ershov principle (\ref{AKEec}) holds for every
extension $(L|K,v)$ of tame fields.
\end{theorem}
\end{example}

\pars
Let us return to our problem of local uniformization.
So far, we have worked with the assumption that we can find a
subfunction field $F_0$ in $F$ such that the restriction of $P$ to $F_0$
is an Abhyankar place and $(F|F_0,P)$ is immediate. But it is not always
possible to achieve the latter. For example, take $F$ to be the rational
function field $K(x_1,x_2)$ and $P$ such that $FP=K$ and $v_P^{ }F$ is a
not finitely generated subgroup of $\Q$; we will construct such a place
$P$ in the next section. But for any $F_0$ on which $P$ is an Abhyankar
place, $v_P^{ }F_0$ is finitely generated, so we will always have that
$v_P^{ }F\ne v_P^{ }F_0\,$.

In this situation, passing to henselizations may help again. Given an
arbitrary place $P$ of $F|K$, we choose an Abhyankar subfunction field
as in (\ref{Asff}). We have that $v_P^{ }F/v_P^{ }F_0$ is a torsion
group and that $FP|F_0P$ is algebraic. Take $F_1$ to be the relative
algebraic closure of $F_0$ in $F^h$. If $\chara FP=\chara K=0$, then by
Lemma~\ref{HLrel1} and Lemma~\ref{HLrel2}, $(F^h|F_1,P)$ is an immediate
extension; so we succeeded again in reducing to the case of immediate
extensions. But if $\chara FP=\chara K=p>0$, then we only know that
$v_P^{ }F^h/v_P^{ }F_1$ is a $p$-group and that $F^hP|F_1P$ is purely
inseparable. So in this case, passing to henselizations is not enough.
But we obtain an immediate extension if we replace $F^h$ and $F_1$ by
their perfect hulls. In fact, to make all of our tools work, we have to
take even bigger extensions. Namely, we have to pass to smallest
algebraic extensions which are tame or at least separably tame fields.
But these extensions still have nice properties; we will talk about them
in Section~\ref{sectblu}.

With this approach, one can deduce the following theorems from
the Generalized Stability Theorem, the Henselian Rationality of Immediate
Function Fields, the results described in Section~\ref{sectrlu}, and the
transitivity of relative uniformization:
\begin{theorem}                             \label{MT1}
Let $F|K$ be a function field of arbitrary characteristic and $P$
a place of $F|K$. Take any elements $\zeta_1,\ldots,\zeta_m\in
{\cal O}_P\,$. Then there exist a finite extension ${\cal F}$ of $F$,
an extension of $P$ to ${\cal F}$, and a finite purely inseparable
extension ${\cal K}$ of $K$ within ${\cal F}$ such that $({\cal F}
|{\cal K},P)$ is uniformizable with respect to $\zeta_1,\ldots,
\zeta_m\,$.
\end{theorem}
\begin{theorem}                             \label{MT1add}
The extension ${\cal F}|F$ can always be chosen to be normal.
\end{theorem}

See [K6] for the proof. These theorems also follow from the results of
Johan de Jong ([dJ]; cf.\ also [AO], [OO]). So this should be an
interesting question:
\mn
{\bf Open Problem 7:} \ Is it possible to recognize counterparts of the
Generalized Stability Theorem and the Henselian Rationality of Immediate
Function Fields in the theory of semi-stable reduction, or in any other
part of de Jong's proof of desingularization by alteration?
\mn
The advantage of proving the above theorems by the described
valuation theoretical approach is that we get additional information
about the extension ${\cal F}|F$. We have already seen in this and the
previous section that in certain cases we do not need an extension,
i.e., we have local uniformization already for $(F|K,P)$. In other
cases, we can obtain ${\cal F}$ from $F$ by one or a tower of Galois
extensions. We will see in Section~\ref{sectblu} that in general, we can
choose ${\cal F}|F.{\cal K}$ to be a separable (but not Galois)
extension with additional information about the related extensions of
value group and residue field.

\parm
A little bit of horror makes an excursion even more interesting. So
let's watch out for bad places.

%
%
\section{Bad places}                        \label{sectbad}
In this section we will show that there are places of function fields
$F|K$ whose value group or residue field are not finitely generated.
By combining the methods you can construct examples where both is the
case. The following two examples can already be found in [ZS], Chapter
VI, \S 15. But our approach (using Hensel's Lemma) is somewhat easier
and more conceptual.

\begin{example}
We construct a place on the rational function field $K(x_1,x_2)|K$
whose value group $G\subset \Q$ is not finitely generated, assuming
that the order of every element in $G/\Z$ is prime to $\chara K$. To
this end, we just find a suitable embedding of $K(x_1,x_2)$ in $K((G))$.
We do this by setting $S:=\{n\in\N\mid 1/n\in G\}$ and
\begin{equation}
x_1\>:=\>t \;\;\;\mbox{ and }\;\;\; x_2:=\sum_{n\in S}^{}t^{-1/n}\;.
\end{equation}
Further, take the valuation $v$ on $K(x_1,x_2)$ to be the restriction of
the canonical valuation $v$ of $K((G))$. We wish to show that $1/S
\subseteq vK(x_1,x_2)$, so that $G\subset vK(x_1,x_2)$. Since
$(K(x_1,x_2),v) \subset (K((G)),v)$, it follows that $G=vK(x_1,x_2)$.
If $x_2$ were algebraic over $K(x_1)$, we would know by
Corollary~\ref{fingentb} that $vK(x_1,x_2)$ is finitely generated. Hence
if it is not, then $x_2$ must be transcendental over $K(x_1)$, so that
$K(x_1,x_2)$ is indeed the rational function field over $K$ in two
variables.

Suppose that $\chara K=0$; then we can get $G=\Q$. Also in positive
characteristic one can define the valuation in such a way that the value
group becomes $\Q$; since then we have to deal with inseparability, our
construction has to be refined slightly, which we will not do here.

Now let us prove our assertion.
We take $(L,v)$ to be the henselization of $(K(x_1,x_2),v)$.
We are going to show that $t^{1/n}\in L$ for all $n\in S$. Suppose we
have shown this for all $n<k$, where $k\in S$ (we can assume that
$k>1$). Then also $s_k:=\displaystyle\sum_{n\in S, n<k} t^{-1/n}\in L$.
We write
\begin{equation}
x_2-s_k\>=\>\sum_{n\in S, n\geq k} t^{-1/n} \>=\> t^{-1/k}(1+c)
\end{equation}
where $c\in K((G))$ with $vc>0$. Hence, $1+c$ is a 1-unit. We have that
$(1+c)^k=t(x_2-s_k)^k \in L$. On the other hand, $(1+c)^kv=((1+c)v)^k
=1^k=1$, which shows that $(1+c)^k$ is again a 1-unit. Since $k\in S$
we know that $\chara LP=\chara K$ does not divide $k$. Hence by
Lemma~\ref{1-unit}, $1+c\in L$. This proves that
$t^{1/k}=(1+c)(x_2-s_k)^{-1}\in L$.

We have now proved that $t^{1/k}\in L$ for all $k\in S$. Hence, $1/k=
vt^{1/k}\in vL$ for all $k\in S$. But since the henselization is an
immediate extension, we know that $vL=vK(x_1,x_2)$, so we have proved
that $1/S\subset vK(x_1,x_2)$.
\end{example}

\begin{example}
We take a field $K$ for which the separable-algebraic closure $K\sep$ is
an infinite extension (i.e., $K$ is neither separable-algebraically
closed nor real closed). We construct a place of the rational function
field $K(x_1,x_2)|K$ whose residue field is not finitely generated. We
choose a sequence $a_n\,$, $n\in\N$ of elements which are
separable-algebraic over $K$ of degree at least $n$. We define an
embedding of $K(x_1,x_2)$ in $K\sep ((t))$ by setting
\begin{equation}
x_1\>:=\>t \;\;\;\mbox{ and }\;\;\; x_2:=\sum_{n\in \N} a_n t^n\;.
\end{equation}
Further, we take the valuation $v$ on $K(x_1,x_2)$ to be the restriction
of the valuation of $K\sep ((t))$. We wish to show that $a_n\in
K(x_1,x_2)v$ for all $n\in\N$, so that $K(x_1,x_2)v|K$ cannot be
finitely generated. If $x_2$ were algebraic over $K(x_1)$, we would know
by Corollary~\ref{fingentb} that $K(x_1,x_2)v|K$ is finitely generated.
So if it is not, then $x_2$ must be transcendental over $K(x_1)$, so
that $K(x_1,x_2)$ is indeed the rational function field over $K$ in two
variables. By a modification of the construction, one can also generate
infinite inseparable extensions of $K$. If $K$ is countable, one can
generate every algebraic extension of $K$ as a residue field of
$K(x_1,x_2)$.

We take again $(L,v)$ to be the henselization of $(K(x_1,x_2),v)$.
We are going to show that $a_n\in L$ for all $n\in\N$. Suppose we
have shown this for all $n<k$, where $k\in\N$.
Then also $s_k:=\sum_{n=1}^{k-1} a_n t^n\in L$. We write
\begin{equation}
\frac{x_2-s_k}{t^k}\>=\>\frac{1}{t^k}\sum_{n=k}^{\infty} a_k t^k\>=\>
a_k (1+c)
\end{equation}
where $c\in K\sep ((t))$ with $vc>0$. Take $f\in K[X]$ to be the minimal
polynomial of $a_k$ over $K$ and note that $f=fv$. Since $a_k\in K\sep$,
we know that $a_k$ is a simple root of $f$. On the other hand,
$a_k=a_k(1+c)v\in Lv$. Hence by Hensel's Lemma (Simple Root Version)
there is a root $a$ of $f$ in $L$ such that $av=a_k$. As we may assume
that the place associated with $v$ is the identity on $K$, this will
give us that $a=a_k\,$; so $a_k\in L$.

We have now proved that $a_n\in L$ for all $n\in\N$. Hence,
$a_n\in Lv=K(x_1,x_2)v$ for all $n\in\N$.
\end{example}

%
%
\section{The role of the transcendence basis and the
dimension}                                  \label{sectrtb}
In our approach described in Section~\ref{sectnon}, we have obtained
the subfunction field $F_0$ on which the restriction of $P$ is an
Abhyankar place by choosing the elements $x_1,\ldots,x_\rho,
y_1,\ldots,y_\tau\,$. But then we have made no effort to improve our
choice. With this ``stiff'' approach (which in fact gives additional
information), one can prove Theorem~\ref{MT1}, but it can be shown that
in general one cannot get local uniformization without an extension of
the function field. I want to show why not. The following example is
particularly interesting since it is also a key example in the model
theory of valued fields of positive characteristic (cf.\
Section~\ref{sectfpt}).
\begin{example}                             \label{exampnhr}
We denote by $\F_p$ the field with $p$ elements. We consider the
following function field of transcendence degree 3 over $\F_p\,$:
\begin{equation}
F\>=\>\F_p(x_1,x_2,y,z)\;\;\;\mbox{ with }\; z^p-z\>=\>x_1-x_2y^p\;.
\end{equation}
Since $x_1\in \F_p(x_2,y,z)$, $F$ is a rational function field. However,
in [K1] (cf.\ also [K2], [K4]) we have shown that there is a rational
place $P$ of $F|\F_p$ such that
$v_P^{ }F=\Z v_P^{ } x_1\times\Z v_P^{ }x_2\,$ (ordered
lexicographically) and that the valued function field $(F|\F_p(x_1,x_2),
P)$ is not henselian generated. It follows that $F$ cannot lie in the
henselization of $(\F_p(t_1,t_2,t_3),P)$ if $t_1,t_2\in F$ are algebraic
over $\F_p(x_1,x_2)$ (and hence lie in $\F_p(x_1,x_2)$ since this is
relatively algebraically closed in $F$). Therefore, Theorem~\ref{MT5}
shows that $(F|\F_p,P)$ admits no local uniformization with $t_1,t_2$
algebraic over $K(x_1,x_2)$. A function field having such a local
uniformization must have degree at least $p$ over $F$. And indeed,
degree $p$ suffices, as
\[(F(x_2^{1/p})|\F_p(x_1,x_2^{1/p}),P)\]
is a rational function field. (It is an interesting fact that there is
also a Galois extension of $\F_p(x_1,x_2)$ of degree $p$ such that the
function field becomes henselian generated.)
\end{example}

The proof that $(F|\F_p(x_1,x_2), P)$ is not henselian generated is
based on showing that $(\F_p(x_1,x_2), P)$ is not existentially closed
in $(F,P)$. I have not found an algebraic proof.
\mn
{\bf Open Problem 8:} \ Develop a method to prove algebraically that a
given valued function field $(F|K,v)$ ($v$ not necessarily trivial on
$K$) is {\it not} henselian generated or inertially generated.

A variant of the example (cf.\ [K7]) shows: {\it There are immediate
transcendental extensions of valued fields which are not uniformizable.}
The example teaches us that relative uniformization will not always hold
without an extension of the function field. Hence, in general we will
have to optimize our choice of the transcendence basis for $F_0$ or even
for $F$ in order to obtain local uniformization for $(F|K,P)$. Given a
transcendence basis $T$ of $F$, it is easy to measure how far $F$ is
from lying in the absolute inertia field $K(T)^i$ of $(K(T),P)$: we just
have to take the degree
\begin{equation}
{\rm ig}(F,T):=[F.K(T)^i:K(T)^i]\;.
\end{equation}
This raises the problem:
\mn
{\bf Open Problem 9:} \ Develop a method to change $T$ in such a way
that ${\rm ig}(F,T)$ decreases.
\mn
In our above example, this is very easy since $F|K$ is actually a
rational function field. But one can modify the example in such a way
that $F|K$ is not rational. Instead of doing this for the above example,
let us look at a slightly simpler example, which will also show that
already a valued rational function field can have an immediate extension
of degree $p$ with defect $p$.

\begin{example}
We take an arbitrary field $K$ of characteristic $p>0$ and work in the
power series field $K((\frac{1}{p^{\infty}}\Z))$ with its canonical
valuation $v$. Recall that $\frac{1}{p^{\infty}}\Z$ is the
$p$-divisible hull $\{m/p^n\mid m\in\Z,\,n\in\N\}$ of $\Z$. We have
that $K((t))\subset K((\frac{1}{p^{\infty}}\Z))$. For every $i\in\N$,
we set $\nu_i:=\sum_{j=1}^{i} j$, and we define:
\begin{equation}
z\>:=\>\sum_{i=1}^{\infty} t^{p^{\nu_i}-p^{-\nu_i}}
\>\in\> K\left(\left(\frac{1}{p^{\infty}}\Z\right)\right)\;.
\end{equation}
We show that $(K((\frac{1}{p^{\infty}}\Z)),v)$ is an
immediate extension of $(K(t,z),v)$. Since both fields have residue
field $K$, we only have to show that $\frac{1}{p^{\infty}}\Z
\subseteq vK(t,z)$. For $k\in\N$, we compute:
\begin{eqnarray*}
z^{p^{\nu_k}}\,-\,\sum_{i=1}^{k} t^{p^{\nu_k+\nu_i}-p^{\nu_k-\nu_i}}
& = & \sum_{i=1}^{\infty}\left(t^{p^{\nu_i}-p^{-\nu_i}}
\right)^{p^{\nu_k}} \,-\,\sum_{i=1}^{k}\left(t^{p^{\nu_i}-
p^{-\nu_i}}\right)^{p^{\nu_k}} \\
& = & \sum_{i=k+1}^{\infty} t^{p^{\nu_k+\nu_i}-p^{\nu_k-\nu_i}}\>=\>
t^{p^{\nu_k+\nu_{k+1}}-p^{\nu_k-\nu_{k+1}}} + \ldots
\end{eqnarray*}
So
\begin{equation}
vt^{p^{\nu_k+\nu_{k+1}}}\left(z^{p^{\nu_k}}\,-\,\sum_{i=1}^{k}
t^{p^{\nu_k+\nu_i}-p^{\nu_k-\nu_i}}\right)^{-1}\>=\>p^{\nu_k-\nu_{k+1}}
\>=\> \frac{1}{p^{k+1}}\;.
\end{equation}
As the element on the left hand side is in $K(t,z)$, this shows
that $p^{-k}\Z\subset vK(t,z)$ for every $k\in\N$. Consequently,
$\frac{1}{p^{\infty}}\Z \subseteq vK(t,z)$, as desired. This also proves
that $z$ is transcendental over $K(t)$ since otherwise, $(vK(t,z):\Z)$
would be finite.

From Section~\ref{sectdef} we know that $\vartheta=\sum_{i\in\N}
t^{-1/p^i} \in K((\Q))$ is a root of the Artin--Schreier polynomial
$X^p-X-1/t$. We see that this power series already lies in the subfield
$K((\frac{1}{p^{\infty}}\Z))$ of $K((\Q))$. Hence, $(K(t,z,\vartheta)|
K(t,z),v)$ is a subextension of $(K((\frac{1}{p^{\infty}}\Z))|K(t,z),
v)$ and thus, it is immediate too. In order to show that it has
non-trivial defect, we have to show that it has a unique extension
of the valuation, or equivalently, that it is linearly disjoint from
the henselization of $(K(t,z),v)$. Since it is a Galois extension of
prime degree, it suffices to show that it does not lie in this
henselization.

\pars
We take the subfield $K(t^{1/p^k}\mid k\in\N)$ of
$K((\frac{1}{p^{\infty}}\Z))$. By definition, $z$ lies in the completion
of $(K(t^{1/p^k}\mid k\in\N),v)\>$ (since the values of the summands
form a sequence which is cofinal in the value group
$\frac{1}{p^{\infty}}\Z$). Since this value group is archimedean, that
is, $v$ has rank 1, we know that the henselization of $(K(t,z),v)$
lies in the completion of $(K(t,z),v)$, which by what we have just shown
lies in the completion of $(K(t^{1/p^k}\mid k\in\N),v)$. On the
other hand, we have seen in Section~\ref{sectdef} that $\vartheta$
does not lie in the completion of $(K(t^{1/p^k}\mid k\in\N),v)$. Hence,
it does not lie in the henselization of $(K(t,z),v)$.

\pars
We have now shown that the function field $K(t,z,\vartheta)|K$ admits
a place with a value group which is not finitely generated and such that
the extension $(K(t,z,\vartheta)| K(t,z),v)$ is immediate of degree $p$
and defect $p$. Now you will point out that our function field
$K(t,z,\vartheta)$ is again rational: since $1/t= \vartheta^p-
\vartheta$, we have that $K(t,z,\vartheta)=K(z,\vartheta)$. So let's
change something. We take a polynomial $f(t)\in K[t]$ and note that
$vf(t)\geq 0$. Now we replace $\vartheta$ by a root $\vartheta_f$ of the
polynomial $X^p-X-(\frac{1}{t}+f(t))$. It can be shown that the new
extension $(K(t,z,\vartheta_f)|K(t,z),v)$ will again be immediate of
degree $p$ and defect $p$. In fact, this is obvious if we choose $f$
without constant term. In that case, $vf(t)>0$ and we know from
Example~\ref{exampHL1} that the polynomial $X^p-X-f(t)$ has a root
in $K(t,z)^h$. By the additivity of $X^p-X$ it follows that the two
polynomials $X^p-X-\frac{1}{t}$ and $X^p-X-(\frac{1}{t}+f(t))$
define the same extension of degree $p$ and defect $p$ over $K(t,z)^h$.
Consequently, also $(K(t,z,\vartheta_f)|K(t,z),v)$ must be immediate of
degree $p$ and defect $p$.

Now we have that
\begin{equation}
\vartheta_f^p-\vartheta_f\>=\>\frac{1}{t}+f(t)\>=\>
\frac{1+tf(t)}{t}\;.
\end{equation}
So the minimal polynomial of $t$ over $K(z,\vartheta_f)$ will be
\begin{equation}
Xf(X)\,-\,(\vartheta_f^p-\vartheta_f)X\,+\,1\;.
\end{equation}
The transition from the representation $F=K(t,z,\vartheta_f)$ with
$\vartheta_f$ algebraic over $K(t,z)$ to the representation
$F=K(z,\vartheta_f,t)$ with $t$ algebraic over $K(z,\vartheta_f)$
may be called \bfind{Artin--Schreier inversion}. With a suitable choice
of $f$, the function field $F=K(t,z,\vartheta_f) |K$ will not be
rational. However, whatever choice of $f$ I computed, I found that after
a little Artin--Schreier surgery on $X^p-X-(\frac{1}{t}+f(t))$ (which
replaces $f$ by a better polynomial), Artin--Schreier inversion will
yield a tame extension $(K(z,\vartheta_f,t)|K(z,\vartheta_f),v)$. So at
least we got rid of the defect, probably even of the ramification.
After all, this is what we expect since we know from Abhyankar's work
(cf.\ [A10]) that $(F|K,P)$ always admits local uniformization for
$\trdeg F|K$ up to 3 (with the possible exception of characteristic 2,
3, 5).
\end{example}

Getting rid of defect and ramification by Artin--Schreier inversion
seems to be the algebraic kernel of local uniformization and, in
particular, of Abhyankar's proofs. However, the following questions
should be answered without a restriction of the dimension (i.e.,
the transcendence degree of $F|K$):
\mn
{\bf Open Problem 10:} \ Prove (or disprove) that by Artin--Schreier
inversion in connection with Artin--Schreier surgery one can always get
rid of the defect. How about ramification?
\mn
The only case where I know that the answer is positive is the case of
Theorem~\ref{stt3}, the Henselian Rationality of Immediate Function
Fields. There, it is the crucial part of the proof. There seems to be no
reason why the answer to the above problem should depend on the
transcendence degree of $F|K$ or on the particular value of the positive
characteristic. Actually, what I am saying is not quite true since the
proof of Theorem~\ref{stt3} so far only works under a strong assumption
about the base field, and a generalization may more easily be achieved
if the restriction of $P$ to that base field is an Abhyankar place. That
might indicate that there is more hope for function fields of
transcendence defect at most 1 than for those with higher transcendence
defect. In dimension 2 ($\trdeg F|K=2$), every place $P$ of $F|K$ will
have transcendence defect at most 1 since there is always a subfunction
field of transcendence degree at least 1 on which the restriction of $P$
is an Abhyankar place. This seems to separate the case of dimension
$\leq 2$ from the case of dimension $\geq 3$. But as Abhyankar was able
to tackle dimension 3 (where the transcendence defect may well be
2), there seems to be no reason why all this shouldn't work for even
higher dimensions.

\pars
By looking at these crucial Artin--Schreier extensions, we have
considered the kernel of the problem. But are we really sure that
we can always reduce to this kernel (for instance, by passing to
ramification fields as described in Section~\ref{sectGST})?
\mn
{\bf Open Problem 11:} \ Is it always (in all dimensions) possible to
pull down local uniformization through tame extensions? That is, if
$({\cal F}|K,P)$ is uniformizable where $({\cal F}^h|F^h, P)$ is a tame
extension, will then also $(F|K,P)$ be uniformizable? Which additional
assumptions do we possibly need? What answers can be extracted from
Abhyankar's work? Is there some generalization of his ``Going Up'' and
``Coming Down'' techniques to all dimensions?
\mn
Again, there is no hint why the dimension should have an influence on
this problem. Possibly it can be found in Abhyankar's work.

If we look at our examples, we see that bad places and defect extensions
in the generation of a function field already appear in dimension 2, so
from this point of view, the dividing line seems to lie between
dimension 1 and dimension 2. Our consideration concerning the
transcendence defect seems to suggest a dividing line between dimension
2 and dimension 3. Also Example~\ref{exampnhr} goes in this direction,
although it is not clear what it actually means for local uniformization
and whether there possibly is an analogue of transcendence degree 2.

There is, however, another point which we have not yet mentioned. If our
place $P$ has rank $>1$, can we then always proceed by induction on the
rank? If we are ready to extend our function field $F$, then the answer
is: yes ([K5], [K6], [K7]). But what if we want to prove local
uniformization without extension of $F$ and we have $P=Q\ovl{Q}$
such that $FQ|K$ is not finitely generated, hence not a function field?
In this situation, $Q$ consumes already transcendence degree 2, and if
we assume that also $\ovl{Q}$ is non-trivial, then we have $\trdeg F|K
\geq 3$. Analogously, if we find that critical things happen in
dimension 3 but not in dimension 2, these things might only develop
their destructive influence in connection with composition of places,
which would lift the critical dimension up to 4. (But to be true, I do
not believe that this could happen. I believe, if we have local
uniformization in dimension 3, then ultimately there will be a proof
which works for dimension 3 in the same way as for all higher
dimensions.)

%
%
\section{The space of all places of $F|K$}       \label{sectSFK}
The set $S(F|K)$ of all places of $F|K$ (where equivalent places are
identified) is called the \bfind{Zariski--Riemann manifold} or just the
\bfind{Zariski space} of $F|K$. See [V] for the definition of the
Zariski topology on $S(F|K)$, its compactness and other properties of
this space. Here, we will consider yet another property.
We have seen in Section~\ref{sectbad} that the Zariski space even of
very simple function fields can contain bad places. On the other hand,
we have seen that there are good places (e.g., Abhyankar places) for
which local uniformization is easier than in the general case. But
do good places or Abhyankar places exist in every Zariski space?

An ad hoc method to prove that this is true is to construct places of
maximal rank. Take a transcendence basis $t_1,\ldots,t_k$ of $F|K$ and
set $K_{k+1}:=K$ and $K_i:=K(t_i,\ldots,t_k)$ for $1\leq i<k$. Take
$P_1$ to be any place of $F|K_2$ such that $t_1P_1=0$. By
Corollary~\ref{fingentb}, $FP_1|K_2$ is finite; hence, $FP_1|K_3$ is a
function field. So we can choose a place $P_2$ of $FP_1|K_3$ such that
$t_2P_2=0$. Again, $FP_1P_2|K_3$ is finite. By induction, we construct
places (in fact, prime divisors) $P_1,\ldots,P_k\,$. Their composition
$P=P_1\ldots P_k$ is a place of $F|K$ of rank $k=\trdeg F|K$. Hence,
$P$ is a place of maximal rank and thus an Abhyankar place of $F|K$.

We wish to give a much more sophisticated method which shows that good
places not only exist, but even are ``very representative'' in every
Zariski space. The general result reads as follows (cf.\ [K3];
the special case of $\chara K=0$ was proved in [KP]):
\begin{theorem}                                   \label{MKP}
Let $F|K$ be a function field in $k$ variables. Let $Q$ be a place of
$F|K$ and $a_1,\ldots,a_m,b_1,\ldots,b_n\in F$. Then there exists a
place $P$ of $F|K$ such that:
\sn
1) \ $v_P^{ }F$ is a finitely generated group and extends the subgroup
of $v_Q^{ }F$ generated by $v_Q^{ }b_1,\ldots,v_Q^{ }b_n\,$,
\sn
2) \ $FP$ is finitely generated over $K$ and extends
$K(a_1Q,\ldots,a_mQ)$,
\sn
3) \ the following holds:
\[\begin{array}{rclcl}
    a_iP & = & a_iQ &\mbox{ for } & 1\leq i \leq m\;,\\
    v_P^{ }b_j & = & v_Q^{ }b_j &\mbox{ for } & 1\leq j\leq n\;.
  \end{array}\]
Moreover, $P$ can be chosen such that $v_P^{ }F$ is a subgroup of the
$p$-divisible hull $\frac{1}{p^\infty}v_Q^{ }F$ of $v_Q^{ }F$ if $\chara
K=p>0$, and $v_P^{ }F\subseteq v_Q^{ }F$ otherwise, and that $FP$ is a
subfield of the perfect hull of $FQ$. Alternatively, if
$r,d\in\N\cup\{0\}$ satisfy
\begin{equation}
\dim Q\leq d\leq k-\rr Q\;,\qquad \rr Q\leq r\leq k-d\;,
\end{equation}
then $P$ may be chosen such that $\dim P = d$ and $\rr P = r$.
(If $r=k-d$, then $P$ is an Abhyankar place of $F|K$.)
\end{theorem}
The theorem tells us that for every place $Q$ of $F|K$ and
every choice of finitely many elements in $F$ there is an Abhyankar
place $P$ which agrees with $Q$ on these elements.

For the proof of the above theorem, see [K3]. The main
ideas are the following. First, we choose an Abhyankar subfunction field
$F_0$ of $(F|K,Q)$ as in (\ref{Asff}), where we replace $P$ by $Q$.
Using the model theory of tame fields, we ``pull the situation down''
into a finite extension $F_1$ of $F_0\,$. This is done as follows.

By our choice of $F_0\,$,
$v_Q^{ }F/v_Q^{ }F_0$ is a torsion group and $FQ|F_0Q$ is algebraic.
Now we take $(L,Q)$ to be a maximal purely wild extension of the
henselization of $(F,Q)$; then by Lemma~\ref{Wtf}, $(L,Q)$ is a tame
field. Further, we take $L'$ to be the relative algebraic closure of
$F_0$ in $L$. Then by Lemma~\ref{trac}, $(L',Q)$ is a tame field and
$(L|L',Q)$ is an immediate extension. Hence, there are elements
$a''_1,\ldots,a''_m$ and $b''_1,\ldots,b''_n$ in $L'$ whose values
or residues coincide with those of $a_1,\ldots,a_m$ and $b_1,\ldots,
b_n\,$. Now we choose generators $t_1,\ldots,t_k,z$ for the
function field $F.L'|L'$, like for $F|K$ in
Example~\ref{examperp}. The elements $a_i,b_i$ are rational functions in
these generators. Similarly as in Example~\ref{examperp}, we want to
pull down these generators to $L'$, preserving the values and residues
of $a_i,b_i\,$. So the existential ${\cal L}_{\rm VF} (L')$-sentence we
employ will now contain also the information that $v_Q^{ }a_i=v_Q^{
}a''_i$ and $b_iQ=b''_iQ$ (note that $a''_i,b''_i$ are constants from
$L'$). Since $(L',v_Q^{ })\ec (L,v_Q^{ })$ by Theorem~\ref{AKEtame}
(because $v_Q^{ }K\ec v_Q^{ }L \mbox{ and } Kv_Q^{ } \ec Lv_Q^{ }$
trivially hold), we know that the existential sentence also holds in
$(L',v_Q^{ })$. This gives us the elements $c_1,\ldots,c_k,d\in L'$ and
thus also the new elements $a'_i,b'_i$ as rational functions in these
new elements, satisfying that $v_Q^{ }a'_i=v_Q^{ }a''_i=v_Q^{ }a_i$ and
$b'_iQ=b''_iQ=b_iQ$.

The elements $c_1,\ldots,c_k,d$ generate a finite extension $(F_1,Q)$ of
$(F_0,Q)$ inside of $(L',Q)$. This extension will be responsible for the
extension of the value group and the residue field. But as it lies in
$L$, we can employ Theorem~\ref{rfmte} to show that $v_Q^{ }F_1/
v_Q^{ }F_0$ is a $p$-group and $F_1Q|F_0Q$ is purely inseparable,
which yields the corresponding assertion in Theorem~\ref{MKP}.

Adjoining enough
transcendental elements and extending $Q$ in a suitable way, we build up
a function field $(F_2,P)$ having dimension $d$ and rational rank $r$
and such that $P$ is an Abhyankar place of $F_2|K$. Finally, using the
Implicit Function Theorem, we embed $F$ over $K$
in the completion of $(F_2,P)$ and pull back $P$ through this embedding.
We construct the embedding in such a way that the image of every $a_i$
and every $b_j$ is very close to $a'_i$ and $b'_j\,$, respectively. This
implies that $a_iP= a'_iP=a'_iQ=a_iQ$ and $v_P^{ }b_j=v_P^{ }b'_j=
v_Q^{ }b'_j=v_Q^{ }b_j\>$ (recall that by construction, $P$ and $Q$
coincide on the field $F_1$ which contains the elements
$a'_1,\ldots,a'_m,b'_1,\ldots,b'_n$).

%

\parm
In [K3], I prove several modifications of Theorem~\ref{MKP}, which
have various applications. Let me give an example.
\begin{example}
A modification of Theorem~\ref{MKP} (cf.\ [K3]) shows that one can
replace $Q$ by a discrete place $P$ such that $FP$ is a subfield of the
perfect hull of $FQ$ (again, one can preserve finitely many residues,
but not values anymore). Hence if $K$ is perfect and $Q$ is a rational
place, then also $P$ will be rational. As $(F,P)$ is discrete, $F$
embeds over $K$ in $K((t))$. If we assume that $K$ is large, then by
Theorem~\ref{largef}, $K\ec K((t))$. Since every existential elementary
sentence holding in $F$ will also hold in the bigger field $K((t))$, it
follows that $K\ec F$. We have proved:
\begin{theorem}                             \label{largefec}
Assume that $K$ is a large field. Assume further that $K$ is perfect
and that $F|K$ admits a rational place $Q$. Then $K$ is existentially
closed in $F$ (in the language of fields).
\end{theorem}
\end{example}

Reviewing the results on local uniformization that we have stated so
far, we see that the best results can be obtained for zero-dimensional
discrete or zero-dimensional Abhyankar places (and if $K$ is assumed to
be algebraically closed, then ``zero-dimensional'' is the same as
``rational''). But the above theorem only renders places $P$ with $\dim
P \geq\dim Q$, so starting from a place which is not zero-dimensional,
we will again get a place which is not zero-dimensional. This can be
overcome by a modification like the one in the last example.
%
%
For the formulation of the results we shall use the \bfind{Zariski patch
topology} for which the basic open sets are the sets of the form
\begin{equation}
\{P\in S(F|K)\mid a_1 P\ne 0\,,\ldots,\,a_k P\ne 0\,;\,
b_1 P=0\,,\ldots,\,b_\ell P=0\}
\end{equation}
with $k,\ell\in\N\cup\{0\}$ and $a_1,\ldots,a_k,b_1,\ldots, b_\ell\in
F\setminus \{0\}$. It is finer than the Zariski topology. But some
proofs showing that the Zariski topology is compact actually show first
that the Zariski patch topology is compact (cf.\ [SP]). Also, the
compactness of the Zariski patch topology and the Zariski topology can
easily be derived from the Compactness Theorem of model theory
(Theorem~\ref{CT}); see [K2], [K3].
\begin{theorem}                             \label{densesp}
The following places lie dense in $S(F|K)$ with respect to the
Zariski patch topology:
\sn
a) \ the zero-dimensional rank 1 Abhyankar places,
\n
b) \ the zero-dimensional places of maximal rank,
\n
c) \ the zero-dimensional discrete places,
\n
d) \ the prime divisors.
\end{theorem}
This can be proved by a combination of Theorem~\ref{MKP} and
Lemma~\ref{exrapl} (cf.\ [K3]). For the proof, one does not need the
model theory of tame fields; an application of Theorem~\ref{mcvfAR} will
suffice. From Theorem~\ref{densesp} together with Theorem~\ref{MTdrlu}
or Theorem~\ref{Abh1}, we obtain:
\begin{corollary}                           \label{corlu}
If $K$ is algebraically closed, then the uniformizable places $P$ lie
dense in $S(F|K)$ with respect to the Zariski patch topology.
\end{corollary}
This result immediately generates the following questions:
\mn
{\bf Open Problem 12:} \ If we have proved local uniformization for a
set of places which lies dense in $S(F|K)$ with respect to the Zariski
patch topology, can we patch the local solutions together to obtain the
global resolution of singularities? If this doesn't work, how about
finer topologies? What other properties of $S(F|K)$ can be deduced from
dense subsets?
\mn
Certainly, it doesn't follow directly from Corollary~\ref{corlu} that
all places in $S(F|K)$ are uniformizable. But it would follow
if the next open problem had a positive answer. Observe that local
uniformization is an open property, that is, if $P$ is uniformizable,
then there is an open neighborhood of $P$ in which every place admits
(the same) local uniformization.
\mn
{\bf Open Problem 13:} \ Can we define something like a ``radius'' of
these ``local uniformization neighborhoods'' and show that there is a
lower bound for this radius?
\mn
After all, being a finitely generated field extension, a function field
only contains ``finite algebraic information''. On the other
hand, it should be clear from the examples of bad places that there are
infinitely many ways of being bad... So it would be nice if we could
forget about bad places. However, the badness expresses itself already
on a transcendence basis, and the lower bound for the radius might only
depend on the algebraic extension above that transcendence basis.

\parm
Now let's have a look at the main open problem of the model
theory of valued fields. It has been around since the work of Ax and
Kochen, and several excellent model theorists have tried their luck on
it, in vain.

%
%
\section{$\Fp((t))$}                        \label{sectfpt}
May I introduce to you my dearest friend and scariest enemy: $\Fp((t))$.
Recall that it appeared on the right hand side of (\ref{equivultra}). On
the left hand side, there were the fields $\Qp$ of $p$-adic numbers. In
a second paper [AK2], Ax and Kochen gave a nice (``recursive'') complete
axiom system for $\Qp$ with its $p$-adic valuation. This generated the
problem to give a nice complete axiom system also for $\Fp((t))$ with
its $t$-adic valuation. One can always give an axiom system by writing
down {\it all\/} sentences which hold in a structure. But ``writing
down'' is very optimistic: there are infinitely many such sentences, and
we may not even have a procedure to generate them in some algorithmic
way. In contrast to this, a finite axiom system causes no problem.
Also schemes like we use for ``algebraically closed'' or ``henselian''
aren't problematic since increasingly large finite subsets of them
can be produced by an algorithm.
This is what ``recursive'' means. Now if we have a complete recursive
axiom system then there is also an algorithm to decide whether a given
elementary sentence holds in every model of that axiom system. This is
what one means when asking the famous question:
\mn
{\bf Open Problem 14:} \ Is the elementary theory of $(\Fp((t)),v_t)$
decidable? In other words, does $(\Fp((t)),v_t)$ admit a
complete recursive axiomatization?
\mn

The complete recursive axiomatization for $(\Qp,v_p)$ is not hard to
state. It is essentially the following:\n
1) \ $(K,v)$ is a valued field,\n
2) \ $(K,v)$ is henselian,\n
3) \ $\chara K=0$,\n
4) \ $Kv=\Fp$,\n
5) \ $vK$ is an ordered abelian group which is elementarily equivalent
to $\Z$.
\n
We can write $Kv=\Fp$ since every field which is elementarily equivalent
to $\Fp$ is already equal to $\Fp$ (because $\Fp$ has finitely many
elements, and their number can thus be expressed by an elementary
sentence). In contrast to this, we cannot write $vK=\Z$; since $\Z$
has infinitely many elements, Theorem~\ref{exsat} implies that there are
many other ordered abelian groups which are elementarily equivalent to
$\Z$. These are called \bfind{$\Z$-groups}. An ordered abelian group is
a $\Z$-group if and only if $\Z$ is a convex subgroup of $G$ and $G/\Z$
is divisible.

Observe that $(\Fp((t),v_t)$ looks very much like $(\Qp,v_p)$. Indeed,
the only axiom that does not hold is axiom 3). So if we replace it by
$3'$): ``$\chara K=p$'', will we get a complete axiom system (which by
our above remarks would be recursive)? We have stated in
Section~\ref{sectdef} that every henselian discretely valued field of
characteristic 0 is a defectless field. From this it follows that in the
presence of the other axioms (including 3)!), axiom 2) implies that
$(K,v)$ is defectless. If we change 3) to ``$\chara K=p$'', this is not
any longer true, and we have to replace axiom 2) by $2'$): ``$(K,v)$ is
henselian and defectless''.

For a long time, many model theorists believed that the
axiom system 1), $2'$), $3'$), 4), 5) could be complete. But based on an
observation by Lou van den Dries, I was able to show in [K1] that this
is not the case (cf.\ [K2]). It is precisely Example~\ref{exampnhr}
which proves the incompleteness. The point is that if $K$ has positive
characteristic, then we have non-linear additive polynomials which we
can use to express additional elementary properties. In characteristic
0, the only additive polynomials are of the form $cx$, so there is
nothing interesting about them. But in positive characteristic, the
image $f(K)$ of an additive polynomial is a subgroup of the
additive group of $K$. If $f_1,\ldots,f_n$ are additive polynomials,
then one can consider the subgroup $f_1(K)+\ldots+f_n(K)$. For certain
choices of the $f$'s, these subgroups have nice elementary properties
if $K$ is elementarily equivalent to $\Fp((t))$. This implies that for
certain choices, one can even show that $K=f_1(K)+\ldots+ f_n(K)$. To
some extent, this has the same flavour as Hensel's Lemma, but the
incompleteness result shows that all this doesn't follow from Hensel's
Lemma, or to be more precise, doesn't even follow from the axioms 1),
$2'$), $3'$), 4), 5). See [K4] for details.

The subgroups of the form $f_1(K)+\ldots+f_n(K)$ are definable by an
elementary sentence using constants from $K$ (as coefficients
of the polynomials $f_i$). This fact leads to the following questions:
\mn
{\bf Open Problem 15:} \ Does the axiom system 1), $2'$), $3'$), 4), 5)
become complete if we add the elementary properties of the subgroups
of the form $f_1(K)+\ldots+f_n(K)$? What other subgroups of $\Fp((t))$
are elementarily definable, and what are their elementary properties?
\mn

As we have seen already that additive polynomials play a crucial role
for local uniformization in positive characteristic, it makes sense
to ask:
\mn
{\bf Open Problem 16:} \ What is the relation between the elementary
properties of $\Fp((t))$ expressible by use of additive polynomials
and algebraic geometry in positive characteristic?
\mn
On the valuation theoretical side, I can say that work in progress
indicates that these elementary properties have a crucial meaning for
the structure theory of valued function fields. For example, it seems
that the Henselian Rationality of Immediate Function Fields can be
generalized to the case of base fields $(K,v)$ which are not tame but
have these properties. By the way, these properties don't play a role
in the model theory of tame fields because all tame fields are perfect
(like all other fields of positive characteristic for which we know
that Ax--Kochen--Ershov principles hold). In contrast to this,
$\Fp((t))$ is not perfect since $t$ has no $p$-th root in $\Fp((t))$.

In comparison to the model theory of tame fields, that of $\Fp((t))$
is much more complex since some tools available for tame fields
will not work anymore. As an example, we do not have an analogue of
the crucial Lemma~\ref{trac} which we used to separate extensions of
tame fields into extensions without transcendence defect and immediate
extensions. Thus, we cannot do this (in general) in the case of
fields which are elementarily equivalent to $\Fp((t))$. We are also not
able to ``slice'' immediate extension into extensions of transcendence
degree 1. But then we would have to develop an analogue of Kaplansky's
theory of immediate extensions for the case of higher transcendence
degree, or even worse, simultaneously for all mixed extensions
without a possibility of separation. However, this is more or less
the generalization of the theory of approximate roots that is recently
discussed.

%
%
\section{Local uniformization vs.\ Ax--Kochen--Ershov}
My first encounter with Zariski's Local Uniformization Theorem
was when I studied the model theoretic proof of the $p$-adic
Nullstellensatz by Moshe Jarden and Peter Roquette [JR]. At one
point, they consider the following situation. They have an extension
$L|K$ of $p$-adically closed fields, and a function field $F|K$
inside of $L|K$. Inside of $F$, they have a certain subring $B$
containing $K$ and an element $g\in B$ which is not a unit in $B$. Now
they wish to show that there is a rational place $P$ of $F|K$ such that
$gP=0$. They take a maximal ideal ${\cal M}$ in $B$ such that $g\in
{\cal M}$. By the existence theorem for places (see [V],
Proposition~1.2), there is a
place $Q$ of $F|K$ such that $B\subseteq {\cal O}_Q$ and ${\cal M}_Q\cap
B={\cal M}$. It follows that $gQ=0$. Since $K\subset B$ we know that
$K\subseteq FQ$. But $Q$ may not be the required place since it may not
be rational. If $FQ|K$ is a function field, then one can proceed as
follows. By the special choice of the ring $B$ one knows that $FQ$ is
contained in some $p$-adically closed extension field $L'$. By the
work of Ax--Kochen and Ershov, one knows that $K\prec L'$. It
follows that $K\ec FQ$, hence by Lemma~\ref{exrapl} there is a rational
place $\ovl{Q}$ of $FQ|K$. So the place $Q\ovl{Q}$ is a rational place
of $F|K$ which satisfies $gQ\ovl{Q}=0\ovl{Q}=0$, as desired.

But it may well happen that $FQ|K$ is not finitely generated.
Then we can't apply Lemma~\ref{exrapl}. In this situation, Jarden and
Roquette use Zariski's Local Uniformization (Theorem~\ref{ZLUT})
to show that there exists a place $P$ of $F|K$ such that $gP=0$ and that
$FP|K$ is finitely generated. More generally, they show:
\begin{lemma}                               \label{JR}
Take a function field $F|K$ of characteristic 0, a place $Q$ of $F|K$
and elements $y_1,\ldots,y_n\in {\cal O}_Q\,$. Then there exists a place
$P$ of $F|K$ such that $y_iP=y_iQ$, $1\leq i\leq n$, $FP\subseteq FQ$
and $FP|K$ is finitely generated.
\end{lemma}
The proof works as follows. After adding elements if necessary, we can
assume that $F=K(y_1,\ldots,y_n)$. By Zariski's Local Uniformization
Theorem, after adding further elements we can assume that
$a=(y_1Q,\ldots,y_nQ)$ is a simple point of the $K$-variety whose
generic point is $(y_1,\ldots,y_n)$. Hence, the local ring ${\cal O}_a$
is regular. Now Jarden and Roquette employ the following lemma from
[A2]:
\begin{lemma}
Suppose that $R$ is a regular local ring with maximal ideal $M$ and
quotient field $F$. Then there exists a place $P$ dominating $R$ such
that $FP=R/M$.
\end{lemma}
\noindent
(In fact, $P$ is the place associated with the order valuation deduced
from $(R,M)$.)
\begin{corollary}
Suppose that $a$ is a simple point of $V$. Then there exists a place $P$
of $F|K$ such that $a=(y_1P,\ldots,y_nP)$ and $FP=K(a)$.
\end{corollary}
\begin{proof}
The residue field ${\cal O}_a/{\cal M}_a$ of ${\cal O}_a$ is isomorphic
to $K(a)$. We identify both fields, so that the residue map ${\cal O}_a
\rightarrow {\cal O}_a/{\cal M}_a=K(a)$ maps every $y_i$ to $a_i\,$. By
applying the foregoing lemma to $R={\cal O}_a\,$, we obtain a place $P$
of $F|K$ dominating ${\cal O}_a$ and such that $FP={\cal O}_a/{\cal M}_a
=K(a)$. Since $P$ dominates ${\cal O}_a\,$, it extends the residue map,
whence $a=(y_1P,\ldots,y_nP)$.
\end{proof}

This proves Lemma~\ref{JR}: by the definition of $a$ we get
that $(y_1P,\ldots,y_nP)=(y_1Q,\ldots,y_nQ)$ and $FP=K(y_1P,\ldots,y_nP)
=K(y_1Q,\ldots,y_nQ)\subset FQ$, showing also that $FP$ is a finitely
generated extension of $K$.

Note that if $a\in K^n$ then the place $P$ obtained from the
foregoing corollary is rational. Hence we have
(you may compare this with Lemma~\ref{exrapl}):
\begin{corollary}
Suppose that $a$ is a simple $K$-rational point of $V$. Then there
exists a rational place $P$ of $F|K$ such that $a=(y_1P,\ldots,y_nP)$.
\end{corollary}

\parm
In my Masters Thesis, I showed how to avoid the use of Zariski's Local
Uniformization Theorem by constructing a place of maximal rank (which
by Corollary~\ref{fingentb} always has a finitely generated residue
field). This trick was then used by Alexander Prestel and Peter Roquette
in their book [PR] for the proof of the $p$-adic Nullstellensatz. It
also provided the first idea for the paper [PK] in which we proved a
version of Theorem~\ref{MKP} for function fields of characteristic 0 by
using the Ax--Kochen--Ershov Theorem. This version has interesting
applications to real algebra and real algebraic geometry ([PK], [P]).
Surprisingly, a paper by Ludwig Br\"ocker and Heinz-Werner Sch\"ulting
[BS] derives about the same results and applications, using resolution
of singularities in characteristic 0 (Hironaka) in the place of the
Ax--Kochen--Ershov Theorem. See also the survey paper [SCH]. The first
question I have in this connection is:
\mn
{\bf Open Problem 17:} \ To which extent is it (easily) possible to
replace the use of resolution of singularities by local uniformization
in real algebraic geometry?
\pars
Seeing that the Ax--Kochen--Ershov Theorem and Zariski's Local
Uniformization Theorem can be used to deduce the same results,
Roquette asked:
\mn
{\bf Open Problem 18:} \ What is the relation between Ax--Kochen--Ershov
Theorem and Zariski's Local Uniformization Theorem? Can one prove one
from the other?
\mn
If that were true, then there would be some hope that a progress in
positive characteristic made on one side could be transferred to the
other side. For instance, as already mentioned, local uniformization or
resolution of singularities in positive characteristic could possibly
help to solve problems in the model theory of valued fields of positive
characteristic.

In view of the details I have told you about, my own preliminary answer
is that the relation between local uniformization and the model theory
of valued fields lies in the facts and theorems from the structure
theory of valued function fields which play a crucial role in both
problems. Therefore, new insights in this structure theory will also
be of importance for both problems.

On the other hand, there are ingredients on either side which do not
appear on the other. For example, Lemma~\ref{perron} does not seem to
play any role on the model theoretic side. Further, the henselization
causes a lot of serious problems for the local uniformization of places
of rank $>1$, whereas it is the best friend of model theorists. The need
to optimize the choice of the transcendence basis (cf.\
Section~\ref{sectrtb}) to avoid these problems and to avoid the defect
has (so far) no analogue on the model theoretic side; however, this may
change with a deeper insight in the theory of $\Fp((t))$. Conversely,
model theory is forced to deal with extensions $(L|K,v)$ of valued
fields where $v$ is non-trivial on $K$. To some extent, we did the same
when we considered relative uniformization. But in contrast to that
situation, the valuations on $K$ in the model theoretic case may be of
arbitrary rank and arbitrarily nasty.

\pars
Our discussion would not be complete if we would not mention the
following nice result, due to Jan Denef [DEN]. It says that
``the existential theory of $\Fp((t))$ is decidable'':
\begin{theorem}
If resolution of singularities holds in positive characteristic,
then there is an algorithm to decide whether a given existential
elementary sentence in the language of valued fields holds in
$(\Fp((t)),v_t)$.
\end{theorem}

%
%
\section{Back to local uniformization in positive
characteristic}                            \label{sectblu}
In the last section, we have seen that for certain applications it
matters to know what the residue fields of our places are. In
particular, when changing a bad place to a good place, we might wish to
keep the residue field within a certain class of fields. For example,
this could be the class of all fields in which the base field $K$ is
existentially closed. (Note that if $K$ is perfect and existentially
closed in $L$, then it will also be existentially closed in any purely
inseparable algebraic extension of $L$.) If we take into the bargain
an extension of the function field in order to obtain local
uniformization, but require that this extension should be normal or even
Galois, then we may not be able anymore to control the corresponding
extension of the residue field. So it makes sense to ask for the
minimal possible change of the residue field.

The key to this question is the fact that there are minimal algebraic
extensions of $(F,P)$ which are tame (or separably tame) fields; we just
have to pass to the henselization of $(F,P)$ and then choose a field $W$
according to Theorem~\ref{rfmte}, cf.\ Lemma~\ref{Wtf}. (For the case of
``tame'', see also Corollary~\ref{cortame}.) Working inside of such an
extension, we can nicely apply the theory of tame and separably tame
fields. In particular, we can use Lemma~\ref{trac} and the transitivity
of relative uniformization to reduce to extensions of the types
discussed in {\bf I)}--{\bf IV)} in Section~\ref{sectrlu}.

On the other hand, if $(L,P)$ is such an extension of $(F,P)$, then by
Theorem~\ref{rfmte}, $v_P^{ }L/v_P^{ }F$ will be a $p$-torsion group if
$\chara K=p>0$, and $LP|FP$ will be purely inseparable. If $\chara K=0$,
then $(L,P)$ is just the henselization of $(F,P)$. As our extension
${\cal F}$ remains inside of $L$, we can show (cf.\ [K7]):
\begin{theorem}                             \label{lucontrol}
In addition to the assertion of Theorem~\ref{MT1}, the finite extension
${\cal F}|F$ can be chosen to be separable and to satisfy:
\sn
a) \ if $\chara K=p>0$, then the finite group $v_P^{ }{\cal F}/v_P^{ }F$
is a $p$-torsion group, and the finite extension ${\cal F}P|FP$ is
purely inseparable,\n
b) \ if $\chara K=0$, then ${\cal F}$ can be chosen to lie in the
henselization of $F$.
\end{theorem}
\n
(Clearly, in the case of $\chara K=0$, our result is weaker than
Zariski's Local Uniformization Theorem. However, it provides some more
information since we can uniformize a finite extension of $F$ within
its henselization while keeping fixed the once chosen transcendence
basis of the subfunction field $F_0|K$ on which $P$ is an
Abhyankar place.)

\pars
The following corollary illustrates the advantage of controlling the
residue field extension. In fact, Theorem~\ref{largefec} can also be
proved by use of this corollary.
\begin{corollary}
Assume that $K$ is perfect and that $P$ is a rational place of $F|K$.
Take $\zeta_1,\ldots,\zeta_m \in {\cal O}_F\,$. Then there is a finite
extension ${\cal F}|F$ and an extension of $P$ to ${\cal F}$ such that
$({\cal F}|K,P)$ is uniformizable with respect to $\zeta_1,\ldots,
\zeta_m$ and $P$ is still a rational place of ${\cal F}|K$.
\end{corollary}

\bn
\bn
\bn
{\bf References}
\newenvironment{reference}%
{\begin{list}{}{\setlength{\labelwidth}{5em}\setlength{\labelsep}{0em}%
\setlength{\leftmargin}{5em}\setlength{\itemsep}{-1pt}%
\setlength{\baselineskip}{3pt}}}%
{\end{list}}
\newcommand{\lit}[1]{\item[{#1}\hfill]}
\begin{reference}
\lit{[A1]} {Abhyankar, S.$\,$: {\it Local uniformization on algebraic
surfaces over ground fields of characteristic $p\ne 0$}, Annals of
Math.\ {\bf 63} (1956), 491--526. Corrections: Annals of Math.\ {\bf 78}
(1963), 202--203}
\lit{[A2]} {Abhyankar, S.$\,$: {\it On the valuations centered in a
local domain}, Amer.\ J.\ Math.\ {\bf 78} (1956), 321--348}
\lit{[A3]} {Abhyankar, S.$\,$: {\it Simultaneous resolution for
algebraic surfaces}, Amer.\ J.\ Math.\ {\bf 78} (1956), 761--790}
\lit{[A4]} {Abhyankar, S.$\,$: {\it On the field of definition of a
nonsingular birational transform of an algebraic surface}, Annals of
Math.\ {\bf 65} (1957), 268--281}
\lit{[A5]} {Abhyankar, S.$\,$: {\it Ramification theoretic methods in
algebraic geometry}, Princeton University Press (1959)}
\lit{[A6]} {Abhyankar, S.$\,$: {\it Uniformization in $p$-cyclic
extensions of algebraic surfaces over ground fields of characteristic
$p$}, Math.\ Annalen {\bf 153} (1964), 81--96}
\lit{[A7]} {Abhyankar, S.$\,$: {\it Reduction to multiplicity less than
$p$ in a $p$-cyclic extensions of a two dimensional regular local ring
($p=\,$characteristic of the residue field)}, Math.\ Annalen {\bf 154}
(1964), 28--55}
\lit{[A8]} {Abhyankar, S.$\,$: {\it Uniformization of Jungian local
domains}, Math.\ Annalen {\bf 159} (1965), 1--43. Correction: Math.\
Annalen {\bf 160} (1965), 319--320}
\lit{[A9]} {Abhyankar, S.$\,$: {\it Uniformization in $p$-cyclic
extensions of a two dimensional regular local domain of residue field
characteristic $p$}, Festschrift zur Ged\"achtnisfeier f\"ur Karl
Weierstra{\ss} 1815--1965, Wissenschaftliche Abhandlungen des Landes
Nordrhein-Westfalen {\bf 33} (1966), 243--317 (Westdeutscher Verlag,
K\"oln und Opladen)}
\lit{[A10]} {Abhyankar, S.$\,$: {\it Resolution of singularities of
embedded algebraic surfaces}, Academic Press, New York (1966);
2nd enlarged edition: Sprin\-ger, New York (1998)}
\lit{[A11]} {Abhyankar, S.$\,$: {\it Nonsplitting of valuations
in extensions of two dimensional regular local domains},
Math.\ Annalen {\bf 153} (1967), 87--144}
\lit{[AO]} {Abramovich, D.\ -- Oort, F.$\,$: {\it Alterations and
resolution of singularities}, this volume}
\lit{[AK1]} {Ax, J.\ -- Kochen, S.$\,$: {\it Diophantine problems
over local fields I}, Amer.\ Journ.\ Math.\ {\bf 87} (1965), 605--630.}
\lit{[AK2]} {Ax, J.\ -- Kochen, S.$\,$: {\it Diophantine problems
over local fields II}, Amer.\ Journ.\ Math.\ {\bf 87} (1965), 631--648.}
\lit{[AX]} {Ax, J.$\,$: {\it A metamathematical approach to some
problems in number theory}, Proc.\ Symp.\ Pure Math.\ {\bf 20},
Amer.\ Math.\ Soc.\ (1971), 161--190}
\lit{[B]} {Bourbaki, N.$\,$: {\it Commutative algebra}, Paris (1972)}
\lit{[BGR]} {Bosch, S. -- G\"untzer, U. -- Remmert, R.$\,$:
{\it Non-Archi\-me\-dean Analysis}, Berlin (1984)}
\lit{[BR]} {Brown, S.\ S.$\,$: {\it Bounds on transfer principles for
algebraically closed and complete valued fields},
Memoirs Amer.\ Math.\ Soc.\ {\bf 15} No.\ 204 (1978)}
\lit{[BS]} {Br\"ocker, L.\ -- Sch\"ulting, H.~W.$\,$:
{\it Valuations of function fields from the geometrical point of view},
J.\ reine angew.\ Math.\ {\bf 365} (1986), 12--32}
\lit{[CH1]} {Cherlin, G.$\,$: {\it Model theoretic algebra},
Lecture Notes Math.\ {\bf 521} (1976)}
\lit{[CH2]} {Cherlin, G.$\,$: {\it Model theoretic algebra},
J.\ Symb.\ Logic {\bf 41} (1976), 537--545}
\lit{[CK]} {Chang, C.\ C.\ -- Keisler, H.\ J.$\,$: {\it Model
Theory}, Amsterdam -- London (1973)}
\lit{[dJ]} {de Jong, A.~J.$\,$: {\it Smoothness, semi-stability and
alterations}, Publications Mathematiques I.H.E.S.\ {\bf 83} (1996),
51--93}
\lit{[DEL]} {Delon, F.$\,$: {\it Quelques propri\'et\'es des corps
valu\'es en th\'eo\-ries des mod\`e\-les}, Th\`ese Paris~VII (1981)}
\lit{[DEN]} {Denef, J.$\,$: personal communication, Oberwolfach,
October 1998}
\lit{[DEU]} {Deuring, M.$\,$: {\it Reduktion algebraischer
Funktionenk\"orper nach Primdivisoren des Konstantenk\"orpers},
Math.\ Z.\ {\bf 47} (1942), 643--654}
\lit{[EK]} {Eklof, P.$\,$: {\it Lefschetz's Principle and local
functors}, Proc.\ Amer.\ Math.\ Soc.\ {\bf 37} (1973), 333--339}
\lit{[EL]} {Elliott, G.~A.$\,$: {\it On totally ordered groups, and
$K_0$}, in: Ring Theory Waterloo 1978, eds.\ D.~Handelman and
J.~Lawrence, Lecture Notes Math.\ {\bf 734}, 1--49}
\lit{[EN]} {Endler, O.$\,$: {\it Valuation theory}, Berlin (1972)}
\lit{[EPP]} {Epp, H.$\,$: {\it Eliminating wild ramification},
Inventiones Math.\ {\bf 19} (1973), 235--249}
\lit{[ER1]} {Ershov, Yu.\ L.$\,$: {\it On the elementary theory of
maximal normed fields}, Dokl.\ Akad.\ Nauk SSSR {\bf 165} (1965),
21--23 \ [English translation in: Sov.\ Math.\ Dokl.\ {\bf 6} (1965),
1390--1393]}
\lit{[ER2]} {Ershov, Yu.\ L.$\,$: {\it On elementary theories of local
fields}, Algebra i Logika {\bf 4}:2 (1965), 5--30}
\lit{[ER3]} {Ershov, Yu.\ L.$\,$: {\it On the elementary theory of
maximal valued fields I} (in Russian), Algebra i Logika {\bf 4}:3
(1965), 31--70}
\lit{[ER4]} {Ershov, Yu.\ L.$\,$: {\it On the elementary theory of
maximal valued fields II} (in Russian), Algebra i Logika {\bf 5}:1
(1966), 5--40}
\lit{[ER5]} {Ershov, Yu.\ L.$\,$: {\it On the elementary theory of
maximal valued fields III} (in Russian), Algebra i Logika {\bf 6}:3
(1967), 31--38}
\lit{[GR]} {Grauert, H.\ -- Remmert, R.$\,$: {\it \"Uber die
Methode der diskret bewerteten Ringe in der nicht archimedischen
Analysis}, Inventiones Math.\ {\bf 2}\linebreak (1966), 87--133}
\lit{[GRA]} {Gravett, K.~A.~H.$\,$: {\it Note on a result of Krull},
Cambridge Philos.\ Soc.\ Proc.\ {\bf 52} (1956), 379}
\lit{[GRE1]} {Green, B.$\,$: {\it Recent results in the theory of
constant reductions}, S\'emi\-naire de Th\'eorie des Nombres, Bordeaux
{\bf 3} (1991), 275--310}
\lit{[GRE2]} {Green, B.$\,$: {\it The Relative Genus Inequality for
Curves over Valuation Rings}, J.\ Alg.\ {\bf 181} (1996), 836--856}
\lit{[GMP1]} {Green, B.\ -- Matignon, M.\ --- Pop, F.$\,$:
{\it On valued function fields I}, manuscripta math. {\bf 65} (1989),
357--376}
\lit{[GMP2]} {Green, B.\ -- Matignon, M.\ --- Pop, F.$\,$:
{\it On valued function fields II}, J.\ reine angew.\ Math.\ {\bf
412} (1990), 128--149}
\lit{[GMP3]} {Green, B.\ -- Matignon, M.\ --- Pop, F.$\,$:
{\it On valued function fields III}, J.\ reine angew.\ Math.\ {\bf
432} (1992), 117--133}
\lit{[GMP4]} {Green, B.\ -- Matignon, M.\ --- Pop, F.$\,$:
{\it On the Local Skolem Property}, J.\ reine angew.\ Math.\ {\bf 458}
(1995), 183--199}
\lit{[GRB]} {Greenberg, M.~J.$\,$: {\it Rational points in henselian
discrete valuation rings}, Publ.\ Math.\ I.H.E.S.\ {\bf 31} (1967),
59--64}
\lit{[GRL]} {Greenleaf, M.$\,$: {\it Irreducible subvarieties and
rational points}, Amer.\ J.\ Math.\ {\bf 87} (1965), 25--31}
\lit{[GRU]} {Gruson, L.$\,$: {\it Fibr\'es vectoriels sur un  
polydisque ultra\-m\'e\-trique}, Ann.\ Sci.\ Ec.\ Super., IV.\ Ser.,
{\bf 177} (1968), 45--89}
\lit{[HO]} {Hodges, W.$\,$: {\it Model Theory}, Encyclopedia of
mathematics and its applications {\bf 42}, Cambridge University Press
(1993)}
\lit{[JA]} {Jacobson, N.$\,$: {\it Lectures in abstract algebra, III.\
Theory of fields and Galois theory}, Springer Graduate Texts in Math.,
New York (1964)}
\lit{[JR]} {Jarden, M.\ -- Roquette, P.$\,$: {\it The Nullstellensatz
over $\wp$--adically closed fields}, J.\ Math.\ Soc.\ Japan {\bf 32}
(1980), 425--460}
\lit{[KA1]} {Kaplansky, I.$\,$: {\it Maximal fields with valuations I},
Duke Math.\ Journ.\ {\bf 9} (1942), 303--321}
\lit{[KA2]} {Kaplansky, I.$\,$: {\it Maximal fields with valuations
II}, Duke Math.\ Journ.\ {\bf 12} (1945), 243--248}
\lit{[K1]} {Kuhlmann, F.-V.: {\it Henselian function fields and tame
fields}, extended version of Ph.D.\ thesis, Heidelberg (1990)}
\lit{[K2]} {Kuhlmann, F.-V.: {\it Valuation theory of fields, abelian
groups and modules}, preprint, to appear in the ``Algebra, Logic and
Applications'' series (Gordon and Breach), eds.\ A.~Mac\-intyre and
R.~G\"obel}
\lit{[K3]} {Kuhlmann, F.--V.$\,$: {\it Places of algebraic function
fields in arbitrary characteristic}, submitted}
\lit{[K4]} {Kuhlmann, F.--V.$\,$: {\it Elementary properties of power
series fields over finite fields}, submitted; prepublication in:
Structures Alg\'ebriques Ordonn\'ees, S\'eminaire Paris VII (1997), and
in: The Fields Institute Pre\-print Series, Toronto (1998)}
\lit{[K5]} {Kuhlmann, F.--V.$\,$: {\it On local uniformization in
arbitrary characteristic}, The Fields Institute Preprint Series,
Toronto (1997)}
\lit{[K6]} {Kuhlmann, F.--V.$\,$: {\it On local uniformization in
arbitrary characteristic I}, submitted}
\lit{[K7]} {Kuhlmann, F.--V.$\,$: {\it On local uniformization in
arbitrary characteristic II}, submitted}
\lit{[KK]} {Kuhlmann, F.-V.\ -- Kuhlmann, S.: {\it Valuation theory of
exponential Hardy fields}, submitted}
\lit{[KP]} {Kuhlmann, F.--V.\ -- Prestel, A.$\,$: {\it On places of
algebraic function fields}, J.\ reine angew.\ Math.\ {\bf 353}
(1984), 181--195}
\lit{[KPR]} {Kuhlmann, F.-V.\ -- Pank, M.\ -- Roquette, P.$\,$:
{\it Immediate and purely wild extensions of valued fields},
manuscripta math.\ {\bf 55} (1986), 39--67}
\lit{[KR]} {Krull, W.$\,$: {\it Allgemeine Bewertungstheorie},
J.\ reine angew.\ Math.\ {\bf 167} (1931), 160--196}
\lit{[L1]} {Lang, S.$\,$: {\it On quasi algebraic closure},
Ann.\ of Math.\ {\bf 55} (1952), 373--390}
\lit{[L2]} {Lang, S.$\,$: {\it Algebra}, New York (1965)}
\lit{[OH]} {Ohm, J.$\,$: {\it The henselian defect for valued
function fields}, Proc.\ Amer.\ Math.\ Soc.\ {\bf 107} (1989)}
\lit{[OO]} {Oort, F.$\,$: {\it Stable pointed curves, alterations and
applications}, Lecture notes of the Working Week on Resolution of
Singularities Tirol 1997}
\lit{[OR]} {Ore, O.$\,$: {\it On a special class of polynomials},
Trans.\ Amer.\ Math.\ Soc.\ {\bf 35} (1933), 559--584}
\lit{[OS]} {Ostrowski, A.$\,$: {\it Untersuchungen zur arithmetischen
Theorie der K\"or\-per}, Math.\ Z.\ {\bf 39} (1935), 269--404}
\lit{[POP]} {Pop, F.$\,$: {\it Embedding problems over large fields},
Ann.\ of Math.\ {\bf 144} (1996), 1--34}
\lit{[P]} {Prestel, A.$\,$: {\it Model theory of fields: an
application to positive semidefinite polynomials}, Soc.\ Math.\ de
France, m\'emoire no.\ {\bf 16} (1984), 53--64}
\lit{[PR]} {Prestel, A.\ -- Roquette, P.$\,$: {\it Formally $p$-adic
fields}, Lecture Notes Math.\ {\bf 1050},
Berlin--Heidelberg--New York--Tokyo (1984)}
\lit{[PZ]} {Prestel, A.\ -- Ziegler, M.$\,$: {\it Model theoretic
methods in the theory of topological fields},
J.\ reine angew.\ Math.\ {\bf 299/300} (1978), 318--341}
\lit{[R1]} {Ribenboim, P.$\,$: {\it Th\'eorie des valuations}, Les
Presses de l'Uni\-versit\'e de Montr\'eal (1964)}
\lit{[R2]} {Ribenboim, P.$\,$: {\it Equivalent forms of Hensel's
lemma}, Expo.\ Math.\ {\bf 3} (1985), 3--24}
\lit{[RO]} {Robinson, A.$\,$: {\it Complete Theories}, Amsterdam
(1956)}
\lit{[RQ1]} {Roquette, P.$\,$: {\it Zur Theorie der Konstantenreduktion
algebraischer Mannigfaltigkeiten}, J.\ reine angew.\ Math.\ {\bf 200}
(1958), 1--44}
\lit{[RQ2]} {Roquette, P.$\,$: {\it A criterion for rational places
over local fields}, J.\ reine angew.\ Math.\ {\bf 292}
(1977), 90--108}
\lit{[RQ3]} {Roquette, P.$\,$: {\it On the Riemann $p$-space of a
field: The $p$-adic analogue of Weierstrass' approximation theorem and
related problems}, Abh.\linebreak
Math.\ Sem.\ Univ.\ Hamburg {\bf 47} (1978), 236--259}
\lit{[RQ4]} {Roquette, P.$\,$: {\it $p$--adische und saturierte
K\"orper. Neue Variationen zu einem alten Thema von Hasse}, Mitteilungen
Math. Gesellschaft Hamburg {\bf 11 \rm Heft 1} (1982), 25--45}
\lit{[RQ5]} {Roquette, P.$\,$: {\it Some tendencies in contemporary
algebra}, in: Perspectives in Mathematics, Anniversary of Oberwolfach
1984, Basel (1984), 393--422}
\lit{[SA]} {Sacks, G.\ E.$\,$: {\it Saturated Model Theory}, Reading,
Massa\-chu\-setts (1972)}
\lit{[SCH]} {Scheiderer, C.$\,$: {\it Real algebra and its
applications to geometry in the last 10 years: some major developments
and results}, in: Real algebraic geometry, Proc.\ Conf.\ Rennes 1991,
eds.\ M.~Coste and M.~F.~Roy, Lecture Notes in Math.\ {\bf 1524} (1992),
1--36 and 75--96}
\lit{[SP]} {Spivakovsky, M.$\,$: {\it Resolution of singularities I:
local uniformization}, manu\-script (1996)}
\lit{[U]} {Urabe, T.$\,$: {\it Resolution of Singularities of Germs in
Characteristic Positive Associated with Valuation Rings of Iterated
Divisor Type}, preprint, 75 pages}
\lit{[V]} {Vaquie, M.$\,$: {\it Valuations}, this volume}
\lit{[WA1]} {Warner, S.$\,$: {\it Nonuniqueness of immediate maximal
extensions of a valuation}, Math.\ Scand.\ {\bf 56} (1985),
191--202}
\lit{[WA2]} {Warner, S.$\,$: {\it Topological fields}, Mathematics
studies {\bf 157}, North Holland, Amsterdam (1989)}
\lit{[WH1]} {Whaples, G.$\,$: {\it Additive polynomials}, Duke
Math.\ Journ.\ {\bf 21} (1954), 55--65}
\lit{[WH2]} {Whaples, G.$\,$: {\it Galois cohomology of additive
polynomials and n-th po\-wer mappings of fields}, Duke Math.\
Journ.\ {\bf 24} (1957), 143--150}
\lit{[Z]} {Zariski, O.$\,$: {\it Local uniformization on
algebraic varieties}, Ann.\ Math.\ {\bf 41} (1940), 852--896}
\lit{[ZS]} {Zariski, O.\ -- Samuel, P.$\,$: {\it Commutative
Algebra}, Vol.\ II, New York--Heidel\-berg--Berlin (1960)}
\end{reference}
\adresse
\end{document}